\documentclass[a4paper, 11pt]{amsart}
\usepackage{amssymb}
\usepackage{tikz-cd,adjustbox,dsfont,partmac}

\usepackage[all]{xy}
\usepackage[linesnumbered,ruled,vlined]{algorithm2e}
\usepackage{array,longtable}
\usepackage[colorlinks=true,linkcolor=magenta,citecolor=magenta]{hyperref}
\usepackage[hmargin=1.5cm, vmargin=2.5cm]{geometry}

\theoremstyle{plain}

\newtheorem{theorem}{Theorem}[section]
\newtheorem*{theorem*}{Theorem}
\newtheorem{lemma}[theorem]{Lemma}
\newtheorem{proposition}[theorem]{Proposition}
\newtheorem{corollary}[theorem]{Corollary}
\newtheorem{conj}[theorem]{Conjecture}

\newtheorem{thmx}{Theorem}

\theoremstyle{definition}
\newtheorem{definition}[theorem]{Definition}
\newtheorem{question}[theorem]{Question}

\newtheorem{remark}[theorem]{Remark}

\begin{document}

\title[Maximality levels of $S_N\leq S_N^+$]{Maximality levels of the classical permutation group in the quantum permutation group}

\author{J.P. McCarthy}
\address{%
Department of Mathematics\\
Munster Technological University\\
Cork\\
Ireland}

\email{jp.mccarthy@mtu.ie}
\date{October 27, 2026}
\begin{abstract}
Progress on the conjecture of Banica and Bichon that the classical permutation group is a maximal quantum subgroup of the quantum permutation group remains limited to a handful of small-parameter results. By Tannaka--Krein duality, any counterexample to this Maximality Conjecture must arise from a category strictly intermediate between the category $\mathcal{NC}$ of non-crossing partitions and the category $\mathcal{P}$ of all partitions. Any such exotic category must therefore contain a linear combination of crossing-partition vectors. The categories generated by $\mathcal{NC}$ together with some such vectors are studied, with a number of generation results. It is shown that no exotic category can contain a linear combination of three crossing-partition vectors, and, using an argument of ChatGPT to account for a five-dimensional gap, there is no exotic category distinguished from $\mathcal{NC}$ or $\mathcal{P}$ at moments of order six.
\end{abstract}

\maketitle
\tableofcontents

\section{Introduction}
About ten years after the 1998 discovery of the quantum permutation group by Wang \cite{wa2}, it was announced in a widely cited survey \cite{bbc} that Banica--Bichon \cite{bb3} had classified the quantum subgroups of the quantum permutation group $S_4^+$, and noted:
\begin{theorem}[Banica--Bichon]\label{fourmax}
The classical permutation group $S_4$ is a maximal quantum subgroup of the quantum permutation group $S_4^+$.
\end{theorem}
This led Banica--Bichon to formulate the following conjecture:
\begin{conj}[The Maximality Conjecture]\label{maxconj}
The quantum subgroup $S_N\leq S_N^+$ is maximal at all $N$.
\end{conj}
Around this time there was a flurry of activity surrounding the quantum permutation group.
Banica--Collins \cite{bac} determined the precise relationship between the representation
theory of $S_N^+$ and non-crossing partitions. Soon after, the seminal paper of
Banica--Speicher \cite{bas} introduced \emph{easy quantum groups}, organising the
relationship between \emph{homogeneous} (orthogonal) compact matrix quantum groups, those compact matrix quantum groups intermediate between $S_N\leq \mathbb{G}\leq O_N^+$, and categories of partitions\footnote{This work relies heavily on their framework: further references to ``Banica--Speicher'' will refer always to \cite{bas}.}. The appearance of partitions in this setting can be understood from symmetry considerations: the action of $S_N\times S_N$ permuting the rows and columns of the fundamental representation implies that relations depend only on patterns of index coincidences. These patterns correspond precisely to set partitions (see Section \ref{symm}).

\bigskip\noindent

 To an orthogonal compact matrix quantum group $\mathbb{G}$ with fundamental representation $u\in M_N(C(\mathbb{G}))$ one can associate a Tannaka--Krein category $\mathcal{C}$, which has for objects $\operatorname{Ob}(\mathcal{C})=\mathbb{N}_0$, the non-negative integers, and for morphisms $T\in\hom_{\mathcal{C}}(k,\ell)$, the linear maps satisfying $Tu^{\otimes k}=u^{\otimes \ell}T$, the Peter--Weyl intertwiners. Note that the entries of $u^{\otimes k}\in M_{N^k}(C(\mathbb{G}))$ are given by:
 $$u^{\otimes k}_{\mathbf{i},\mathbf{j}}=u_{i_1j_1}u_{i_2j_2}\cdots u_{i_kj_k};\qquad \mathbf{i},\mathbf{j}\in \{1,2,\dots,N\}^k.$$
  The entries of the matrix $Tu^{\otimes k}-u^{\otimes \ell}T$ give relations that hold in the algebra of continuous functions $C(\mathbb{G})$.
A matrix quantum subgroup $\mathbb{H}\leq \mathbb{G}$, given by a comultiplication-preserving quotient $C_{\mathrm{u}}(\mathbb{G})\twoheadrightarrow C_{\mathrm{u}}(\mathbb{H})$, gives a contravariant inclusion of Tannaka--Krein categories $\mathcal{C}_{\mathbb{G}}\leq \mathcal{C}_{\mathbb{H}}$: that is the `smaller' algebra of continuous functions $C_{\mathrm{u}}(\mathbb{H})$ satisfies more relations than the  `larger' algebra of continuous functions $C_{\mathrm{u}}(\mathbb{G})$.  A counterexample to the Maximality Conjecture \ref{maxconj} --- a strictly intermediate quantum subgroup $S_N\lneq \mathbb{G}\lneq S_N^+$ --- will be called an \emph{exotic} quantum permutation group. This necessarily gives a strictly intermediate  exotic Tannaka--Krein category:
$$\mathcal{C}_{S_N^+}\lneq \mathcal{C}_{\mathbb{G}}\lneq \mathcal{C}_{S_N}.$$
Banica--Collins \cite{bac} determined that the Tannaka--Krein category of $S_N^+$ is the ``category of non-crossing partitions'' (at parameter $N$), and Banica--Speicher presented that the Tannaka--Krein category of $S_N$ is the ``category of all partitions'' (at parameter $N$).
 Therefore an exotic Tannaka--Krein category is a strict intermediate:
 \begin{equation}\mathcal{NC}\lneq\mathcal{C}\lneq \mathcal{P}.\label{strict}\end{equation}

 Banica--Speicher explain how to associate a linear map $T_p$ to a partition $p$, and this makes it natural to talk about a partition being in the Tannaka--Krein category $\mathcal{C}_{\mathbb{G}}$ of a compact matrix quantum group, thereby giving relations that hold in $C(\mathbb{G})$. For example,
 \begin{align*}
   \Labab\in\mathcal{C}_{\mathbb{G}} & \implies C(\mathbb{G})\text{ commutative}\qquad (\text{Proposition \ref{abelian}}), \\
   \Laaa\in\mathcal{C}_{\mathbb{G}} & \implies u^{2}_{ij}=u_{ij}\qquad(1\leq i,j\leq N).
 \end{align*}
The partitions involved come with the additional data that they are of a set of  $k$ `upper' points and $\ell$ `lower' points (for $k,\,\ell\geq 0$), the set of which is denoted $\mathcal{P}(k,\ell)$.
Banica--Speicher defined a homogeneous compact matrix quantum group  $S_N\leq \mathbb{G}\leq O_N^+$ to be \emph{easy} when there exists a category of partitions $\mathcal{D}\leq\mathcal{P}$ such that
 $$\hom_{\mathcal{C}_{\mathbb{G}}}(k,\ell)=\operatorname{span}(T_p\colon p\in\mathcal{D}(k,\ell)).$$
  See Section \ref{CoP} for more.
  A second major advance was obtained by Banica--Curran--Speicher \cite{bcs}.
\begin{theorem}[Banica--Curran--Speicher \cite{bcs}]\label{noez}
There is no easy exotic quantum permutation group.
\end{theorem}
A non-easy homogeneous compact matrix quantum group is one whose $\hom_{\mathcal{C}}(k,\ell)$ spaces are not all spanned by linear maps coming from partitions, but at least some by linear combinations of such linear maps.

Non-easy homogeneous compact matrix quantum groups $S_N\leq \mathbb{G}\leq O_N^+$ do exist \cite{maa}, and so the result of Banica--Curran--Speicher
does not rule out a non-easy exotic quantum permutation group. It follows from (\ref{strict}) that  an exotic Tannaka--Krein category contains a linear combination $T=\sum_{i=1}^m \alpha_i T_{p_i}$ of  partition-linear maps with crossing-partition terms, and a basic strategy of partial progress towards the Maximality Conjecture is to consider $\langle \mathcal{NC},T\rangle$, the smallest Tannaka--Krein category containing $\mathcal{NC}$ and $T$, and show that it does not generate anything strictly intermediate to $\mathcal{NC}$ or $\mathcal{P}$. Theorem \ref{noez} says that for any $p\in\mathcal{P}(k,\ell)$, $\langle\mathcal{NC},T_p\rangle =\mathcal{P}$ or $\mathcal{NC}$.

\bigskip\noindent

 Banica \cite{ban} defined the ``easiness level'' of a homogeneous compact matrix quantum group $S_N\leq\mathbb{G}\leq O_N^+$ to be the minimal $m$ such that the Tannaka--Krein category is generated by linear combinations of length $m$
$$\mathcal{C}=\left\langle \sum_{i=1}^m\alpha_iT_{p_i}\in \hom_{\mathcal{C}_{\mathbb{G}}}(k,\ell)\colon \alpha_i\in\mathbb{C},\,p_i\in\mathcal{P}(k,\ell),\,k,\ell\geq 0\right\rangle,$$
and that a homogeneous quantum subgroup $\mathbb{H}\leq \mathbb{G}$ is maximal at easiness level $m$ if there exists no easiness level $m$ intermediate quantum subgroup $\mathbb{H}\leq \mathbb{K}\leq\mathbb{G}$. The next advances on the Maximality Conjecture came with this definition about ten years after Theorem \ref{noez}:
\begin{theorem}[Banica (\cite{ban}; Th. 7.10, Prop. 8.3)]\label{bane2} The following statements hold:
\begin{enumerate}
  \item The inclusion $S_5<S_5^+$ is maximal.
  \item There is no easiness level two exotic quantum permutation group.
\end{enumerate}
\end{theorem}
Since these results, little further progress has been made: the author has attempted to outline an approach using idempotent states \cite{mc1,mc3}. This approach yielded  only the result that any \emph{finite} exotic quantum permutation group
must have an algebra of functions of dimension  exceeding $1439$.

\bigskip\noindent

Any exotic quantum permutation group is distinguished from $S_N^+$ by (sufficiently) higher order moments with respect to the Haar state. That is, if $\varphi$ is the Haar state of an exotic quantum permutation group (say with fundamental representation $v$), and $h$ the Haar state of $S_N^+$, then there would exist a $k>0$ such that:
$$\varphi(v_{i_1j_1}v_{i_2j_2}\cdots v_{i_kj_k})\neq h(u_{i_1j_1}u_{i_2j_2}\cdots u_{i_kj_k})\qquad(\mathbf{i},\mathbf{j}\in \{1,2,\dots,N\}^k).$$
  The author explicitly computed the fourth order moments for an exotic quantum permutation group, and showed they are the same as those of $S_N^+$ at moment order four\footnote{This observation is already implicit from dimension counts; see the comments
after Theorem \ref{frob}.} \cite{mc2}. Some of these fourth order moments were found independently \cite{fsw}. In an unpublished note \cite{fre}, an argument of Freslon--Speicher is presented that  shows that  exotic quantum permutation groups are not distinguished from $S_N^+$ at ``moment level'' five:
\begin{theorem}[Freslon--Speicher]
The quantum subgroup $S_N\leq S_N^+$ is maximal at moment level five.
\end{theorem}
This argument will be presented in Appendix \ref{app:freslon}.
These results essentially summarise the current knowledge regarding the conjecture: some works related to the problem include \cite{bcf,grw,juw,scw}; Banica's book \cite{ba2} devotes a section to the problem; while both Freslon \cite{fr2} and Weber \cite{web} name it as an open problem. The Maximality Conjecture, mentioned as an open problem in any talk, lecture, or  survey on the quantum permutation group, remains wide open.

\bigskip\noindent

Despite this interest, relatively few approaches have focused on pushing
these small-parameter results further. The previous state of the art on the problem can  be summarised as: there is no exotic quantum permutation group
\begin{enumerate}
  \item  at $N\leq 5$ (Wang for $N\leq 3$; Banica--Bichon for $N=4$; Banica for $N=5$);
  \item at easiness level two (Banica);
  \item at moment level five (Freslon--Speicher).
\end{enumerate}

The present work makes incremental advances beyond this state of the art. Banica explicitly identified the proof of maximality at easiness level three (or higher) as requiring new ideas and as a ``good question'' \cite{ban}. This has been achieved, pushing Theorem \ref{bane2} (2) by one:
\begin{thmx}[Theorem \ref{thez3}]
The quantum subgroup $S_N\leq S_N^+$ is maximal at easiness level three.
\end{thmx}
Using Frobenius Reciprocity (Theorem \ref{frob}), it suffices to work with the collection of $\hom(0,k)\subseteq (\mathbb{C}^N)^{\otimes k}$ spaces, where, for a partition $p$ of $\{1,2,\dots,k\}$, $\xi_p$ is written for $T_p$. Theorem A is proved by taking a linear combination $\xi\in\hom_{\mathcal{P}}(0,k)$ of length three $$\xi:=\alpha_1\xi_{p_1}+\alpha_2\xi_{p_2}+\alpha_3\xi_{p_3},$$
and showing that the Tannaka--Krein category containing $\mathcal{NC}$ and $\xi$ is all of\footnote{unless $\xi\in\hom_{\mathcal{NC}}(0,k)$ already.} $\mathcal{P}$, the Tannaka--Krein category of $S_N$.
 This advance requires a careful study of partitions (see Section \ref{part}). The Tannaka--Krein methods used by Banica in establishing Theorem \ref{bane2} (2) also had to be organised, and extended, leading to general results that contribute to the broader programme. One such general result in this context is:
\begin{thmx}[Theorem \ref{alleq}]
Let $\xi\in\hom_{\mathcal{P}}(0,k)$ be a nonzero linear combination of crossing-partition vectors such that every summand has equal set of crossings. Then the Tannaka--Krein category generated by $\mathcal{NC}$ together with $\xi$ is the category of all partitions.
\end{thmx}
This handles the case in which the crossing sets of the
summands coincide, one of the configurations needed in
the proof of Theorem \ref{thez3}.
Freslon--Speicher's argument uses the properties of the Haar state of any exotic quantum permutation group to conclude maximality of $S_N\leq S_N^+$ at moment level five (see Appendix \ref{app:freslon}). These moments are  determined by the dimension of the $\hom(0,5)$ space, and this work uses Tannaka--Krein duality to give an alternative proof of maximality at moment level five. Similar deployment of Tannaka--Krein methods at moment level six left a five dimensional gap, but ChatGPT gave a strategy to finish the proof:

\begin{thmx}[Theorem \ref{mom6}]
$S_N\leq S_N^+$ is maximal at moment level six.
\end{thmx}
The author is of the opinion that advances on these small-parameter questions necessarily sharpen our understanding of the inclusion $S_N\leq S_N^+$, and are worthy foothills of any full assault on the Maximality Conjecture. In organising the necessary background for these brute-force assaults on small-parameter questions, it is hoped that this work comprises a basic primer for any such full assault.

\bigskip\noindent

The paper is organised as follows: Section \ref{cmqp} is a basic introduction to the theory of (orthogonal) compact matrix quantum groups; this paper works largely at the level of the dense Hopf*-algebra of regular functions $\mathcal{O}(\mathbb{G})\subseteq C(\mathbb{G})$. In Section \ref{symm} it will be explained why partitions emerge in the representation theory of homogeneous compact matrix quantum groups. In Section \ref{part}, a careful study of partitions is undertaken. Section \ref{Rep} introduces the representation theory of (orthogonal) compact matrix quantum groups from the point of view of Tannaka--Krein duality, and introduces and illustrates how a category of partitions gives a Tannaka--Krein category.  The choice has  been made to define the linear maps coming from partitions using indices, and to suppress their specific action on elementary tensors. The standard introduction to Tannaka--Krein duality establishes that given a Tannaka--Krein category $\mathcal{C}$, one can find a compact matrix quantum group $\mathbb{G}$ whose category of representations is $\mathcal{C}$, using e.g. the approach of Malacarne \cite{mal}. However, when the orthogonal quantum group $O_N^+$ is treated as an ambient compact matrix quantum group, as in the case of homogeneous compact matrix quantum groups (which include exotic quantum permutation groups), things are more straightforward, and the exact role of the Tannaka--Krein axioms is perhaps clearer to see (see Remark \ref{amb}). Accordingly, a stylistic choice is made  to suppress the use of the word ``category'', and instead call a Tannaka--Krein category by a \emph{Tannaka--Krein Dual}. In the case of an orthogonal compact matrix quantum group, the author is of the opinion that the categorical language slightly obscures that the morphisms $T\in\hom(k,\ell)$ are concrete linear maps. Similarly, what are known as the \emph{category operations} --- which at the categorical level resemble 2-morphisms --- get a careful and deliberate analysis in Section \ref{TKO}, and get called instead by \emph{Tannaka--Krein Operations}. These are key because in this context, what is important in Tannaka--Krein duality are the morphisms. In particular, in order to show that a potential exotic Tannaka--Krein dual is actually all of $\mathcal{P}$,  frequently it will need to be shown that the \emph{morphism} $\xi_{\Labab}$ is in the Tannaka--Krein Dual, and these Tannaka--Krein Operations are needed for this.

\bigskip\noindent

The new advances on the problem begin in Section \ref{Generation}, which comprises a series of results which will come together to prove Theorem A in Section \ref{Easi}. In Section \ref{Moment}, attention will turn towards the question of whether an exotic quantum permutation group can be distinguished from $S_N^+$ at moments of order five or six. In the Tannaka--Krein picture, via Theorem \ref{Haardet}, this is a question about $\dim \hom_{\mathcal{C}}(0,5)$ and $\dim \hom_{\mathcal{C}}(0,6)$ for exotic $\mathcal{NC}\lneq \mathcal{C}\lneq \mathcal{P}$. The first part of the Tannaka--Krein approach to moment level six is rather brute-force in nature, requiring the use of mathematical software. The second part, the argument of AI, is more elegant, although still requiring mathematical software. Appendix~\ref{app:freslon} presents Freslon--Speicher's unpublished proof of maximality at moment level five; Appendix~\ref{app:maple} describes the supporting Maple computations and how to reproduce them.

\subsection{Notation}\label{notation}

The parameter $N$ is suppressed in the partition maps
and Tannaka--Krein duals.
The same symbols $\mathcal{NC}$ and $\mathcal P$ denote
both the partition categories and their associated
Tannaka--Krein duals at parameter $N$.
Definitions and conventions are given in the sections
indicated below.

\begingroup
\small
\renewcommand{\arraystretch}{1.15}
\setlength{\LTleft}{0pt}
\setlength{\LTright}{\fill}

\begin{longtable}{
@{}
>{\raggedright\arraybackslash}p{0.29\textwidth}
>{\raggedright\arraybackslash}p{0.65\textwidth}
@{}
}
\textbf{Notation} & \textbf{Meaning}\\
\hline
\endfirsthead

\textbf{Notation} & \textbf{Meaning}\\
\hline
\endhead

\multicolumn{2}{@{}l@{}}{
\emph{Partitions: Sections~\ref{part} and~\ref{CoP}}
}\\[3pt]

$[k]$ &
The ordered set $\{1,\ldots,k\}$.\\

$\mathcal P(A),\ \mathcal P(k)$ &
Partitions of $A$ and of $[k]$, respectively.\\

$\mathcal P(k,\ell)$ &
Partitions with $k$ upper and $\ell$ lower points;
$\mathcal P(k)$ is identified with $\mathcal P(0,k)$.\\

$|p|$ &
The number of blocks of $p$.\\

$[x]_p,\ x\sim_p y$ &
The block containing $x$, and the relation that $x,y$
belong to the same block of $p$.\\

$p\preceq q$ &
The partition $p$ refines $q$.\\

$\mathcal P_4([k])$ &
The set of four-element subsets of $[k]$.\\

$\mathcal{NC},\ \mathcal{CR}$ &
Non-crossing partitions and partitions containing a crossing,
with the appropriate number of points understood.\\

$\mathcal P_2,\ \mathcal{NC}_2$ &
Pair partitions and non-crossing pair partitions.\\

$\operatorname{cr}(p)$ &
The set of four-element crossings of $p$.\\

$\chi(p),\ \chi_0(p)$ &
The crossers of $p$ and their complement in $[k]$.\\

$p^{\mathcal{CR}},\ p^{\mathcal{NC}}$ &
The restrictions of $p$ to its crossers and non-crossers,
respectively.\\

$R_A$ &
Restriction to $A$.
Its action on partition vectors includes powers of $N$
arising from deleted blocks.\\

$B_k,\ C_k,\ S(k,j)$ &
The Bell number, Catalan number, and Stirling number
of the second kind.\\[5pt]

\multicolumn{2}{@{}l@{}}{
\emph{Maps and duals: Sections~\ref{TKD}--\ref{TKO}}
}\\[3pt]

$\mathcal C,\ \mathcal C_{\mathbb G}$ &
The Tannaka--Krein dual under consideration, or that
associated with the specified quantum group $\mathbb G$.\\

$\hom_{\mathcal C}(k,\ell)$ &
Intertwiners from $(\mathbb C^N)^{\otimes k}$ to
$(\mathbb C^N)^{\otimes\ell}$ belonging to $\mathcal C$.\\

$\delta_p$ &
The indicator that indices on the same block agree.
Indices on different blocks are allowed to agree.\\

$T_p,\ \xi_p$ &
The map associated with $p$.
For $p\in\mathcal P(0,k)$, the column map $T_p$ is
identified with the vector $\xi_p$.\\

$\langle\mathcal{NC},\xi\rangle$ &
The Tannaka--Krein dual generated by $\mathcal{NC}$
and $\xi$.\\

$\hom_{\mathcal{CR}}(0,k)$ &
The crossing span
$\operatorname{span}(\xi_p:p\in\mathcal{CR}(k))$.\\

$\operatorname{Op}_{\mathcal C}(k,\ell;k',\ell')$ &
Tannaka--Krein operations taking
$\hom_{\mathcal C}(k,\ell)$ into
$\hom_{\mathcal C}(k',\ell')$.
These operations need not be linear.\\

$\operatorname{Op}_{\mathcal C}(k,\ell)$ &
The abbreviation
$\operatorname{Op}_{\mathcal C}(0,k;0,\ell)$.\\

$\operatorname{Op}_{\mathcal C}(k)$ &
Operations from $\hom_{\mathcal C}(0,k)$ into itself.\\

$T_r,\ T_{r^n}$ &
Cyclic rotation of the tensor legs, and its iterates,
on $\hom_{\mathcal C}(0,k)$.\\

$\overline{T_s}$ &
Conjugate-linear reflection on
$\hom_{\mathcal C}(0,k)$.\\

$\operatorname{rot}_L^m,\ \operatorname{rot}_R^m$ &
Left-side and right-side rotation operations.
Positive $m$ means anticlockwise and negative $m$
means clockwise.\\

$\operatorname{red}_{\mathcal{CR}}$ &
Removal of the non-crossing terms in a chosen
partition expression.\\

$\pi_{\mathcal{CR}}$ &
For $k\le N$, the well-defined projection onto the
crossing component; in this range it is identified
with $\operatorname{red}_{\mathcal{CR}}$.\\[5pt]

\multicolumn{2}{@{}l@{}}{
\emph{Moment calculations: Section~\ref{Moment}}
}\\[3pt]

$M$ &
The matrix of reduced crossing coefficients of the
chosen operations: $10\times10$ at moment five and
$66\times71$ at moment six.\\

$V=[v_1|\cdots|v_5]$ &
At moment six, the matrix whose columns give the
chosen basis of $\ker M$.\\

$A(N)$ &
The $66\times66$ submatrix used to establish the
rank of the moment-six matrix $M(N)$.\\

$R$ &
The restriction $T_r|_{\ker M}$ at moment six.\\

$P$ &
The projection $(I+R+R^2)/3$ on $\ker M$,
with image $\operatorname{Fix}(R)$.\\

$t=(t_1,\ldots,t_5)$ &
Coordinates in the kernel basis:
$\xi=\sum_{i=1}^5t_iv_i$.\\

$Q(t)$ &
The quadratic form obtained from a tensor product
with conjugate-linear reflection.\\

$A_\xi$ &
The endomorphism in $\hom_{\mathcal C}(3,3)$ obtained
by right-side anticlockwise rotation of $\xi\in\hom_{\mathcal C}(0,6)$.\\

$\theta_0,\ \theta$ &
The vector obtained by right-side clockwise rotation of $A_\xi^2$ back to six
lower points, and its crossing component
$\theta=\pi_{\mathcal{CR}}(\theta_0)$.\\

$g_N,\ h_N,\ k_N$ &
The homogeneous quadratic polynomials used in
Proposition~\ref{propoutside}: the first two arise
from kernel membership and the third from fixedness.\\

\end{longtable}
\endgroup

\section{Compact matrix quantum groups}\label{cmqp}
Modern references for compact matrix quantum groups include the books of Banica \cite{ba1} and Freslon \cite{fr1}.
\begin{definition}[Woronowicz \cite{wo1}]\label{ocmqg}
\textbf{An} \emph{algebra of continuous functions on an orthogonal compact matrix quantum group} is a unital $\mathrm{C}^\ast$-algebra $C(\mathbb{G})$:
  \begin{enumerate}
    \item[(i)] generated by the entries of a unitary matrix $u\in M_N(C(\mathbb{G}))$, such that
    \item[(ii)] $u=\overline{u}$, and
    \item[(iii)] there exists a $*$-homomorphism $\Delta:C(\mathbb{G})\to C(\mathbb{G})\underset{\min}{\otimes}C(\mathbb{G})$ called the \emph{comultiplication}, such that
    $$\Delta(u_{ij})= \sum_{k=1}^Nu_{ik}\otimes u_{kj}.$$
  \end{enumerate}
  The matrix  $u\in M_N(C(\mathbb{G}))$ is called the \emph{fundamental representation}.
 \end{definition}
The universal example is \textbf{the} quantum orthogonal group introduced by Wang \cite{wa1}:

\begin{definition}
The algebra of continuous functions on the free orthogonal quantum group $O_N^+$ is the universal $\mathrm C^\ast$-algebra $C_{\mathrm u}(O_N^+)$ generated by the entries of a matrix $u\in M_N(C_{\mathrm u}(O_N^+))$ satisfying conditions {\rm (i)} and {\rm (ii)} of Definition \ref{ocmqg}.
\end{definition}
Algebras of continuous functions on orthogonal compact matrix quantum groups admit a dense Hopf*-algebra $\mathcal{O}(\mathbb{G})\subseteq C(\mathbb{G})$ of \emph{regular functions}, which admits a faithful, invariant Haar state $h:\mathcal{O}(\mathbb{G})\to \mathbb{C}$. In general, the algebra of regular functions on an orthogonal compact matrix quantum group does not have a unique completion to a $\mathrm{C}^*$-algebra (hence the use of \textbf{An} in Definition \ref{ocmqg}). Following the approach of Freslon's book \cite{fr1}, this work will operate on the level of the regular functions, with occasional references to the $\mathrm{C}^*$-setting.

 \bigskip\noindent

   Quotients of $\mathcal{O}(O_N^+)$ by Hopf*-ideals give algebras of regular functions on other orthogonal compact matrix quantum groups. As a general principle, in order to respect the comultiplication structure, these ideals are generated by \emph{Peter--Weyl intertwiners}, linear maps $T\in M_{N^\ell\times N^k}(\mathbb{C})$ such that:
 $$Tu^{\otimes k}=u^{\otimes \ell}T,$$
 where $u^{\otimes k}$ is the Kronecker tensor product of $k$ copies of $u$.  The precise mechanism by which intertwiners generate Hopf*-ideals will be recalled in Section \ref{Rep}.  Examples of ideals given by such intertwiners include:
 \begin{align*}
   I_N & =\langle u_{ij}-u_{ij}^2\colon 1\leq i,j\leq N\rangle, \\
   J_N & =\langle u_{ij}u_{k\ell}-u_{k\ell}u_{ij}\colon 1\leq i,j,k,\ell\leq N\rangle, \\
   F_N & =\langle 1-u_{NN}\rangle,
 \end{align*}
quotienting $\mathcal{O}(O_N^+)$ by which gives, respectively, the algebras of regular functions on \emph{the} quantum permutation group $S_N^+$ \cite{wa2}, the classical orthogonal group $O_N$, and an isotropy subgroup $O_{N-1}^+$. Quotienting by a Hopf*-ideal corresponds to the notion of a quantum subgroup:
\begin{definition}
A quantum subgroup $\mathbb{H}\leq \mathbb{G}$ is given by a surjective $\ast$-homomorphism $\pi: \mathcal{O}(\mathbb{G})\to \mathcal{O}(\mathbb{H})$ that respects the comultiplication in the sense that
$$\Delta_{\mathcal{O}(\mathbb{H})}\circ\pi=(\pi\otimes \pi)\circ \Delta;$$
that is $\ker \pi$ is a Hopf*-ideal.
\end{definition}
That is, there are quantum subgroups
\[
S_N^+ \leq O_N^+, \qquad
O_N \leq O_N^+, \qquad
O_{N-1}^+ \leq O_N^+ .
\]

Furthermore, quotienting by combinations of the above ideals yields
\[
S_N \leq O_N^+, \qquad
S_{N-1}^+ \leq O_N^+, \qquad
O_{N-1} \leq O_N^+, \qquad
S_{N-1} \leq O_N^+ .
\]

\bigskip\noindent

All of the quantum groups of interest in this work are homogeneous compact matrix quantum groups, and there is a specific interest in \emph{exotic quantum permutation groups}:
\begin{definition}
An orthogonal compact matrix quantum group $\mathbb{G}\leq O_N^+$ is called:
\begin{enumerate}
  \item \emph{homogeneous} when it contains the classical permutation group $S_N\leq\mathbb{G}\leq O_N^+$,
  \item an \emph{exotic quantum permutation group} when it is a strict intermediate $S_N\lneq \mathbb{G}\lneq S_N^+$.
\end{enumerate}
\end{definition}
Do exotic quantum permutation groups exist? The Maximality Conjecture predicts that they do not.
\subsection{Partitions via Index Permutation Symmetry}\label{symm}
As will be seen explicitly in Section \ref{CoP}, partitions play a prominent role in the representation theory of homogeneous compact matrix quantum groups. A homogeneous quantum group $\mathbb{G}\leq O_N^+$ has a classical version $G$ containing $S_N\leq G$, whose algebra of regular functions is the quotient of $\mathcal{O}(\mathbb{G})$ by the commutator ideal $J_N$. It is possible to anticipate the emergence of partitions from the point of view of $*$-automorphisms of the algebra of regular functions of homogeneous $\mathbb{G}$ that lift the automorphisms $\alpha_{\sigma,\tau}(f)(g)=f(\sigma^{-1}g\tau)$ of $\mathcal{O}(G)$ induced by the $S_N\times S_N$-action on $G$:

$$\begin{tikzcd}
\mathcal{O}(\mathbb{G}) \arrow[r, "\pi_{\sigma,\tau}"] \arrow[d, two heads]
    & \mathcal{O}(\mathbb{G}) \arrow[d, two heads] \\
\mathcal{O}(G) \arrow[r, "\alpha_{\sigma,\tau}"] & \mathcal{O}(G)
\end{tikzcd}$$
Commutativity of the diagram is the content of Proposition \ref{2.8} below.

\bigskip\noindent

These automorphisms permute the indices
of the matrix coefficients $u_{ij}$. Therefore relations between
the generators can depend only on the pattern of coincidences among the indices.
The use of these $*$-automorphisms is often implicit in the study of homogeneous compact matrix quantum groups, but here they will be made explicit.
\begin{proposition}[Index permutation symmetry]
Let $S_N\leq \mathbb{G}\leq O_N^+$ be homogeneous via a Hopf*-ideal $I\subset \mathcal{O}(O_N^+)$.
For every $(\sigma,\tau)\in S_N\times S_N$ there exists a $*$-automorphism
$\pi_{\sigma,\tau}:\mathcal{O}(\mathbb{G})\to\mathcal{O}(\mathbb{G})$
satisfying
$$
\pi_{\sigma,\tau}(u_{ij})=u_{\sigma(i),\tau(j)} .
$$
Therefore $I$ must be invariant under every $\pi_{\sigma,\tau}$.
\end{proposition}
\begin{proof}
  This follows from Proposition \ref{2.8}, below.
\end{proof}
\begin{proposition}
Let $\mathbb{G}\leq O_N^+$ with classical version $G\leq \mathbb{G}$. There is a diagram of quantum subgroups:
$$\begin{tikzcd}[row sep=1.2em, column sep=2.5em, cells={nodes={anchor=center}}]
G \arrow[r] & \mathbb{G} \\
G \cap S_N \arrow[u]\arrow[r] & \mathbb{G} \cap S_N^+\arrow[u]
\end{tikzcd}$$
\end{proposition}
\begin{proof}
It suffices to describe the Hopf*-ideals corresponding to these quotients: the classical version $G\leq \mathbb{G}$ is given by $\pi_{\operatorname{ab}}:\mathcal{O}(\mathbb{G})\to \mathcal{O}(\mathbb{G})/J_N$; and the quantum subgroup $\mathbb{G}\cap S_N^+$ is given by $\pi_{\operatorname{qp}}:\mathcal{O}(\mathbb{G})\to \mathcal{O}(\mathbb{G})/I_N$. Then $G\cap S_N\leq G$ is given by $\pi_{\operatorname{ab}}(I_N)$, and $G\cap S_N\leq \mathbb{G}\cap S_N^+$ is given by $\pi_{\operatorname{qp}}(J_N)$.
\end{proof}

Let $\pi:\mathcal{O}(\mathbb{G})\to \mathcal{O}(G\cap S_N)$ denote the
quotient map. Each $\sigma\in G\cap S_N$ then defines a character on
$\mathcal{O}(\mathbb{G})$ by
$$\operatorname{ev}_\sigma(f)=\pi(f)(\sigma)\qquad (f\in\mathcal{O}(\mathbb{G})).$$
These characters allow the permutation action of $G\cap S_N$ on the rows and columns of the fundamental representation to extend to an algebra $*$-automorphism of $\mathcal{O}(\mathbb{G})$.
\begin{lemma}\label{permute}
Suppose $\mathbb{G}\leq O_N^+$  with fundamental representation $u\in M_N(\mathcal{O}(\mathbb{G}))$, and classical version $G\leq \mathbb{G}$.  Then for all states $\varphi$ on $\mathcal{O}(\mathbb{G})$ and $\sigma,\tau\in G\cap S_N$:
$$(\operatorname{ev}_{\sigma^{-1}}\star \varphi\star \operatorname{ev}_\tau)(u_{i_1j_1}\cdots u_{i_mj_m})=\varphi(u_{\sigma(i_1)\tau(j_1)}\cdots u_{\sigma(i_m)\tau(j_m)}).$$
\end{lemma}
\begin{proof}
Noting that $\pi(u_{ij})$ is the projection $\mathds{1}_{j\to i}\in \mathcal{O}(G\cap S_N)$, given by $\mathds{1}_{j\to i}(\sigma)=\delta_{i,\sigma(j)}$,  this is then a slight generalisation of (Prop. 6.4, \cite{mc1}), albeit with the same proof.
\end{proof}
 \begin{proposition}\label{2.8}
 Suppose $\mathbb{G}\leq O_N^+$  with fundamental representation $u\in M_N(\mathcal{O}(\mathbb{G}))$, and classical version $G\leq \mathbb{G}$. For $\sigma,\tau\in G\cap S_N$,  the map
 $$\pi_{\sigma,\tau}(u_{i_1j_1}\cdots u_{i_mj_m})=u_{\sigma(i_1)\tau(j_1)}\cdots u_{\sigma(i_m)\tau(j_m)},$$ extended linearly, is an algebra $*$-automorphism $\pi_{\sigma,\tau}:\mathcal{O}(\mathbb{G})\to \mathcal{O}(\mathbb{G})$.
    \end{proposition}
    \begin{proof} To show that the map is well defined, suppose
$$
f=\sum_a \alpha_a x_a=\sum_b \beta_b y_b$$
are two representations of $f$ as linear combinations of monomials
in the generators $u_{ij}$.
Applying the Haar state, which is faithful on $\mathcal{O}(\mathbb{G})$, to $|\sum_a\alpha_a x_a-\sum_b \beta_by_b|^2=0$, then using Lemma \ref{permute}  and the invariance of the Haar state:
    \begin{align}h(u_{i_1j_1}\cdots u_{i_mj_m})=h(\pi_{\sigma,\tau}(u_{i_1j_1}\cdots u_{i_mj_m}))\qquad (\sigma,\,\tau\in G\cap S_N).\label{inv}\end{align}
Using this equation to replace instances of $h(f^*g)$ with $h(\pi_{\sigma,\tau}(f^*g))$ gives:

  \begin{align*}h\left(\left|\pi_{\sigma,\tau}\left(\sum_a \alpha_ax_a\right)-\pi_{\sigma,\tau}\left(\sum_b\beta_b y_b\right)\right|^2\right)&=0,
\end{align*}
which implies $\pi_{\sigma,\tau}$ is well-defined. That $\pi_{\sigma,\tau}$ is  a *-homomorphism is straightforward. Suppose  $f\in \mathcal{O}(\mathbb{G})$ is nonzero, so that $h(|f|^2)>0$. Using (\ref{inv}) again:
\begin{align*}
  h\left(|\pi_{\sigma,\tau}(f)|^2\right) & =h\left(\pi_{\sigma,\tau}(f)^*\pi_{\sigma,\tau}(f)\right) \\
   & =h\left(\pi_{\sigma,\tau}(f^*f)\right) \\
   & =h(f^*f)=h(|f|^2)>0,
\end{align*}
therefore $\ker \pi_{\sigma,\tau}=\{0\}$. Note further that for any $f\in \mathcal{O}(\mathbb{G})$, $$\pi_{\sigma,\tau}(\pi_{\sigma^{-1},\tau^{-1}}(f))=f,$$ and so $\pi_{\sigma,\tau}$ is surjective.
    \end{proof}

The following observations clarify the scope of Proposition \ref{2.8} and its relationship to the $\mathrm{C}^*$-setting:
\begin{remark}
\begin{enumerate}
\item The argument can be more straightforward in the case of compact matrix quantum groups defined by universal $\mathrm{C}^*$-algebras: for example, consider the fundamental representation $u\in M_N(C_{\text{u}}(O_N^+))$, and for $\sigma,\tau\in S_N$, consider the following matrix, which is obtained by permuting the rows and columns:

$$u^{\sigma,\tau}_{ij}=u_{\sigma(i),\tau(j)}.$$
The entries of $u^{\sigma,\tau}$ generate $C_{\text{u}}(O_N^+)$, and the entries of $u^{\sigma,\tau}$ also satisfy the relations of $C_{\text{u}}(O_N^+)$, and thus $\pi_{\sigma,\tau}$ is a *-homomorphism (even a *-automorphism) by the universal property.
\item In the case of a quantum permutation group
$\mathbb{G}\leq S_N^+$, let $G\leq S_N$ denote its
classical version. If the substitution
$\pi_{\sigma,\tau}(u_{ij})=u_{\sigma(i),\tau(j)}$
defines a unital $*$-homomorphism of
$\mathcal{O}(\mathbb{G})$, then
$\varepsilon\circ\pi_{\sigma,\tau}$ is a character.
Since
$$
(\varepsilon\circ\pi_{\sigma,\tau})(u_{ij})
=\delta_{\sigma(i),\tau(j)}
=\delta_{i,(\sigma^{-1}\tau)(j)},
$$
it follows that $\sigma^{-1}\tau\in G$.
In particular, the one-sided substitutions
$\pi_{\sigma,e}$ and $\pi_{e,\sigma}$ define
$*$-automorphisms precisely when $\sigma\in G$,
the converse following from Proposition~\ref{2.8}.
For a two-sided substitution, however, the assumption
that both $\sigma$ and $\tau$ lie outside $G$
does not itself rule out an automorphism.
\item Note that $\pi_{\sigma,\tau}$ is not a Hopf*-algebra homomorphism in general as it does not preserve the counit. However $\pi_{\tau,\tau}:\mathcal{O}(\mathbb{G})\to\mathcal{O}(\mathbb{G})$ is a Hopf*-algebra automorphism that can be viewed as a lift of the inner automorphism of the classical version $G$, $g\mapsto \tau^{-1}g\tau$.
\end{enumerate}
\end{remark}
In the case, therefore, of a homogeneous  compact matrix quantum group $S_N\leq \mathbb{G}\leq O_N^+$,  any pair $(\sigma,\tau)\in S_N\times S_N$ gives a $*$-automorphism $\pi_{\sigma,\tau}:\mathcal{O}(\mathbb{G})\to\mathcal{O}(\mathbb{G})$. This implies, in particular, that the Hopf*-ideal $I$ defining  $\mathcal{O}(\mathbb{G})=\mathcal{O}(O_N^+)/I$ satisfies:
$$f\in I\implies \pi_{\sigma,\tau}(f)\in I\qquad (\sigma,\tau\in S_N).$$
Therefore relations defining $\mathcal{O}(\mathbb{G})$ must be invariant under the permutations $\pi_{\sigma,\tau}$ with $(\sigma,\tau)\in S_N\times S_N$. Such relations depend only on which indices are equal. See Section \ref{CoP} for the precise correspondence.
\section{Partitions}\label{part}
An overarching reference for this section is Nica--Speicher \cite{nis}. Let $k>0$ be an integer, $[k]:=\{1,2,\dots,k\}$, and, for $x,y\in [k]$, $x<y$, denote the interval $[x,y]=\{x,x+1,x+2,\dots,y\}$. A partition $p$ on a finite non-empty set $A$ is a set of non-empty blocks $B_i\subset A$, which disjointly union to $A$:
$$A=\bigsqcup_{i=1,\dots,|p|}B_i.$$
Denote the set of partitions of $A$ by $\mathcal{P}(A)$. In particular, for $A=[k]$, write $\mathcal{P}(k):=\mathcal{P}([k])$.
 If $B\in p$, and $x\in B$, denote $[x]_p=B$. A partition $p\in \mathcal{P}(k)$ induces an equivalence relation $\sim_p$ on $[k]$:
$$x\sim_p y\iff [x]_p=[y]_p.$$

\begin{definition}
Given a partition $p\in\mathcal{P}(k)$, a crossing $\kappa$ is a $\{\kappa_1,\kappa_2,\kappa_3,\kappa_4\}\in\mathcal{P}_4([k])$ such that:
$$\kappa_1<\kappa_2<\kappa_3<\kappa_4,\qquad \kappa_1\sim_p \kappa_3,\quad \kappa_2\sim_p \kappa_4,\text{ but }\kappa_1\not\sim_p \kappa_2.$$
\end{definition}
By abuse of notation, write $(\kappa_1,\kappa_2,\kappa_3,\kappa_4)$ for $\{\kappa_1,\kappa_2,\kappa_3,\kappa_4\}$ when the ordering  $\kappa_1<\kappa_2<\kappa_3<\kappa_4$ is known, otherwise write $\kappa=\{\kappa_1,\kappa_2,\kappa_3,\kappa_4\}$ if the $\kappa_i$ form a crossing, but the exact ordering of $\kappa$ is unknown. When drawing a partition, place the elements of $[k]$ in order from left-to-right\footnote{See Remark \ref{order}.}, so,
for example, in the below there are two crossings, $5<6<7<9$ and $5<6<7<10$:
\begin{equation}\BigPartition{
\Psingletons 0 to 0.25:1,4,8
\Pblock 0 to 0.25:2,3
\Pblock 0 to 0.25:5,7
\Pblock 0 to 0.5:6,9,10
\Ptext(5,-0.2){5}
\Ptext(6,-0.2){6}
\Ptext(7,-0.2){7}
\Ptext(9,-0.2){9}
\Ptext(10,-0.2){10}
}\label{eqpart}\end{equation}
\begin{definition}\label{2.2}
Collect all the $\mathcal{P}(k)$ into an ambient set:
$$\mathcal{P}:=\bigsqcup_{k\geq 1}\mathcal{P}(k).$$
Let $\mathcal{NC}(k)\subseteq \mathcal{P}(k)$ be the set of partitions of $[k]$ that have no crossings, and:
$$\mathcal{NC}:=\bigsqcup_{k\geq 1}\mathcal{NC}(k).$$
Let $\mathcal{CR}:=\mathcal{P}\backslash \mathcal{NC}$ be the set of partitions with a crossing. Given a partition $p\in \mathcal{CR}(k)$, let $\operatorname{cr}(p)\subset \mathcal{P}_4([k])$ be the set of crossings in $p$.
\end{definition}
This definition will be extended in Section \ref{CoP} to include the data of $k$ `upper' points and $\ell$ `lower points'.
\begin{lemma}\label{containlemma}
Let $p\in \mathcal{CR}$ be a partition and $\kappa\in \operatorname{cr}(p)$. Then for all $x\in [\kappa_i]_p$, there is a crossing in $\operatorname{cr}(p)$ containing $x$.
\end{lemma}
\begin{proof}
If $x\in \kappa$, there is nothing to do. Assume that $x\not\in\kappa$ and consider the crossing:
$$\kappa_1<\kappa_2<\kappa_3<\kappa_4.$$
Assume $x\in [\kappa_1]_p$; then wherever $x$ is in the ordering of the $\kappa_i$, there is a crossing including $x$:
\begin{align}
  x<\kappa_1<\kappa_2<\kappa_3<\kappa_4 & \implies (x,\kappa_2,\kappa_3,\kappa_4)\in\operatorname{cr}(p)\nonumber \\
  \kappa_1<x<\kappa_2<\kappa_3<\kappa_4 & \implies  (x,\kappa_2,\kappa_3,\kappa_4)\in\operatorname{cr}(p)\nonumber\\
  \kappa_1<\kappa_2<x<\kappa_3<\kappa_4 & \implies(\kappa_1,\kappa_2,x,\kappa_4)\in\operatorname{cr}(p)\label{eq1} \\
  \kappa_1<\kappa_2<\kappa_3<x<\kappa_4 & \implies(\kappa_1,\kappa_2,x,\kappa_4)\in\operatorname{cr}(p)\nonumber  \\
  \kappa_1<\kappa_2<\kappa_3<\kappa_4<x & \implies(\kappa_2,\kappa_3,\kappa_4,x)\in\operatorname{cr}(p)\nonumber
\end{align}
Similarly if $x\in[\kappa_2]_p$.
\end{proof}
Define $\chi(p)=\cup_{\kappa\in\operatorname{cr}(p)}\kappa$, the subset $\chi(p)\subseteq [k]$ of crossers in $p$, and $\chi_0(p):=[k]\backslash\chi(p)$ the set of non-crossers.
\begin{definition}
Given a non-empty $A\subset[k]$, the \emph{restriction map} $R_{A}:\mathcal{P}(k)\to \mathcal{P}(A)$ is given by:

$$
R_A(p)=\{B\cap A:B\in p,\ B\cap A\neq\emptyset\}.
$$

\end{definition}

\begin{corollary}
A partition $p\in \mathcal{P}$ decomposes as $p=p^{\mathcal{CR}}\sqcup p^{\mathcal{NC}}$, with every element in a block $B_1\in p^{\mathcal{CR}}$ in a crossing, and every element in a block $B_2\in p^{\mathcal{NC}}$ not in a crossing.
\end{corollary}
\begin{proof}
By Lemma \ref{containlemma}, the decomposition is simply
$$p=R_{\chi(p)}(p)\sqcup R_{\chi_0(p)}(p).$$
\end{proof}
\begin{theorem}\label{th0}
If $\operatorname{cr}(p)=\operatorname{cr}(q)$, then $p^{\mathcal{CR}}=q^{\mathcal{CR}}$.
\end{theorem}
\begin{proof}
Suppose that $\operatorname{cr}(p)=\operatorname{cr}(q)$, and $\kappa\in \operatorname{cr}(p)\cap \operatorname{cr}(q)$ such that $[\kappa_i]_{p}\neq [\kappa_i]_{q}$.  Assume without loss of generality that there exists $x\in [\kappa_i]_{p}\backslash [\kappa_i]_{q}$. With indices assumed $\mod 2$, as $\kappa_j\sim_p \kappa_{j+2}$ and $\kappa_j\sim_q \kappa_{j+2}$, and $x\sim_p \kappa_i$, it is not the case that $x\in\kappa$. Therefore $x$ can be placed in one of  five positions, then giving new crossings in $\operatorname{cr}(p)$ depending on $x\in[\kappa_1]_{p}$ or $x\in [\kappa_2]_{p}$. In the first case, there are the same five cases as in (\ref{eq1}). Similarly if $x\in [\kappa_2]_p$. However, recall that $\operatorname{cr}(q)=\operatorname{cr}(p)$, and so each of these cases gives $x\in [\kappa_i]_q$,
and so $p^{\mathcal{CR}}=q^{\mathcal{CR}}$.
\end{proof}

The set $[k]$ can be split into alternating subintervals of crossers $[x_a,y_a]\subset\chi(p)$ and non-crossers $[z_b,w_b]\subset \chi_0(p)$. So say, for  $1\in \chi(p)$ (similar if $1\in\chi_0(p)$):
$$[k]=[1,y_1]\sqcup [z_1,w_1]\sqcup [x_2,y_2]\sqcup[z_2,w_2]\sqcup\cdots$$
As an illustration, in the following
$$\LPartition{0.25:3}{1:1,15;0.5:2,5;0.75:4,8,11;0.25:9,10;0.5:12,13,14;0.25:6,7},$$
the intervals of non-crossers and crossers are:
$$[15]=\{1\}\sqcup\{2\}\sqcup\{3\}\sqcup[4,5]\sqcup[6,7]\sqcup\{8\}\sqcup[9,10]\sqcup\{11\}\sqcup[12,15].$$
Given a permutation $\sigma\in S_k$ and, for $A\subset[k]$, $\sigma(A):=\{\sigma(a)\colon a\in A\}$, define the operation $\sigma:\mathcal{P}(k)\to\mathcal{P}(k)$ given by
$$B\in p\iff \sigma(B)\in \sigma(p).$$
If $r^n\in S_k$ is a cyclic permutation, $r^n(i)=1+(i+n-1)\mod k$, then $r^n(p)$ is a \emph{rotation} of $p\in\mathcal{P}(k)$. For example,
$$r^2(\LPartition{0.25:6}{0.25:4,5;0.5:1,3;0.75:2,7})=\LPartition{0.25:1}{0.25:6,7;0.5:3,5;0.75:4,2}.$$
The following lemma states that when restricted to a crossing union the non-crossers, the block of a non-crosser lies between two adjacent elements of the crossing:
\begin{lemma}\label{nclemma}
Let $p\in \mathcal{CR}(k)$ and $\kappa\in\operatorname{cr}(p)$. If $a\in\chi_0(p)$ is a non-crosser, then, up to rotation, $[a]_{R_{\kappa\cup\chi_0(p)}(p)}\subseteq[z,w]$, some interval of non-crossers $[z,w]\subset \chi_0(R_{\kappa\cup\chi_0(p)}(p))$, entirely contained between two elements of $\kappa$.
\end{lemma}
\begin{proof} Restrict to $R_{\kappa\cup\chi_0(p)}(p)$, which has only one crossing.  Consider $b\neq a$ in $[a]_p$: by Lemma \ref{containlemma}, $b\in\chi_0(p)$, and the only ways to place $a,b\in\chi_0(p)$ in the ordering $\kappa_1<\kappa_2<\kappa_3<\kappa_4$ without forming crossings with $a,b$ are:
$$a<\kappa_1<\kappa_4<b,\qquad b<\kappa_1<\kappa_4<a,$$
which can be rotated to $\kappa_1<\kappa_4<b<a$ respectively $\kappa_1<\kappa_4<a<b$; or
$$\kappa_i<a<b<\kappa_{i+1},\qquad \kappa_i<b<a<\kappa_{i+1}\qquad (\text{for some }1\leq i\leq 3).$$
\end{proof}
A partition $p$ \emph{refines} $q$, written $p\preceq q$, if every block $B\in p$ is a subset of a block of $q$. The following result, about $p^{\mathcal{CR}}$ refining $q$ restricted to the crossers of $p$, may be of independent interest.
\begin{proposition}\label{prop1}
Suppose that $\operatorname{cr}(p)\subset \operatorname{cr}(q)$ for partitions $p,q\in \mathcal{CR}(k)$. Then $p^{\mathcal{CR}}\preceq R_{\chi(p)}(q)$.
\end{proposition}

\begin{proof}
Note that by $\operatorname{cr}(p)\subset \operatorname{cr}(q)$, $\chi(p)\subseteq \chi(q)$, so that $\operatorname{cr}(R_{\chi(p)}(q))\subset\operatorname{cr}(q)$. Let $B$ be a block  in $p^{\mathcal{CR}}$, and let $\{x_1,x_2\}\subset B$ with $x_1<x_2$. It will be shown that $x_1\sim_q x_2$.

\bigskip\noindent

If $\{x_1,x_2\}$ itself appears as the same-block pair in a crossing of $p$, then
$x_1\sim_q x_2$ immediately. Note this happens if $B=\{x_1,x_2\}$.  Otherwise, choose a $z\neq x_2$ such that $\{x_1,z\}$
is a same-block pair in a crossing of $p$. Using this, the constructions of crossings in $p$ with $x_1\sim_px_2$ or $x_2\sim_pz$ yield $x_1\sim_qx_2$, immediately, and via transitivity, respectively.

\begin{enumerate}
\item[\textbf{Case 1}:] $z<x_1<x_2$: then the crossing in $p$ containing $\{x_1,z\}$ has one of the forms
$$(a,z,b,x_1),\qquad (z,a,x_1,b).$$
Thus $a<z<b<x_1<x_2$, or $z<a<x_1<b<x_2$, or $z<a<x_1<x_2<b$.
These yield respectively, the crossings
$$(a,z,b,x_2),\qquad (a,x_1,b,x_2),\qquad (z,a,x_2,b)$$
in $p$, hence also in $q$.  Each forces $x_2$ to be equivalent in $q$ to
$z$ or $x_1$, and therefore $x_1\sim_q x_2$.
\item[\textbf{Case 2}:] $x_1<z<x_2$: let the crossing in $p$ be $(a,x_1,b,z)$ or $(x_1,a,z,b)$; then one of
$$a<x_1<b<z<x_2,\qquad
x_1<a<z<b<x_2,\qquad
x_1<a<z<x_2<b$$
occurs.  These give the crossings
$$(a,x_1,b,x_2),\qquad
(a,z,b,x_2),\qquad
(x_1,a,x_2,b)$$
respectively.  In each case $x_1\sim_q x_2$ (directly or via $z$).
\item[\textbf{Case 3}:] $x_1<x_2<z$: again the crossing in $p$ is $(a,x_1,b,z)$ or $(x_1,a,z,b)$.
The ordering must be one of
$$a<x_1<x_2<b<z,\quad
x_1<a<x_2<z<b,\quad
x_1<x_2<a<z<b,\quad
a<x_1<b<x_2<z.$$
The corresponding crossings are
$$(a,x_2,b,z),\qquad
(x_1,a,x_2,b),\qquad(x_2,a,z,b),\qquad
(a,x_1,b,x_2),$$
each of which is also a crossing in $q$, and each forces $x_1\sim_q x_2$
(directly or via $z$).
\end{enumerate}

In all cases, $x_1\sim_q x_2$.  Since $\{x_1,x_2\}$ was an arbitrary
pair from the block $B$, this shows that $p^{\mathcal{CR}}$ is a refinement of
$R_{\chi(p)}(q)$.
\end{proof}
\begin{remark}
If $\operatorname{cr}(p)\subset \operatorname{cr}(q)$, then if a block $B\in R_{\chi(p)}(q)$ has non-trivial decomposition  in $p^{\mathcal{CR}}$ --- that is $|R_B(p)|>1$ ---  then $|B|\geq 4$. This is because every crossing block must have at least two elements.
\end{remark}
\section{Representation Theory}\label{Rep}
Given a compact matrix quantum group with fundamental representation $u\in M_N(\mathcal{O}(\mathbb{G}))$, rather than focussing on irreducible representations, Tannaka--Krein duality centres the Peter--Weyl representations $(u^{\otimes k})_{k\geq 0}$, and their intertwiners.  In this section, the precise relationship  between partitions coming from index permutation symmetry (see Section \ref{symm}) and the Tannaka--Krein representation theory of homogeneous compact matrix quantum groups will be established.  As in the case of Section \ref{cmqp}, modern references for the representation theory of compact matrix quantum groups include the books of Banica \cite{ba1}, and Freslon \cite{fr1}. As referred to in the Introduction, the categorical language will be somewhat suppressed.

\subsection{Tannaka--Krein Duality}\label{TKD}

Tannaka--Krein duality associates to a compact matrix quantum group $\mathbb{G}\leq O_N^+$ with self-conjugate fundamental representation $u\in M_N(\mathcal{O}(\mathbb{G}))$, a collection of vector spaces $\hom_{\mathcal{C}}(k,\ell)\subseteq M_{N^{\ell}\times N^{k}}(\mathbb{C})$, namely:
$$\hom_{\mathcal{C}}(k,\ell):=\hom(u^{\otimes k},u^{\otimes\ell})\qquad(k,\ell\geq 0),$$
where
$$\hom(u^{\otimes k},u^{\otimes \ell}):=\{T\in M_{N^{\ell}\times N^{k}}(\mathbb{C})\colon Tu^{\otimes k}=u^{\otimes \ell}T\}.$$
The collection of all such vector spaces
$$\mathcal{C}=\bigsqcup_{k,\ell\geq 0}\hom_{\mathcal{C}}(k,\ell),$$
is called the \emph{Tannaka--Krein dual} of $\mathbb{G}$ and inherits from the general representation-theoretic structure the following properties:
\begin{enumerate}
  \item $S,T\in \mathcal{C}\implies S\otimes T\in\mathcal{C}$,
  \item $S\in \hom_{\mathcal{C}}(\ell,r)$, $T\in\hom_{\mathcal{C}}(k,\ell)$, implies $ST\in\hom_{\mathcal{C}}(k,r)$,
  \item $T\in\mathcal{C}\implies T^*\in\mathcal{C}$,
  \item $\operatorname{id}_N\in\mathcal{C}$,
  \item $T_{\Uaa}\in\mathcal{C}$ where $T_{\Uaa}\in \hom_{\mathcal{C}}(2,0)$ is a row vector with columns indexed by $[N]\times [N]$ and entries:
  $$[T_{\Uaa}]_{1,(i,j)}=\delta_{i,j}.$$
\end{enumerate}
Given a quantum subgroup $\mathbb{G}\leq O_N^+$ with Tannaka--Krein dual $\mathcal{C}$, the kernel of the  *-homomorphism $\pi:\mathcal{O}(O_N^+)\twoheadrightarrow \mathcal{O}(\mathbb{G})$, given by $u_{ij}\mapsto \pi(u_{ij})$, is precisely the ideal in $\mathcal{O}(O_N^+)$ given by:
\begin{equation} I_{\mathcal{C}}:=\operatorname{span}\left([Tu^{\otimes k}-u^{\otimes \ell}T]_{\mathbf{i},\mathbf{j}}\colon T\in \hom_{\mathcal{C}}(k,\ell);\,k,\,\ell\geq 0;\,\mathbf{i}\in [N]^\ell,\,\mathbf{j}\in [N]^{k}\right).\label{Hopfideal}\end{equation}
On the other hand, if there is a collection of vector spaces $\mathcal{C}$ satisfying  these axioms, the ideal $I_{\mathcal{C}}$ generated in this way by $\mathcal{C}$ is a Hopf*-ideal so that $\mathcal{O}(O_N^+)/I_{\mathcal{C}}$ gives a compact matrix quantum group $\mathbb{G}_{\mathcal{C}}\leq O_N^+$. Moreover, this duality is bijective: $\mathcal{C}_{\mathbb{G}_{\mathcal{C}}}=\mathcal{C}$. This is the Tannaka--Krein Reconstruction Theorem (see \cite{mal} for a modern treatment).
\begin{remark}\label{amb}
To show that $I_{\mathcal{C}}\subset\mathcal{O}(O_N^+)$ is a Hopf*-ideal in this setting requires showing that it is an ideal, $*$-closed, and a coideal. Notably, $I_{\mathcal{C}}$ is a coideal by construction; for $\mathbf{i}\in [N]^{\ell}$, $\mathbf{j}\in [N]^k$:
$$\Delta([Tu^{\otimes k}-u^{\otimes \ell}T]_{(\mathbf{i},\mathbf{j})})=\sum_{\mathbf{t}\in [N]^\ell}u^{\otimes\ell}_{(\mathbf{i},\mathbf{t})}\otimes[Tu^{\otimes k}-u^{\otimes \ell}T]_{\mathbf{t},\mathbf{j}}+\sum_{\mathbf{m}\in [N]^k}[Tu^{\otimes k}-u^{\otimes \ell}T]_{\mathbf{i},\mathbf{m}}\otimes u^{\otimes k}_{(\mathbf{m},\mathbf{j})};$$
a calculation that requires none of the properties (1)-(5). In contrast, the key property to show that $I_{\mathcal{C}}$ is an ideal is (1), closure in $\mathcal{C}$ under tensor products. All five properties (1)-(5) are required to show that $I_{\mathcal{C}}$ is $*$-closed. See \cite{mal} for precise details.
\end{remark}

\subsection{Categories of Partitions}\label{CoP}
In this section, and indeed throughout the rest of the paper, rather than writing $\mathcal{NC}^{(N)}$ or $\mathcal{P}^{(N)}$, references to specific Tannaka--Krein duals such as $\mathcal{NC}$ and $\mathcal{P}$ will suppress the fact that these are given at some parameter $N$. Let $\mathcal{P}(k,\ell)$ be the set of partitions  of $[k+\ell]$ with the additional data that the elements of $[k]$ are designated `upper points', and the elements of $[k+1,k+\ell]$ are `lower points'. At parameter $N$, each partition $p\in \mathcal{P}(k,\ell)$ gives a linear map $T_p\in M_{N^{\ell}\times N^k}(\mathbb{C})$, with entries:
\begin{equation}\label{eq:partition-map}
[T_{p}]_{(i_1,\dots,i_\ell),(j_1,\dots,j_k)}=\delta_p(\mathbf{i},\mathbf{j}).
\end{equation}
The symbol $\delta_p$ is defined as follows: on the $k$ upper points of $p$ write, left-to-right, the indices $j_1,j_2,\dots,j_k$; and similarly on the $\ell$ lower points of $p$, write, again left-to-right, the indices $i_1,i_2,\dots,i_\ell$. If, for all $x,y\in[k+\ell]$, $x\sim_py$ implies the written indices are equal, then $\delta_p(\mathbf{i},\mathbf{j})=1$; else $\delta_p(\mathbf{i},\mathbf{j})=0$.

\bigskip\noindent

The set of non-crossing partitions and crossing partitions in $\mathcal{P}(k,\ell)$ are respectively denoted $\mathcal{NC}(k,\ell)$ and $\mathcal{CR}(k,\ell)$. Add the $(k,\ell)$ data to Definition \ref{2.2} and redefine $\mathcal{P}$ by:
$$\mathcal{P}:=\bigsqcup_{k,\,\ell\geq 0}\mathcal{P}(k,\ell).$$
Banica--Speicher  defined the notion of a \emph{category of partitions} $\mathcal{D}\leq\mathcal{P}$, and proved that for every $N\in\mathbb{N}$, a category of partitions gives rise to a Tannaka--Krein dual in turn giving a homogeneous compact matrix quantum group $S_N\leq \mathbb{G}_{\mathcal{D}}\leq O_N^+$. By abuse of notation, the Tannaka--Krein dual of $\mathbb{G}$  will also be denoted by $\mathcal{D}$:
$$\mathcal{D}=\operatorname{span}\left(T_p\colon p\in\mathcal{D}\right).$$

This furthermore splits into
$$\mathcal{D}=\bigsqcup_{k,\ell\geq 0}\hom_{\mathcal{D}}(k,\ell),$$
where $\mathcal{D}(k,\ell)=\mathcal{D}\cap\mathcal{P}(k,\ell)$, and $\hom_{\mathcal{D}}(k,\ell):=\operatorname{span}\left(T_p\colon p\in\mathcal{D}(k,\ell)\right)$.  Whether $\mathcal{D}$ refers to the Tannaka--Krein dual of $\mathbb{G}\leq O_N^+$ or the category of partitions will be clear from context. A compact matrix quantum group $S_N\leq \mathbb{G}\leq O_N^+$ whose Tannaka--Krein dual is given by a category of partitions is called \emph{easy}.
\begin{proposition}[Banica--Speicher]
A(n orthogonal) category of partitions is:
\begin{enumerate}
  \item Closed under horizontal concatenation $\otimes$; $p,\,q\in\mathcal{D}\implies p\otimes q\in\mathcal{D}$,
  \item Closed under conformable vertical concatenation $\circ$; $p\in\mathcal{D}(k,\ell)$, $q\in\mathcal{D}(\ell,r)$, then $q\circ p\in\mathcal{D}(k,r)$,
  \item Closed under horizontal reflection ${}^*$; $p\in\mathcal{D}(k,\ell)\implies p^*\in\mathcal{D}(\ell,k)$,
  \item Contains the identity $\idpart\in\mathcal{D}(1,1)$,
  \item Contains the cup $\Uaa\in\mathcal{D}(2,0)$.
\end{enumerate}
\end{proposition}
The above axioms do not involve the parameter $N$. This enters via the following:
\begin{proposition}[Banica--Speicher]\label{func}
For a Tannaka--Krein dual given by a category of partitions:
\begin{enumerate}
  \item $T_p\otimes T_q=T_{p\otimes q}$,
  \item  $T_p\circ T_{q}=N^{c(p,q)}T_{p\circ q}$, where $c(p,q)$ is the number of closed blocks in $p\circ q$,
  \item $T_p^*=T_{p^*}$,
  \item $I_N=T_{\idpart}$,
  \item $T_{\Uaa}\in\mathcal{D}$.
\end{enumerate}
\end{proposition}
Examples of categories of partitions include $\mathcal{P}$ and $\mathcal{NC}$. Also, $\mathcal{P}_2$, where $\mathcal{P}_2\leq\mathcal{P}$ consists of all partitions whose blocks are all of size two, and $\mathcal{NC}_2=\mathcal{P}_2\cap\mathcal{NC}$. Indeed, in this orthogonal setting, $\mathcal{NC}_2$ and $\mathcal{P}$ are extremal categories of partitions, so that for a category of partitions $\mathcal{D}$:
$$\mathcal{NC}_2\leq \mathcal{D}\leq \mathcal{P}.$$

\begin{remark}\label{order}
Note that for the purpose of detecting crossings in elements of $\mathcal{P}(k,\ell)$, the elements of the $k$ upper points are ordered from left-to-right, while the elements of the $\ell$ lower points are ordered from right-to-left. For example, the basic crossing $\{1,3\}\sqcup\{2,4\}\in\mathcal{P}(2,2)$ is drawn as
$$\BigPartition{
\Pline (1,0.85) (3,0.15)
\Pline (1,0.15) (3,0.85)
\Ptext(1,-0.125){4}
\Ptext(3,-0.125){3}
\Ptext(1,1.125){1}
\Ptext(3,1.125){2}
}\text{ rather than }\BigPartition{
\Pline (1,0.75) (1,0.25)
\Pline (3,0.25) (3,.75)
\Ptext(1,0){3}
\Ptext(3,0){4}
\Ptext(1,1){1}
\Ptext(3,1){2}
}.$$
The order used to detect crossings should be distinguished
from the tensor-index convention in the definition of
$\delta_p$: the indices $j_1,\ldots,j_k$ and
$i_1,\ldots,i_\ell$ are written left-to-right on their
respective rows. Thus, in the displayed crossing, lower
point $3$ carries index $i_2$, while lower point $4$
carries index $i_1$. For one-row partitions, we retain the left-to-right
convention of \eqref{eqpart}. Reversing the order of all
points preserves whether a partition is crossing, so
this convention gives the same distinction between
$\mathcal{NC}(k)$ and $\mathcal{CR}(k)$.
\end{remark}
Given a matrix $T\in M_{N^\ell\times N^k}(\mathbb{C})$, denote by $\langle \mathcal{C},T\rangle$ the Tannaka--Krein dual \emph{generated} by $\mathcal{C}$ and $T$, that is, the smallest Tannaka--Krein dual  such that $\mathcal{C}\leq \langle \mathcal{C},T\rangle$ and $T\in \langle \mathcal{C},T\rangle$. Record a triviality for completeness:
\begin{proposition}\label{proptriv}
Let nonzero $\xi=\sum_{i=1}^m\alpha_i\xi_{p_i}$ such that all the $p_i\in \mathcal{NC}(k)$. Then $\langle\mathcal{NC},\xi\rangle=\mathcal{NC}$.
\end{proposition}
\begin{proposition}\label{abelian}
Let $\mathbb{G}\leq O_N^+$ with Tannaka--Krein dual $\mathcal{C}\geq \mathcal{NC}_2$. The Tannaka--Krein dual of the classical version $G\leq \mathbb{G}$ is  $\langle\mathcal{C},T_{\crosspart}\rangle$.
\end{proposition}
\begin{proof}
Let $\mathcal{D}=\langle\mathcal{C},T_{\crosspart}\rangle$. It is required to show that if $\mathcal{D}$ contains $T_{\crosspart}$,  then the commutator ideal $J_N\subset I_{\mathcal{D}}$.
Using \eqref{eq:partition-map}, calculate the entries
of $T_{\crosspart}\in M_{N^2\times N^2}(\mathbb{C})$ by drawing:

$$[T_{\crosspart}]_{(i_1,i_2),(j_1,j_2)}\sim \BigPartition{
\Pline (1,0.85) (3,0.15)
\Pline (1,0.15) (3,0.85)
\Ptext(1,-0.125){$i_1$}
\Ptext(3,-0.125){$i_2$}
\Ptext(1,1.125){$j_1$}
\Ptext(3,1.125){$j_2$}
}\implies [T_{\crosspart}]_{(i_1,i_2),(j_1,j_2)}=\delta_{i_1,j_2}\delta_{i_2,j_1}.$$
From here consider the $((a,b),(c,d))$-entry of $T_{\crosspart}u^{\otimes 2}-u^{\otimes 2}T_{\crosspart}$:
\begin{align*}
  [T_{\crosspart}u^{\otimes 2}-u^{\otimes 2}T_{\crosspart}]_{((a,b),(c,d))} & =\sum_{k_1,k_2=1}^{N}\left([T_{\crosspart}]_{(a,b),(k_1,k_2)}u^{\otimes 2}_{((k_1,k_2),(c,d))}-u^{\otimes 2}_{((a,b),(k_1,k_2))}[T_{\crosspart}]_{(k_1,k_2),(c,d)}\right) \\
   & =u^{\otimes 2}_{((b,a),(c,d))}-u^{\otimes 2}_{((a,b),(d,c))} \\
   & =u_{bc}u_{ad}-u_{ad}u_{bc},
\end{align*}
that is, the relations arising from $T_{\crosspart}$ are precisely the generators $u_{bc}u_{ad}-u_{ad}u_{bc}$.
\end{proof}
\begin{theorem}
Denoting $\mathcal{C}_{\mathbb{G}}$ as the Tannaka--Krein dual of a quantum subgroup $\mathbb{G}\leq O_N^+$:
\begin{enumerate}
\item $\mathcal{C}_{S_N}=\mathcal{P}$,
\item $\mathcal{C}_{O_N}=\mathcal{P}_2$,
\item $\mathcal{C}_{S_N^+}=\mathcal{NC}$,
\item $\mathcal{C}_{O_N^+}=\mathcal{NC}_2$.
\end{enumerate}
\end{theorem}
\begin{proof}
\begin{enumerate}
\item This was described as ``well-known'' by Banica--Speicher (Th. 1.10). Here, Proposition \ref{abelian}  together with Theorem \ref{noez} gives
    $$\mathcal{C}_{S_N}=\langle \mathcal{C}_{S_N^+},T_{\crosspart}\rangle=\mathcal{P}.$$
  \item This is a classical result of Brauer \cite{bra} which was an inspiration for the seminal work of Banica--Speicher.
  \item Banica--Collins (\cite{bac}, Prop. 2.1).
  \item Banica--Collins (\cite{bc2}, Th. 3.1).
\end{enumerate}
\end{proof}

Tannaka--Krein duality says that an exotic quantum permutation group gives rise to an exotic Tannaka--Krein dual:
\begin{equation}\mathcal{NC}\lneq \mathcal{C}\lneq \mathcal{P}.\end{equation}

Note that $\hom_{\mathcal{C}}(0,k)\subseteq M_{N^k\times 1}(\mathbb{C})$, that is, the elements of $\hom_{\mathcal{C}}(0,k)$ are vectors in $(\mathbb{C}^N)^{\otimes k}$, and instead write $\xi_p$ for $T_p$ ($p\in \mathcal{P}(0,k)$).
In the following, unless stated otherwise, when
$$\xi=\sum_{i=1}^m\alpha_i\xi_{p_i},$$
is written, it means each $p_i\in \mathcal{P}(0,k)$, for some uniform $k\geq 1$, and $\xi\neq 0$.

\bigskip\noindent

If $f:\mathcal{P}(k)\to \mathcal{P}(\ell)$, by abuse of notation denote by $f$ also the extension to $f:\hom_{\mathcal{P}}(0,k)\to \hom_{\mathcal{P}}(0,\ell)$. This requires some care: the na\"{\i}ve $f(\xi_p)=\xi_{f(p)}$ is not correct in general because of closed blocks (see Prop. \ref{func} (2)). For example, recall the restriction map $R_A:\mathcal{P}(k)\to \mathcal{P}(|A|)$; for $p=\LPartition{0.25:6}{0.25:1,2,3;0.25:4,5;0.25:7,8}$ and $A=[1,3]\cup[7,8]$, due to the two closed blocks in the middle:
\begin{align*}
  R_{A}(\xi_p) & =N^2\xi_{R_A(p)}.
\end{align*}

\begin{proposition}
The restriction map $R_A\in \hom_{\mathcal{NC}}(k,|A|)$.
\end{proposition}
\begin{proof}
The restriction maps come from tensor products of $\idpart\in \mathcal{NC}(1,1)$ and $\upsingleton\in \mathcal{NC}(1,0)$. Given $A\subseteq[k]$, associate to each  $x\in [k]$ a
$${p_x}=\begin{cases}
        \idpart, & \mbox{if } x\in A, \\
        \upsingleton, & \mbox{otherwise}.
      \end{cases}$$

Now, $R_A=T_p$ where  $p=\bigotimes_{x=1}^k {p_x}$.
\end{proof}
\subsection{Tannaka--Krein Operations}\label{TKO}
An intertwiner $T\in\hom_{\mathcal{C}}(k,\ell)$  relates the objects $u^{\otimes k}$ and $u^{\otimes \ell}$, but the axioms of a Tannaka--Krein dual allow relations between intertwiners. Let $\mathcal{C}$ be a Tannaka--Krein dual.  The composition property gives an implication:
$$T\in \hom_{\mathcal{C}}(k,\ell),\,S\in\hom_{\mathcal{C}}(r,k)\implies TS\in\hom_{\mathcal{C}}(r,\ell).$$
Putting this another way, one can consider the multiplication operations:
\begin{align*}
L_T:\hom_{\mathcal{C}}(r,k)\to \hom_{\mathcal{C}}(r,\ell);\quad L_T(S)=TS,
\\ R_S:\hom_{\mathcal{C}}(k,\ell)\to \hom_{\mathcal{C}}(r,\ell);\quad R_S(T)=TS.
\end{align*}
Similarly, the vector space structure gives an implication:
$$S,\,T\in\hom_{\mathcal{C}}(k,\ell)\implies S+T\in\hom_{\mathcal{C}}(k,\ell).$$
Therefore, given $T\in\hom_{\mathcal{C}}(k,\ell)$, another (affine) operation is given by:
$${+}_{T}:\hom_{\mathcal{C}}(k,\ell)\to \hom_{\mathcal{C}}(k,\ell),\qquad {+}_T(S)=S+ T.$$

Suppose for a map $\Phi$ there is an implication:
$$T\in\hom_{\mathcal{C}}(k,\ell)\implies \Phi(T)\in\hom_{\mathcal{C}}(k',\ell'),$$
then call $\Phi$ a \emph{Tannaka--Krein operation}, and collect all such maps into  $\operatorname{Op}_{\mathcal{C}}(k,\ell;k',\ell')$. In particular, denote $\operatorname{Op}_{\mathcal{C}}(k,\ell):=\operatorname{Op}_{\mathcal{C}}(0,k;0,\ell)$, and $\operatorname{Op}_{\mathcal{C}}(k):=\operatorname{Op}_{\mathcal{C}}(k,k)$. Note in the case of the Tannaka--Krein operation $L_T$, the map is associated with a $T\in \hom_{\mathcal{C}}(k,\ell)$: if $T\in \hom_{\mathcal{C}}(k,\ell)$, then $L_T\in \operatorname{Op}_{\mathcal{C}}(k,\ell)$. This means there is an embedding $\hom_{\mathcal{C}}(k,\ell)\subset \operatorname{Op}_{\mathcal{C}}(k,\ell)$. However, there are more Tannaka--Krein operations than those got from the $\hom_{\mathcal{C}}$ spaces alone.  The tensor and adjoint structures imply that the inclusion $\hom_{\mathcal{C}}(k,\ell)\subset \operatorname{Op}_{\mathcal{C}}(k,\ell)$ can be strict. Furthermore, not only are some of the $\operatorname{Op}_{\mathcal{C}}$ spaces  larger than the corresponding $\hom_{\mathcal{C}}$ spaces, the $\operatorname{Op}_{\mathcal{C}}(k,\ell;k',\ell')$ spaces with $k'\neq k$, $\ell'\neq\ell$ contain no multiplication operations. For example, to construct a rotation $\hom_{\mathcal{C}}(k,\ell)\to \hom_{\mathcal{C}}(0,k+\ell)$, Banica--Speicher use the tensor structure, the rotation map being, for a certain $v\in\hom_{\mathcal{C}}(0,2k)$:
$$\xi_p\mapsto (\operatorname{id}^{\otimes k}\otimes \xi_p)\circ v\qquad (p\in \mathcal{P}(k,\ell)).$$

However, this rotation map is an element of $\operatorname{Op}_{\mathcal{C}}(k,\ell;0,k+\ell)$, and this $\operatorname{Op}_{\mathcal{C}}$ space contains no multiplication operations. There are a number of rotations that are all category operations. Denote by
$\operatorname{rot}_R^m$ the right-side rotation, $m<0$ for clockwise, and $m>0$ for anticlockwise. For an anticlockwise, right-side rotation, the bottom right $m$ points are rotated one-by-one until they form the top right $m$ points, so that, extended linearly:
\begin{equation}\label{eq:rotations}
\operatorname{rot}_R^m:\hom_{\mathcal{C}}(k,\ell)\to \hom_{\mathcal{C}}(k+m,\ell-m)\qquad (-k\leq m\leq \ell),
\end{equation} and similarly for $\operatorname{rot}_L^m$ the left-side rotations.

\bigskip\indent

For $T\in\hom_{\mathcal{C}}(k,\ell)$, $S\in\hom_{\mathcal{C}}(k',\ell')$, the tensor structure gives $\otimes_T:\hom_{\mathcal{C}}(k',\ell')\to \hom_{\mathcal{C}}(k'+k,\ell'+\ell)$, and ${}_S\otimes:\hom_{\mathcal{C}}(k,\ell)\to \hom_{\mathcal{C}}(k'+k,\ell'+\ell)$;
$$\otimes_T(S)=S\otimes T,\qquad {}_S\!\otimes (T)=S\otimes T.$$
This map $\otimes_T\in\operatorname{Op}_{\mathcal{C}}(k',\ell';k'+k,\ell'+\ell)$, and this $\operatorname{Op}_{\mathcal{C}}$ space contains no multiplication operations.  Not every use of the tensor structure is naturally outside $\hom_{\mathcal{C}}(k,\ell)$. For example, for $\mathcal{C}\geq \mathcal{NC}$, $\xi_p\mapsto \xi_p\otimes \xi_q$ with $q\in\mathcal{NC}(\ell)$ is a map in $\hom_{\mathcal{NC}}(k,k+\ell)$.

\bigskip\noindent

A permutation $\sigma\in S_k$ gives a $T_\sigma\in\hom_{\mathcal{P}}(k,k)$ given by the partition $\sigma=\bigsqcup_{i=1}^k\{i,2k+1-\sigma(i)\}$ such that $T_\sigma(\xi_p)=\xi_{\sigma(p)}$, but if $\sigma\neq e$, then the partition $\sigma\in \mathcal{CR}(k,k)$, so  $T_\sigma\not\in \hom_{\mathcal{NC}}(k,k)$. Recall that if $r^n\in S_k$ is a cyclic permutation, then $r^n:\mathcal{P}(0,k)\to \mathcal{P}(0,k)$ is a rotation.
 Given the crossing nature of non-identity $T_{\sigma}\in\hom_{\mathcal{P}}(k,k)$, it is therefore noteworthy that the Tannaka--Krein duals of interest here are closed under such rotations (Proposition \ref{rots}). Let $s(i)=1+(k-i)\mod k$ be the reflection $s\in S_k$, and  $s(p)$ the \emph{reflection} of $p\in\mathcal{P}(k)$. The linear map $T_s$
 also cannot be got from an element of $\hom_{\mathcal{NC}}(k,k)$, but its conjugate $\overline{T_s}$ is an element of $\operatorname{Op}_\mathcal{C}(k)$:

\begin{proposition}[Banica--Speicher]\label{rots} Let $D_k=\langle r,s\rangle$ be the dihedral group of order $2k$.
For any Tannaka--Krein dual $\mathcal{NC}_2\leq \mathcal{C}\leq \mathcal{P}$, the rotations $T_{r^n}\in\operatorname{Op}_{\mathcal{C}}(k)$, and conjugate-linear reflection $\overline{T_s}\in\operatorname{Op}_{\mathcal{C}}(k)$.
\end{proposition}
 Therefore, there is an action $D_k\curvearrowright \hom_{\mathcal{C}}(0,k)$ so that $D_k\subset\operatorname{Op}_{\mathcal{C}}(k)$ (but note that $\overline{T_s}$ being conjugate-linear means there is only a representation of $\mathbb{Z}_k$ rather than of all of $D_k$).
This means, via Theorem \ref{noez}, that for $\mathcal{NC}\leq \mathcal{C}\lneq \mathcal{P}$ there is a strict inclusion $\hom_{\mathcal{C}}(k,k)\subsetneq \operatorname{Op}_{\mathcal{C}}(k)$ ($k\geq 2$), witnessed by e.g. $T_{r}$.

\bigskip\noindent

Note that if $T\in\hom_{\mathcal{C}}(k,\ell)$, $L_T$ is an element of many $\operatorname{Op}_{\mathcal{C}}$ spaces, that is, $L_T(S)=TS$ is an element of $\mathcal{C}$ whenever $T$ is composable with $S\in\mathcal{C}$:
$$L_T\in\bigcap_{r\geq 0}\operatorname{Op}_{\mathcal{C}}(r,k;r,\ell).$$
However $\Phi\in\operatorname{Op}_{\mathcal{C}}$ being composable with $S\in\mathcal{C}$ does not imply that $\Phi(S)\in\mathcal{C}$. For example, $T_r\in\operatorname{Op}_{\mathcal{NC}}(k)$ is composable with $\operatorname{id}\in \hom_{\mathcal{NC}}(k,k)$, but $T_r\circ \operatorname{id}=T_r$ is not an element of  $\hom_{\mathcal{NC}}(k,k)$. That is, $T_r\in\operatorname{Op}_{\mathcal{NC}}(k)$, but not an element of $\operatorname{Op}_{\mathcal{NC}}(k,k;k,k)$.

\bigskip\noindent

The domain of $\Phi\in\operatorname{Op}_{\mathcal{C}}$ is therefore restricted; but this is not to say that the composition of $\Phi\in\operatorname{Op}_{\mathcal{C}}$ with $T\in\mathcal{NC}$ or $\Phi'\in\operatorname{Op}_{\mathcal{C}}$ is forbidden; in general, the compositions $\Phi\circ T$ and $\Phi\circ \Phi'$ can themselves be Tannaka--Krein operations. For example, for $T\in \hom_{\mathcal{C}}(k,\ell)$, the composition $T\circ T_r\in\operatorname{Op}_{\mathcal{C}}(k,\ell)$. This will be seen later with a family of maps in $\operatorname{Op}_{\mathcal{C}}(6,4)$ whose joint kernel is symmetric under the rotation action, the symmetry coming from the fact the family of maps themselves are symmetric under  $T\mapsto T\circ T_{r^n}$.

\bigskip\noindent

In an exotic Tannaka--Krein dual, with the parameter $N$ suppressed, one must be careful not to confuse the use of the vector space structure to subtract the non-crossing parts of a vector from $\hom_{\mathcal{C}}(0,k)$ with the existence of a well-defined idempotent onto the subspace\footnote{Note that $\mathcal{CR}$ is not a category of partitions and not a Tannaka--Krein dual.} $$\hom_{\mathcal{CR}}(0,k):=\operatorname{span}(\xi_p\colon p\in\mathcal{CR}(k)).$$
Consider the linear map
$$\Phi_{k,N}:\mathbb{C}\mathcal{P}(k)\to (\mathbb{C}^N)^{\otimes k}, \qquad p\mapsto \xi_p.$$
This maps:
$$\Phi_{k,N}(\mathbb{C}\mathcal{NC}(k))=\hom_{\mathcal{NC}}(0,k);\qquad \Phi_{k,N}(\mathbb{C}\mathcal{CR}(k))=\hom_{\mathcal{CR}}(0,k).$$
\begin{theorem}[Gromada--Weber \cite{grw}]\label{inj}
For any $k,N>0$,
$$\{\xi_p\colon p\in\mathcal{P}(0,k),\, |p|\leq N\}$$
forms a basis of $\hom_{\mathcal{P}}(0,k)$. As a corollary, $\Phi_{k,N}$ is injective if and only if $k\leq N$.
\end{theorem}

The non-injectivity of $\Phi_{k,N}$ has a number of  corollaries. Firstly, a linear combination of crossing-partition vectors when added to $\mathcal{NC}$ does not necessarily generate a larger Tannaka--Krein dual:
\begin{corollary}\label{bad}
For every $N\geq 3$, there exists, unique up to scaling, a nonzero $\xi\in\hom_{\mathcal{CR}}(0,N+1)$ such that  $\langle \mathcal{NC},\xi\rangle=\mathcal{NC}$.
\end{corollary}
\begin{proof}
The only partition in $\mathcal{P}(N+1)$ that does not appear in the basis in Theorem \ref{inj} is the  discrete partition $\idpart^{\otimes (N+1)}\in\mathcal{NC}(N+1)$. Therefore
$$\xi_{\idpart^{\otimes (N+1)}}=\sum_{p\in\mathcal{P}(N+1)\backslash \{\idpart^{\otimes (N+1)}\}}\beta_p\xi_p\qquad (\beta_p\in\mathbb{C}).$$
As the vectors generating $\hom_{\mathcal{NC}}(0,N+1)$ are linearly independent (\cite{ba1}, Th. 9.9 for $N\geq 4$; \cite{bac}, Th. 3.1, computation of $\det G_{4N}>0$ at $N=3$ for $N=3$), necessarily crossing partitions appear in the linear combination. Subtract from the linear combination any non-crossing partitions:
$$\xi_{\idpart^{\otimes (N+1)}}-\sum_{p\in\mathcal{NC}(N+1)\backslash \{\idpart^{\otimes (N+1)}\}}\beta_p\xi_p=\sum_{p\in\mathcal{CR}(N+1)}\beta_p\xi_p.$$
As $\eta:=\sum_{p\in\mathcal{CR}(N+1)}\beta_p\xi_p$ is in $\mathcal{NC}$, $\langle \mathcal{NC},\eta\rangle =\mathcal{NC}$.
\end{proof}
At $k=N+1$, up to scaling, this vector is unique. There are more such vectors at $k=N+2$:
\begin{corollary}\label{int}
For every $N\geq 4$,
$$\dim \left(\hom_{\mathcal{NC}}(0,N+2)\cap \hom_{\mathcal{CR}}(0,N+2)\right)=\binom{N+2}{2}+1.$$
\end{corollary}
\begin{proof}
Firstly $\dim \hom_{\mathcal{NC}}(0,N+2)=C_{N+2}$, the $(N+2)$-th Catalan number (Th. 9.10, \cite{ba1}). Also
$$\hom_{\mathcal{NC}}(0,N+2)+ \hom_{\mathcal{CR}}(0,N+2)=\hom_{\mathcal{P}}(0,N+2),$$
which, at parameter $N$, has dimension equal to the number of partitions of $[N+2]$ with at most $N$ blocks, determined by Stirling numbers of the second kind:
$$\dim \hom_{\mathcal{P}}(0,N+2)=\sum_{\ell=1}^{N}S(N+2,\ell)=B_{N+2}-\binom{N+2}{2}-1.$$
Note that crossing partitions must have at least two blocks of size at least two. Therefore each $p\in\mathcal{CR}(N+2)$ has at most $|p|=N$ blocks, and so each of the vectors in $\hom_{\mathcal{CR}}(0,N+2)$ are linearly independent by Theorem \ref{inj}, and there are $B_{N+2}-C_{N+2}$ of them. The result then follows by Grassmann's Dimension Formula.
\end{proof}
The following was known to Wang \cite{wa2}; here is a proof using Tannaka--Krein duality:
\begin{corollary}
$S_3^+=S_3$.
\end{corollary}
\begin{proof}
Rerun the proof of Corollary \ref{bad} at $N=3$, to find:
$$\xi_{\Labab}=\frac{1}{\beta_{\Labab}}\left(\xi_{\idpart^{\otimes 4}}-\sum_{p\in\mathcal{NC}(N+1)\backslash \{\idpart^{\otimes 4}\}}\beta_p\xi_p\right),$$
that is $\xi_{\Labab}\in\mathcal{NC}$ at $N=3$.
\end{proof}
In the region where $\Phi_{k,N}$ has non-trivial kernel, while the reduction to the crossing part is not a well-defined map, nonetheless the vector space structure gives the following implication:
$$\eta=\eta_1+\eta_2\implies \eta_2\in \langle\mathcal{NC},\eta\rangle\qquad(\eta_1\in\hom_{\mathcal{NC}}(0,k),\,\eta_2\in \hom_{\mathcal{CR}}(0,k)).$$
The Tannaka--Krein operation that implements this particular implication is $-_{\eta_1}$. However, instead  this implication will be written suggestively by $\operatorname{red}_{\mathcal{CR}}(\eta)=\eta_2$, while understanding $\operatorname{red}_{\mathcal{CR}}$ is not a well-defined map in $\operatorname{Op}_{\mathcal{C}}(k)$ (it is representative-dependent).
\begin{proposition}
For $k\leq N$, there exists in $\operatorname{Op}_{\mathcal{C}}(k)$ a well-defined idempotent
$$\pi_{\mathcal{CR}}:\hom_{\mathcal{C}}(0,k)\to \hom_{\mathcal{CR}}(0,k).$$
\end{proposition}
\begin{proof}
In the region $k\leq N$,
$$\hom_{\mathcal{C}}(0,k)=\hom_{\mathcal{NC}}(0,k)\oplus \left(\hom_{\mathcal{C}}(0,k)\cap \hom_{\mathcal{CR}}(0,k)\right).$$
Therefore, each $\eta\in\hom_{\mathcal{C}}(0,k)$ splits as $\eta_1+\eta_2\in\hom_{\mathcal{NC}}(0,k)\oplus \hom_{\mathcal{CR}}(0,k)$ uniquely, and so $\operatorname{red}_{\mathcal{CR}}$ is a well-defined map, identified with $\pi_{\mathcal{CR}}$.
\end{proof}
The case of equal crossing parts is accompanied by a linear independence result:
\begin{theorem}\label{li}
Let $N\geq 4$. Suppose that $p_1,\dots,p_m$ are distinct elements of $\mathcal{CR}(k)$ with $\operatorname{cr}(p_i)=\operatorname{cr}(p_j)$ for all $1\leq i,j\leq m$. Then $\{\xi_{p_i}\}_{i=1}^m$ is linearly independent.
\end{theorem}
\begin{proof}

The general form of $\xi_{p_i}$ is:
$$\xi_{p_i}=\sum_{j_1,\dots,j_{|p_i|}=1}^N e_{j_{f(1)}}\otimes e_{j_{f(2)}}\otimes \cdots \otimes e_{j_{f(k)}},$$
where $f:\{1,2,\dots,k\}\twoheadrightarrow \{1,2,\dots,|p_i|\}$   with $b\sim_{p_i} c\implies f(b)=f(c)$. By Theorem \ref{th0}, all the crossing parts are equal $p_i^{\mathcal{CR}}=p_j^{\mathcal{CR}}$, and the union of these crossers is uniformly some $\chi(p)\subset [k]$. Recalling $\chi_0(p)=[k]\backslash \chi(p)$, the set of non-crossers, take the $\sigma\in S_{k}$ that maps the $r$-th crosser to $r$, and the $s$-th non-crosser to $|\chi(p)|+s$, and consider the associated change of basis matrix as applied to $(\mathbb{C}^N)^{\otimes k}$.

\bigskip\noindent

Note $\sigma$ is order-preserving on $\chi(p)$ and on its complement $\chi_0(p)$. It sends crossers in $p_i$ to crossers in $\sigma(p_i)$, because if $\kappa\in \operatorname{cr}(p_i)$, then so is $\sigma(\kappa)=\{\sigma(\kappa_1),\sigma(\kappa_2),\sigma(\kappa_3),\sigma(\kappa_4)\}$ because $\sigma$ is block-preserving, and order preserving on $\chi(p)$. Suppose there is a crossing $\kappa\in \operatorname{cr}(\sigma(p_i))$ that is not of this form. It cannot be that the pre-images $\sigma^{-1}(\kappa_i)$ are all non-crossers in $p_i$,  because $\sigma^{-1}$ is order preserving on $\sigma(\chi_0(p))$ and the pre-images would form a crossing among non-crossers in $p_i$. Assume without loss of generality that the blocks $[\sigma^{-1}(\kappa_1)]_{p_i}\subset \chi_0(p)$ and $[\sigma^{-1}(\kappa_2)]_{p_i}\subset \chi(p)$. A crossing must satisfy:
$$\kappa_1<\kappa_2<\kappa_3<\kappa_4,$$
but for all $x\in \chi_0(p)$, $y\in \chi(p)$,
$$\sigma(y)<\sigma(x).$$
In particular, under the hypothesis
$$\sigma(\sigma^{-1}(\kappa_2))<\sigma(\sigma^{-1}(\kappa_1))\implies \kappa_2<\kappa_1,$$
and a similar issue if $[\sigma^{-1}(\kappa_2)]_{p_i}\subset \chi_0(p)$.

\bigskip\noindent

The associated change-of-basis therefore sends $\xi_{p_i}$ to $\xi_p\otimes \xi_{q_i}$, with $p\in \mathcal{CR}$ the image of the common crossing part $\sigma(p_i^{\mathcal{CR}})$, and $q_i\in \mathcal{NC}$ the image of the non-crossing part $\sigma(p_i^{\mathcal{NC}})$. Applying this change of basis to a linear combination $\sum_{i=1}^m \alpha_i \xi_{p_i}$ gives
$$\sum_{i=1}^m\alpha_i(\xi_p\otimes \xi_{q_i})=\xi_p\otimes\left(\sum_{i=1}^m\alpha_i \xi_{q_i}\right).$$
The vectors $\xi_q\in \hom_{\mathcal{NC}}(0,k)$ are linearly independent for $N\geq 4$ (\cite{ba1}, Th. 9.9), and so the linear combination is only equal to zero if all the $\alpha_i$ are zero.
\end{proof}
As an illustration, where $\sigma=(9\,\,5)(10\,\,6)(11\,\,7)(12\,\,8)$
\begin{align*}
\xi&=\alpha_1\xi_{\LPartition{0.4:7,8}{0.4:1,3;0.4:9,11;0.6:2,4;0.6:10,12;0.4:5,6}}+\alpha_2\xi_{\LPartition{}{0.4:1,3;0.4:9,11;0.6:2,4;0.6:10,12;0.6:5,8;0.4:6,7}}+\alpha_3\xi_{\LPartition{0.4:6,8}{0.4:1,3;0.4:9,11;0.6:2,4;0.6:10,12;0.6:5,7}}
\\\implies P_{\sigma}^{-1}(\xi)&=\alpha_1\xi_{\LPartition{0.4:11,12}{0.4:1,3;0.4:5,7;0.6:2,4;0.6:6,8;0.4:9,10}}+\alpha_2\xi_{\LPartition{}{0.4:1,3;0.4:5,7;0.6:2,4;0.6:6,8;0.6:9,12;0.4:10,11}}+\alpha_3\xi_{\LPartition{0.4:10,12}{0.4:1,3;0.4:5,7;0.6:2,4;0.6:6,8;0.6:9,11}}
\\&=\xi_{\LPartition{}{0.4:1,3;0.6:2,4;0.4:5,7;0.6:6,8}}\otimes\left(\alpha_1\xi_{\Laabc}+\alpha_2\xi_{\Labba}+\alpha_3\xi_{\Labac}\right)
\end{align*}
It can be speculated that the production of new linear independence results, and alternative bases of $\hom_{\mathcal{P}}(0,k)$, could be very fruitful.

\section{Generation Results}\label{Generation}
Let $\xi\in \hom_{\mathcal{P}}(0,k)$ and consider $\mathcal{C}=\langle\mathcal{NC},\xi\rangle$, the Tannaka--Krein dual generated by $\mathcal{NC}$ together with $\xi$. In this section, the Tannaka--Krein operations arising from $\mathcal{NC}$ will be used to establish the result that $\langle\mathcal{NC},\xi\rangle=\mathcal{P}$ under various conditions on $\xi$. For the remainder of the paper, assume $N\geq 4$ unless stated otherwise.

\bigskip\noindent

The first result here is that there is no easy exotic quantum permutation group (Theorem \ref{noez}):
\begin{theorem*}[Banica--Curran--Speicher]
  Let $p$ be a crossing partition. Then the Tannaka--Krein dual generated by the non-crossing partitions and a single crossing partition $p\in\mathcal{CR}$ is the category of all partitions.
  $$\langle \mathcal{NC},p\rangle=\mathcal{P}.$$
  \end{theorem*}
\begin{proof}
Assume $p\in \mathcal{CR}(k,\ell)$ and let $\kappa\in\operatorname{cr}(p)$. After rotation to  $\mathcal{CR}(0,k+\ell)$:
$$R_{\kappa}(p)=\Labab,$$
the unique element of $\mathcal{CR}(4)$, that can be rotated into $\crosspart \in \mathcal{CR}(2,2)$, the basic crossing, which together with $\mathcal{NC}$ generates $\mathcal{P}$.\end{proof}

As an illustration, consider the crossing $\kappa=\{4,5,6,9\}$:

$$R_{\kappa}\left(\LPartition{}{0.75:1,2,3,7,8;0.5:4,6;0.25:5,9}\right)=\Labab\text{ , and }R_{\kappa}(\xi_{\LPartition{}{0.75:1,2,3,7,8;0.5:4,6;0.25:5,9}})=N\xi_{\Labab}.$$
The following generalisation will be used repeatedly:
\begin{proposition}\label{lem1}
  Let nonzero $\xi=\sum_{i=1}^m \alpha_i\xi_{p_i}$, with $p_i\in \mathcal{P}(k)$, and suppose there exists a  crossing ``unique to $p_j$'':
  $$\kappa\in \operatorname{cr}(p_j)\left\backslash\bigcup_{i\neq j}\operatorname{cr}(p_i)\right.\,.$$
  Then $\alpha_j=0$ or $\langle \mathcal{NC},\xi \rangle=\mathcal{P}$.
\end{proposition}
\begin{proof}Let $\kappa\in \operatorname{cr}(p_j)$ be ``unique to $p_j$''. Note, with the powers of $N$ coming from deleted blocks:
$$R_{\kappa}(\xi_{p_i})=\begin{cases}
                                                                     N^{|p_j|-2}\xi_{\Labab}, & \mbox{ if } i=j, \\
                                                                     N^{|p_i|-c_i}\xi_{R_\kappa(p_i)}, & \mbox{ for some $0\leq c_i\leq 4$, otherwise}.
                                                                   \end{cases}.$$
                                                                   As $\kappa$ is not a crossing for any $p_i\neq p_j$,  the $R_{\kappa}(p_i)\in \mathcal{NC}(4)$. Hence:
\begin{align*}
 \operatorname{red}_{\mathcal{CR}}(R_\kappa(\xi))&=N^{|p_j|-2} \alpha_j\xi_{\Labab}.
\end{align*}
Therefore $\alpha_j=0$, or  $\xi_{\Labab}\in \langle \mathcal{NC},\xi\rangle$. If $\xi_{\Labab}\in\langle\mathcal{NC},\xi\rangle$, by Theorem \ref{noez}, $\langle \mathcal{NC},\xi\rangle=\mathcal{P}$.

\end{proof}
As an illustration, with $\xi=\alpha_1\xi_{\LPartition{}{0.4:1,3;0.8:2,4;0.8:5,6}}+\alpha_2N\xi_{\LPartition{}{0.2:1,2;0.6:3,5;0.2:4,6}}+\alpha_3\xi_{\LPartition{0.2:1}{0.6:2,4,6;0.2:3,5}}$

\begin{align*}
R_{[1,4]}(\xi)
&=\alpha_1N\xi_{\Labab}
  +\alpha_2N\xi_{\Laabc}
  +\alpha_3\xi_{\Labcb},\\
\operatorname{red}_{\mathcal{CR}}
    \bigl(R_{[1,4]}(\xi)\bigr)
&=\alpha_1N\xi_{\Labab}.
\end{align*}

Theorem \ref{bane2} (2) will be proved here. Banica's proof uses semi-circle cappings:

\begin{definition}
Given $i,j\in [k]$, with $i<j$, and $s=\idpart^{\otimes (i-1)}\otimes\sqcup\otimes \idpart^{\otimes (k-j)}$, the \emph{semi-circle capping} $S_{i,j}:\mathcal{P}(k)\to \mathcal{P}(k+i-j-1)$ is given by:
$$S_{i,j}:=T_{s}\circ R_{[1,i]\cup[j,k]}.$$
\end{definition}
Let $p':=R_{[1,i]\cup[j,k]}(p)$. Note $p' \in \mathcal{P}(k+i-j+1)$ and $T_{s}:\mathcal{P}(k+i-j+1)\to \mathcal{P}(k+i-j-1)$ merges blocks $[i]_{p'}$ and $[j]_{p'}$, then restricts to $[1,i-1]\cup [j+1,k]$.

\begin{proposition}
The semi-circle capping $S_{i,j}\in \hom_{\mathcal{NC}}(k,k+i-j-1)$.
\end{proposition}
As an illustration, $S_{4,7}=\UPartition{\Ps0.7:5,6;0:1,2,3,8}{0.5:4,7}$ and
$$S_{4,7}\left(\alpha_1\xi_{\LPartition{}{0.8:1,3;0.4:2,4;0.4:5,6;0.4:7,8}}+\alpha_2\xi_{\LPartition{0.4:5,6,7,8}{0.8:1,3;0.4:2,4}}\right)= N\alpha_1\xi_{\Labab}+N^2\alpha_2\xi_{\Labac}.$$
Theorem \ref{bane2} (2) is the following:
\begin{theorem}[Banica, \cite{ban}]\label{e2}
Let nonzero $\xi=\alpha \xi_{p}+\beta \xi_{q}$, with $\alpha,\beta\neq 0$, and $p\neq q$. Then $\langle \mathcal{NC},\xi\rangle=\mathcal{P}$ or $\mathcal{NC}$.
\end{theorem}

\begin{proof}
 Note if $p,q\in \mathcal{NC}$ then $\xi\in \hom_{\mathcal{NC}}(0,k)$, and so $\langle \mathcal{NC},\xi\rangle=\mathcal{NC}$ by Proposition \ref{proptriv}. If $p\in \mathcal{NC}$ and $q\in\mathcal{CR}$, $\operatorname{red}_{\mathcal{CR}}(\xi)=\beta\xi_q$, and $\langle \mathcal{NC},q\rangle=\mathcal{P}$ by Theorem \ref{noez}. Consider therefore $p,q\in \mathcal{CR}$:

\begin{enumerate}
  \item[\textbf{Case 1}:]  If there is a crossing ``unique to $p$'' or ``unique to $q$'', then $\langle \mathcal{NC},\xi\rangle=\mathcal{P}$ by Proposition \ref{lem1}.
  \item[\textbf{Case 2}:] Assume that $\operatorname{cr}(p)=\operatorname{cr}(q)$. Choose a $\kappa\in\operatorname{cr}(p)\cap\operatorname{cr}(q)$. Without loss of generality, there exists a block $[x]_p\in p^{\mathcal{NC}}$ with an element $y\in[x]_p\backslash [x]_q$ (otherwise $p=q$). Assume without
loss of generality that $x<y$. Consider $R_{\kappa\cup\{x,y\}}(p)$: by Lemma \ref{nclemma}, up to a rotation, there is no $x<\kappa_i<y$. Therefore, after a rotation:
     \begin{align*}
     p':=R_{\kappa\cup\{x,y\}}(p)=\LPartition{}{0.25:1,3;0.5:2,4;0.5:5,6}&\implies S_{4,5}(p')={\Labab}
     \\q':=R_{\kappa\cup\{x,y\}}(q)=\LPartition{0.25:5,6}{0.25:1,3;0.5:2,4}&\implies S_{4,5}(q')={\Labac}
     \end{align*}
     That is, there is an element of $\langle\mathcal{NC},\xi\rangle$ with a crossing ``unique to one partition'':
     \begin{align*}
     (S_{4,5}\circ R_{\kappa\cup\{x,y\}})(\xi)=\alpha N^{|p|-3}\xi_{\Labab}+\beta N^{|q|-4}\xi_{\Labac}.
     \end{align*}
By Theorem \ref{noez}, $\langle\mathcal{NC},\xi\rangle=\mathcal{P}$.
\end{enumerate}
\end{proof}

A new proof of Theorem \ref{e2} (Case 2) can be obtained from this more general result (which can be adapted to give an alternative proof of Theorem \ref{li}):
\begin{theorem}\label{alleq}
Let nonzero $\xi=\sum_{i=1}^m\alpha_i\xi_{p_i}$ such that    the $p_i\in\mathcal{CR}$ with equal $\operatorname{cr}(p_i)=\operatorname{cr}(p_j)$ for all $i,j\in [m]$.
Then  $\langle\mathcal{NC},\xi\rangle=\mathcal{P}$.
\end{theorem}
\begin{proof}
If $\xi\neq 0$, some $\alpha_j\neq 0$.  By Theorem \ref{th0}, each $p_j=p^{\mathcal{CR}}_j\sqcup p_j^{\mathcal{NC}}$ is such that the $p^{\mathcal{CR}}_i=p^{\mathcal{CR}}_j$ for all $i,j\in[m]$. Define $\chi(p):=\chi(p_j)$ (by assumption this choice is uniform in $j$). Choose any $\kappa$ in the common set of crossings. Let $A:=\kappa\cup \chi_0(p)$ be the crossing $\kappa$ together with all the non-crossers. Writing $q_j:=R_A(p_j)$, as the crossing parts are common, the number of deleted blocks is $|p^{\mathcal{CR}}_i|-2$, uniform in $i$:
      \begin{align*}
      R_A(\xi)&= \sum_{i=1}^mN^{|p^{\mathcal{CR}}_j|-2}\alpha_i \xi_{q_i}\implies \sum_{i=1}^m\alpha_i \xi_{q_i}\in\langle\mathcal{NC},\xi\rangle.
      \end{align*}
     Note  the $q_i\in \mathcal{CR}(\ell)$ with $\ell=k-|\chi(p)|+4$, and $q_i^{\mathcal{CR}}=\Labab$.    Consider in $q_j$ the $\ell-4$ non-crossers in subsets $A_u\subset [\ell]$, and the four crossers in order:
          $$[\ell]=A_0\sqcup\{\kappa_1\}\sqcup A_1 \sqcup \{\kappa_2\}\sqcup A_2\sqcup\{\kappa_3\}\sqcup A_3\sqcup\{\kappa_4\}\sqcup A_4.$$
          By Lemma \ref{nclemma}, note that $x\sim_{q_i} y$ with both $x$ and $y$ in different $A_u$ can only happen with one in $A_0$ and another in $A_4$, and a suitable rotation sets $A_0=\emptyset$.

\bigskip\noindent

To be distinct, $p_i\neq p_j$ for $i\neq j$, crossing partitions with equal crossing parts must be distinguished by their non-crossing parts, $R_{\chi_0(p)}(p_i)\neq R_{\chi_0(p)}(p_j)$. Disregarding $A_u=\emptyset$ as appropriate:
$$R_{\chi_0(p)}(p_i)=\bigsqcup_{u=1}^4R_{A_u}(p_i),$$
disjoint by Lemma \ref{nclemma} in the sense that for any blocks $B_u\subset A_u$, $B_v\subset A_v$, $u\neq v$, $B_u\cap B_v=\emptyset$. This implies that the $q_i\neq q_j$ are unequal, distinguished by at least one $R_{A_u}(q_i)\neq R_{A_u}(q_j)$. The disjointness between the $A_u$ implies that
$$\xi_{q_i}=\sum_{a,b=1}^{N}e_a\otimes\xi_{R_{A_1}(q_i)}\otimes e_b\otimes \xi_{R_{A_2}(q_i)}\otimes e_a\otimes \xi_{R_{A_3}(q_i)}\otimes e_b\otimes \xi_{R_{A_4}(q_i)}.$$
Define $r_u=R_{A_u}(q_j)$, $u\in [1,4]$ as appropriate. Define:
$$S_u=\{r\in\mathcal{NC}(|A_u|)\colon R_{A_u}(q_i)=r,\,\text{ for some }i\in [m]\}.$$
This set includes $r_u$, and, by (\cite{ba1}, Th. 9.9), $X_u=\{\xi_{r}\}_{r\in S_u}$ is  linearly independent, and therefore there exists $\varphi_u=\sum_{r\in S_u}\beta_r\xi_{r}^*$ in $\hom_{\mathcal{NC}}(|A_u|,0)$ such that:
$$\varphi_u(\xi_r)=\begin{cases}
                     1, & \mbox{if } r=r_u \\
                     0, & \mbox{otherwise}.
                   \end{cases}$$
Consider, with $\varphi_u$ excluded for $A_u=\emptyset$ as appropriate, a map $T\in \hom_{\mathcal{NC}}(\ell,4)$ by:
$$T=\operatorname{id}\otimes \varphi_1\otimes \operatorname{id}\otimes\varphi_2\otimes \operatorname{id}\otimes\varphi_3\otimes \operatorname{id}\otimes \varphi_4,$$
and consider $T(\xi_{q_i})$. If every non-crossing interval satisfies $R_{A_u}(q_i)=r_u$ the non-crossing interval of $q_j$, then $T(\xi_{q_i})=\xi_{\Labab}$, otherwise it gives zero. But if all of the non-crossing intervals are equal, then $R_{\chi_0}(p_i)=R_{\chi_0}(p_j)$, and $p_i=p_j$. Therefore, by construction,
$$
T\left(\sum_{i=1}^m\alpha_i\xi_{q_i}\right)
=\alpha_j\xi_{\Labab},
$$
with $\alpha_j\ne0$.
\end{proof}

\begin{proposition}\label{propall2}
Let nonzero $\xi=\sum_{i=1}^m\alpha_i\xi_{p_i}$ such that for each $j\in[m]$, $m\geq 2$, there exists:
$$\kappa^j\in\left.\bigcap_{i\neq j}\operatorname{cr}(p_i)\right\backslash\operatorname{cr}(p_j).$$
Then  $\langle\mathcal{NC},\xi\rangle=\mathcal{P}$.
\end{proposition}
This can be shown using the language of indicator functions. Given a set of $m$ partitions $p_i\in\mathcal{P}([k])$, and a four-element subset $\kappa\in\mathcal{P}_4([k])$, form the vector $\eta^{\kappa}\in\mathbb{C}^m$ with components:
$$\eta^{\kappa}_j=\mathds{1}_{\kappa}(\operatorname{cr}(p_j)).$$
\begin{proposition}
Let nonzero $\xi=\sum_{i=1}^{m}\alpha_i\xi_{p_i}$
 such that:
$$\operatorname{span}\left(\eta^\kappa\colon\kappa\in\mathcal{P}_4([k])\right)=\mathbb{C}^m.$$
Then $\langle\mathcal{NC},\xi\rangle=\mathcal{P}$.
\end{proposition}
\begin{proof}
For some $0\leq c_i\leq 4$, and $q_i\in \mathcal{NC}(4)$,
$$R_{\kappa}(\xi)-\sum_{i\colon\kappa\not\in\operatorname{cr}(p_i)}\alpha_i N^{|p_i|-c_i}\xi_{q_i}=\sum_{i\colon \kappa\in\operatorname{cr}(p_i)}\alpha_iN^{|p_i|-2}\xi_{\Labab}.$$
Define variables $x_j:=\alpha_jN^{|p_j|-2}$. For all the coefficients $\sum_{i\colon \kappa\in\operatorname{cr}(p_i)}\alpha_iN^{|p_i|-2}$ to be zero, avoiding a nonzero crossing coefficient, the linear system
$$\left\{\sum_{i\colon \kappa\in\operatorname{cr}(p_i)}x_i=0\colon\kappa\in \mathcal{P}_4([k])\right\},$$
must have non-trivial solutions. If the vectors $\eta^{\kappa}$ are spanning, a subset $K\subset\cup_{i}\operatorname{cr}(p_i)$ gives a basis $\{\eta^\kappa\colon\kappa\in K\}$, which implies that the linear system:
$$\left\{\sum_{i\colon \kappa\in\operatorname{cr}(p_i)}x_i=0\colon \kappa\in K\right\},$$
indeed has no non-trivial solutions, implying $\xi=0$.
\end{proof}
This is easily seen to be the case for Proposition \ref{propall2}. Where there are only one (Theorem \ref{bane2} (2), Theorem \ref{alleq}), two (Proposition \ref{propall}), or $m-1$ (Proposition \ref{propdis}) elements in $\cup_i\operatorname{cr}(p_i)$, this straightforward approach of the consideration of all $R_{\kappa}(\xi)$ giving zero coefficients of $\xi_{\Labab}$ cannot work as the system isn't of full rank. It can be the case that there is a reduction $m\to m-1$ if some $\alpha_j$ can be forced equal to zero:
\begin{proposition}\label{propall}
Let nonzero $\xi=\sum_{i=1}^m\alpha_i\xi_{p_i}$ such that there exists:
$$\kappa\in\bigcap_{i=1}^m\operatorname{cr}(p_i)\quad\text{ and }\quad \lambda\in\left.\bigcap_{i\neq j}\operatorname{cr}(p_i)\right\backslash\operatorname{cr}(p_j).$$
Then $\alpha_j=0$ and $\sum_{i\neq j}\alpha_i\xi_{p_i}\in \langle\mathcal{NC},\xi\rangle$, or $\alpha_j\neq 0$ and $\langle\mathcal{NC},\xi\rangle=\mathcal{P}$.
\end{proposition}
\begin{proof}
Note that for any $\mu\in\operatorname{cr}(p)$,
$$R_\mu(\xi_{p})=N^{|p|-2}\xi_{\Labab},$$
because all but two blocks are deleted. Also, as $\lambda\not\in \operatorname{cr}(p_j)$, $R_{\lambda}(\xi_{p_j})=N^{|p_j|-c}\xi_{R_\lambda(p_j)}$, for some $0\leq c\leq 4$ and $R_\lambda(p_j)\in\mathcal{NC}(4)$.  Therefore:
\begin{align*}
R_\kappa(\xi)&=\sum_{i=1}^m\alpha_i N^{|p_i|-2}\xi_{\Labab},
\\  R_{\lambda}(\xi) & =\alpha_j N^{|p_j|-c}\xi_{R_{\lambda}(p_j)}+\sum_{i\neq j} N^{|p_i|-2}\alpha_i\xi_{\Labab} \\
   \implies \operatorname{red}_{\mathcal{CR}}(R_{\lambda}(\xi))&=\sum_{i\neq j}N^{|p_i|-2}\alpha_i\xi_{\Labab}.
   \end{align*}
   From here
   \begin{align*} R_\kappa(\xi)-\sum_{i\neq j} N^{|p_i|-2}\alpha_i\xi_{\Labab}=\alpha_j N^{|p_j|-2}\xi_{\Labab}
   \end{align*}
   which implies that $\alpha_j=0$, or $\xi_{\Labab}\in\langle\mathcal{NC},\xi\rangle$.
\end{proof}
The final proposition in this section is more specific, and the proof is a brute-force case-by-case analysis:

\begin{proposition}\label{propdis}
Let $p,q,r\in\mathcal{CR}(k)$ and let
$\xi=\alpha_1\xi_p+\alpha_2\xi_q+\alpha_3\xi_r$
be nonzero, with
$$
\operatorname{cr}(p)
=\operatorname{cr}(q)\sqcup\operatorname{cr}(r).
$$
Then  $\langle\mathcal{NC},\xi\rangle=\mathcal{P}$.

\end{proposition}

\begin{proof}
Take $\kappa\in \operatorname{cr}(p)\cap \operatorname{cr}(q)\cap\overline{\operatorname{cr}(r)}$ and $\lambda\in  \operatorname{cr}(p)\cap \operatorname{cr}(r)\cap\overline{\operatorname{cr}(q)}$. Then, for some $0\leq c_r,c_q\leq 4$, and $r',\,q'\in\mathcal{NC}$:
      \begin{align*}
         R_{\kappa}(\xi) & =(\alpha_1N^{|p|-2}+\alpha_2 N^{|q|-2})\xi_{\Labab}+\alpha_3 N^{|r|-c_r}\xi_{r'}, \\
         R_{\lambda}(\xi) & =(\alpha_1N^{|p|-2}+\alpha_3 N^{|r|-2})\xi_{\Labab}+\alpha_2 N^{|q|-c_q}\xi_{q'},  \\
       \end{align*}
       but it is possible for both coefficients of $\xi_{\Labab}$ to be zero. Record these being zero:
           \begin{align}
             \alpha_1N^{|p|-2}+\alpha_2 N^{|q|-2} & =0\nonumber \\
             \alpha_1N^{|p|-2}+\alpha_3 N^{|r|-2} & =0\label{case0}
           \end{align}Note if any $\alpha_i=0$ then $\xi=0$. Consider $R_{\kappa\cup\lambda}(\xi)$ and break into the four cases parameterised by $|\kappa\cup\lambda|\in\{5,6,7,8\}$ (note that $|\kappa\cup \lambda|=4$ if and only if $\kappa=\lambda$).
       \begin{enumerate}
         \item[\textbf{Case 1}:] In the case $|\kappa\cup\lambda|=5$, as $R_{\kappa\cup \lambda}(p)\in \mathcal{P}(5)$ has at least two crossings, it cannot have a singleton (otherwise there would be only a single crossing), therefore the block-size pattern of $R_{\kappa\cup\lambda}(p)$ is 3-2. There are five possibilities, each of which have two crossings:
             $$\LPartition{}{0.25:1,3;0.5:2,4,5},\,\LPartition{}{0.25:2,4;0.5:1,3,5},\,\LPartition{}{0.25:3,5;0.5:1,2,4},\,\LPartition{}{0.25:1,4;0.5:2,3,5},\,\LPartition{}{0.25:2,5;0.5:1,3,4},$$
           which, by inspection, can each be rotated to $\LPartition{}{0.25:1,3;0.5:2,4,5}$, so assume rotation to $R_{\kappa\cup\lambda}(p)=\LPartition{}{0.25:1,3;0.5:2,4,5}$. Consider $R_{\kappa\cup\lambda}(q)$, which contains $\kappa$ as a crossing. The block-size patterns for elements of $\mathcal{CR}(5)$ are 3-2 and 2-2-1, and the first of these options has two crossings, but $q$ only has one (otherwise there would be a crossing common to both $q$ and $r$). Therefore the block-size pattern of $R_{\kappa\cup\lambda}(q)$ is 2-2-1, which means  the block of size three in $R_{\kappa\cup\lambda}(p)$ is split in $R_{\kappa\cup\lambda}(q)$. It cannot be that
           $$R_{\kappa\cup\lambda}(q)=\LPartition{0.25:2}{0.5:1,3;0.6:4,5},$$
           as this has no crossing, therefore, without loss of generality, assume that:
           $$R_{\kappa\cup\lambda}(q)=\LPartition{0.25:4}{0.25:1,3;0.5:2,5}\text{ , and }R_{\kappa\cup\lambda}(r)=\LPartition{0.25:5}{0.25:1,3;0.5:4,2}.$$
Therefore, after a rotation, noting that $p$ loses all but two blocks, and $q$ and $r$ lose all but three blocks:
$$R_{\kappa\cup\lambda}(\xi)=\alpha_1N^{|p|-2}\xi_{\LPartition{}{0.25:1,3;0.5:2,4,5}}+\alpha_2N^{|q|-3}\xi_{\LPartition{0.25:4}{0.25:1,3;0.5:2,5}}+\alpha_{3}N^{|r|-3}\xi_{\LPartition{0.25:5}{0.25:1,3;0.5:4,2}}.$$
Now preserve $\{1,2,3\}$ and merge $\{4,5\}$ using  $T_{\BigPartition{
\Pblock 0.75 to 0.625:4,5
\Psingletons 0.75 to 0.625:1,2,3
\Psingletons 0.75 to 0.375:1,2,3
\Pline (1,0.375) (1,0.625)
\Pline (2,0.375) (2,0.625)
\Pline (3,0.375) (3,0.625)
\Pline (4.5,0.375) (4.5,0.625)
}}\in\hom_{\mathcal{NC}}(5,4)$ to get
\begin{align*}(T_{\BigPartition{
\Pblock 0.75 to 0.625:4,5
\Psingletons 0.75 to 0.625:1,2,3
\Psingletons 0.75 to 0.375:1,2,3
\Pline (1,0.375) (1,0.625)
\Pline (2,0.375) (2,0.625)
\Pline (3,0.375) (3,0.625)
\Pline (4.5,0.375) (4.5,0.625)
}}\circ R_{\kappa\cup\lambda})(\xi)&=\alpha_1N^{|p|-2}\xi_{\Labab}+\alpha_2N^{|q|-3}\xi_{\Labab}+\alpha_3N^{|r|-3}\xi_{\Labab}
\\&=\left(\alpha_1N^{|p|-2}+\alpha_2N^{|q|-3}+\alpha_3N^{|r|-3}\right)\xi_{\Labab}.
\end{align*}
Suppose that the coefficients of $\xi_{\Labab}$ here, along with those of $R_{\kappa}(\xi)$ and $R_{\lambda}(\xi)$, are all zero:
\begin{align*}
  \alpha_1N^{|p|-2}+\alpha_2N^{|q|-2} & =0 \\
  \alpha_1N^{|p|-2}+\alpha_3N^{|r|-2} & =0 \\
 \alpha_1N^{|p|-2}+\alpha_2N^{|q|-3}+\alpha_3N^{|r|-3}  & =0
\end{align*}
The first two equations give:
$$\alpha_2N^{|q|-3}=\alpha_3N^{|r|-3}=-\alpha_1N^{|p|-3},$$
which when substituted into the third gives:
$$\alpha_1N^{|p|-3}(N-2)=0,$$
which, as $N>3$, implies $\alpha_1=0$, and so $\xi=0$, a contradiction. Therefore, in this case, $\langle\mathcal{NC},\xi\rangle=\mathcal{P}$.
\item[\textbf{Case 2}:] Let $\kappa=\{\kappa_a,\kappa_b,\mu_a,\mu_b\}$ and $\lambda=\{\lambda_a,\lambda_b,\mu_a,\mu_b\}$. Assume without loss of generality that $\kappa_a<\kappa_b$, $\lambda_a<\lambda_b$, $\mu_a<\mu_b$ (relabelling after rotations if necessary). With the crossings in $p$ in inherited from $q$ and $r$, if $\mu_a\sim \mu_b$ in any partition, then $\mu_a\sim \mu_b$ in all partitions.
\begin{enumerate}
  \item[\textbf{Subcase 2a}:] $\mu_a\sim\mu_b$ in all partitions. Then, also $\kappa_a\sim\kappa_b$ in $q$ and $p$, and $\lambda_a\sim\lambda_b$ in $r$ and $p$. It is not possible that $\mu_a\sim_p \kappa_b,\,\lambda_b$ or $\mu_b\sim_p\kappa_a,\,\lambda_a,$ as $\kappa,\lambda\in\operatorname{cr}(p)$.
  \begin{enumerate}
    \item[\textbf{Subsubcase 2ai}:]  Suppose $\kappa_a\sim_p\lambda_a$. Then $\{\kappa_a,\kappa_b,\lambda_a,\lambda_b\}\in R_{\kappa\cup\lambda}(p)$ and so, in order that crossings with $\{\mu_a,\mu_b\}$ be made, up to a rotation, with $\kappa_a,\,\lambda_a<\mu_a$, and $\mu_a<\kappa_b,\,\lambda_b<\mu_b$:
                 $$R_{\kappa\cup\lambda}(p)=
                 \BigPartition{
\Pblock 0 to 0.25:2,4,8,10
\Pblock 0 to 0.5:6,12
\Ptext(6,-0.2){$\mu_a$}
\Ptext(12,-0.2){$\mu_b$}
},$$
                 placing $\mu=(\kappa_a,\mu_a,\lambda_b,\mu_b)$ in $\operatorname{cr}(p)$. Note that if $\mu\in\operatorname{cr}(q)$, then
                 $$\mu\cap \lambda=\{\mu_a,\lambda_b,\mu_b\}\implies |\mu\cap\lambda|=3\implies|\mu\cup\lambda|=5,$$
                 which is \textbf{Case 1} (and similarly if $\mu\in\operatorname{cr}(r)$).
    \item[\textbf{Subsubcase 2aii}:] Therefore assume that $\lambda_a\not\sim_p\kappa_a$.  This means that
                 $$R_{\kappa\cup\lambda}(p)=\BigPartition{
\Pblock 0 to 0.25:4,8
\Pblock 0 to 0.5:2,10
\Pblock 0 to 0.75:6,12
\Ptext(6,-0.2){$\mu_a$}
\Ptext(12,-0.2){$\mu_b$}
}\quad\text{ or }\quad \BigPartition{
\Pblock 0 to 0.25:2,8
\Pblock 0 to 0.5:4,10
\Pblock 0 to 0.75:6,12
\Ptext(6,-0.2){$\mu_a$}
\Ptext(12,-0.2){$\mu_b$}
}.$$
Therefore, where
                  $$s=\BigPartition{
\Psingletons 0.25 to 0.375:1.5,3,4.5,6
\Psingletons 0.75 to 0.625:3,6
\Pblock 0.75 to 0.625:1,2
\Pblock 0.75 to 0.625:4,5
\Pline (1.5,0.375) (1.5,0.625)
\Pline (4.5,0.375) (4.5,0.625)
\Pline (3,0.375) (3,0.625)
\Pline (6,0.375) (6,0.625)
},$$
                  $T_s\in\hom_{\mathcal{NC}}(6,4)$ merges $\{\lambda_a,\kappa_a\}$, merges $\{\lambda_b,\kappa_b\}$, and preserves $\mu_a,\mu_b$, yielding $$(T_s\circ R_{\kappa\cup\lambda})(\xi_p)=N^{|p|-3}\xi_{\Labab}.$$
It is not possible that $\lambda_a\sim_q \lambda_b$ or $\kappa_a\sim_r\kappa_b$ as these would give $\lambda\in \operatorname{cr}(q)$, $\kappa\in\operatorname{cr}(r)$ respectively. If $\kappa_a\sim_q\lambda_a$ or $\kappa_a\sim_q\lambda_b$; or $\lambda_a\sim_r\kappa_a$ or $\lambda_a\sim_r\kappa_b$ as these would form crossings $\{\lambda_a,\mu_a,\kappa_b,\mu_b\}$, $\{\kappa_a,\mu_a,\lambda_b,\mu_b\}$ in $\operatorname{cr}(q)$; respectively $\{\kappa_a,\mu_a,\lambda_b,\mu_b\}$, $\{\lambda_a,\mu_a,\kappa_b,\mu_b\}$ in $\operatorname{cr}(r)$ that would intersect with $\lambda$, respectively $\kappa$ in three places, reverting to \textbf{Case 1}. Therefore assume that $R_{\kappa\cup\lambda}(q)$ and $R_{\kappa\cup\lambda}(r)$ are each one of four different partitions:
\begin{align}\BigPartition{
\Psingletons 0 to 0.25:2,8
\Pblock 0 to 0.75:6,12
\Pblock 0 to 0.5:4,10
\Ptext(6,-0.2){$\mu_a$}
\Ptext(12,-0.2){$\mu_b$}
},&\qquad \BigPartition{
\Psingletons 0 to 0.25:2,10
\Pblock 0 to 0.75:6,12
\Pblock 0 to 0.5:4,8
\Ptext(6,-0.2){$\mu_a$}
\Ptext(12,-0.2){$\mu_b$}
},\nonumber\\[2ex] \BigPartition{
\Psingletons 0 to 0.25:4,8
\Pblock 0 to 0.75:6,12
\Pblock 0 to 0.5:2,10
\Ptext(6,-0.2){$\mu_a$}
\Ptext(12,-0.2){$\mu_b$}
},&\qquad\BigPartition{
\Psingletons 0 to 0.25:4,10
\Pblock 0 to 0.75:6,12
\Pblock 0 to 0.5:2,8
\Ptext(6,-0.2){$\mu_a$}
\Ptext(12,-0.2){$\mu_b$}
}.\label{case2aii}
\end{align}
Note that in these four cases
$$(T_s\circ R_{\kappa\cup\lambda})(\xi_q)=N^{|q|-4}\xi_{\Labab};\quad (T_s\circ R_{\kappa\cup\lambda})(\xi_r)=N^{|r|-4}\xi_{\Labab}.$$
It follows that
$$(T_s\circ R_{\kappa\cup\lambda})(\xi)=(\alpha_1N^{|p|-3}+\alpha_2N^{|q|-4}+\alpha_3N^{|r|-4})\xi_{\Labab}.$$
However if this is zero and (\ref{case0}) are also satisfied, then $\xi=0$. It follows that in this case $\langle\mathcal{NC},\xi\rangle=\mathcal{P}$.
   \end{enumerate}
  \item[\textbf{Subcase 2b}:] $\mu_a\not\sim\mu_b$ in all partitions. Then, also $\kappa_a\not\sim\kappa_b$ in $q$ and $p$, and $\lambda_a\not\sim\lambda_b$ in $r$ and $p$. On the other hand $\mu_a\sim\kappa_a$, $\mu_b\sim\kappa_b$ in $q$ and $p$; $\mu_a\sim \lambda_a$, $\mu_b\sim \lambda_b$ in $r$ and $p$. It follows that $p=\{\kappa_a,\lambda_a,\mu_a\}\sqcup\{\kappa_b,\lambda_b,\mu_b\}$. It also follows that $R_{\kappa}(q)=\{\kappa_a,\mu_a\}\sqcup\{\kappa_b,\mu_b\}$, and $R_{\lambda}(r)=\{\lambda_a,\mu_a\}\sqcup\{\lambda_b,\mu_b\}$. This implies that, in $R_{\kappa\cup\lambda}(p)$, $\kappa_a$ and $\mu_a$ cannot be adjacent; and $\lambda_a$ and $\mu_a$ cannot be adjacent. If $\kappa_a$ and $\lambda_a$ are adjacent in $\kappa\cup\lambda$, then the $\kappa_a$ in $\kappa=\{\kappa_a,\kappa_b,\mu_a,\mu_b\}$ may be replaced with $\lambda_a$ to produce a new crossing $\mu=\{\lambda_a,\kappa_b,\mu_a,\mu_b\}$ in $\operatorname{cr}(p)$. If $\mu\in\operatorname{cr}(q)$, then $|\mu\cap\lambda|=3$; similarly if $\mu\in\operatorname{cr}(r)$; both reverting to \textbf{Case 1}. Similarly it can be assumed that $\kappa_b$ is not adjacent to $\lambda_b$.  Therefore it can be assumed that
      $$R_{\kappa\cup \lambda}(p)=\BigPartition{
\Pblock 0 to 0.25:1,3,5
\Pblock 0 to 0.5:2,4,6
},$$
and thus $\mu=\{\lambda_a,\kappa_b,\mu_a,\mu_b\}$ in $\operatorname{cr}(p)$, with a similar reversion to \textbf{Case 1}.
\end{enumerate}

        \item[\textbf{Case 3}:] In the case $|\kappa\cup\lambda|=7$, there is one common crosser $|\kappa\cap\lambda|=1$, say $\kappa\cap\lambda=\{1\}$. Then $\kappa=(1,\kappa_2,\kappa_3,\kappa_4)$, and $\lambda=(1,\lambda_e,\lambda_o,\lambda_{e+2})$. This implies that in $p$, $\{1,\kappa_3,\lambda_o\}\subset [\kappa_1]_p$. Then, there are four possibilities for placing $\kappa_3$ in $\lambda$, so that $R_{\lambda\cup\{\kappa_3\}}(p)$ is one of:
    \begin{align}\BigPartition{
\Pblock 0 to 0.25:4,10
\Pblock 0 to 0.5:2,6,8
\Ptext(2,-0.2){$1$}
\Ptext(4,-0.2){$\lambda_e$}
\Ptext(8,-0.2){$\lambda_o$}
\Ptext(10,-0.2){$\lambda_{e+2}$}
\Ptext(6,-0.2){$\kappa_3$}
},&\qquad \BigPartition{
\Pblock 0 to 0.25:4,10
\Pblock 0 to 0.5:2,6,8
\Ptext(2,-0.2){$1$}
\Ptext(4,-0.2){$\lambda_e$}
\Ptext(6,-0.2){$\lambda_o$}
\Ptext(10,-0.2){$\lambda_{e+2}$}
\Ptext(8,-0.2){$\kappa_3$}
},\nonumber\\[2ex]\BigPartition{
\Pblock 0 to 0.25:4,8
\Pblock 0 to 0.5:2,6,10
\Ptext(2,-0.2){$1$}
\Ptext(4,-0.2){$\lambda_e$}
\Ptext(6,-0.2){$\lambda_o$}
\Ptext(8,-0.2){$\lambda_{e+2}$}
\Ptext(10,-0.2){$\kappa_3$}
}, &\qquad\BigPartition{
\Pblock 0 to 0.25:6,10
\Pblock 0 to 0.5:2,4,8
\Ptext(2,-0.2){$1$}
\Ptext(6,-0.2){$\lambda_e$}
\Ptext(8,-0.2){$\lambda_o$}
\Ptext(10,-0.2){$\lambda_{e+2}$}
\Ptext(4,-0.2){$\kappa_3$}
}.\label{case3}
\end{align}
These give, respectively, crossings in $p$:
$$\mu=(1,\lambda_e,\kappa_3,\lambda_{e+2}),\qquad\mu'=(1,\lambda_e,\kappa_3,\lambda_{e+2}),$$
$$\mu''=(\lambda_e,\lambda_{o},\lambda_{e+2},\kappa_3),\qquad\mu'''=(\kappa_3,\lambda_e,\lambda_o,\lambda_{e+2}),$$
all of which intersect $\lambda$ at three points, reverting to \textbf{Case 1}.
        \item[\textbf{Case 4}:] In the case $|\kappa\cup\lambda|=8$, $|\kappa\cap\lambda|=0$, it will be shown that $R_{\kappa\cup\lambda}(p)$ has block-size pattern 2-2-2-2. It is not possible that $R_{\kappa\cup\lambda}(q)$ (or $R_{\kappa\cup\lambda}(r)$) has a crossing block of size greater than two, as, if it did, it would have to include an element of $\lambda$, thus by Lemma \ref{containlemma}, giving a crossing $\mu$ such that $|\lambda\cap\mu|\neq 0$, reverting to a previous case. Therefore the crossing part of $R_{\kappa\cup\lambda}(q)$ has a 2-2 block-size pattern, as does the crossing part of $R_{\kappa\cup\lambda}(r)$. Therefore, each of $R_{\kappa\cup\lambda}(q)$ and $R_{\kappa\cup\lambda}(r)$ have exactly one crossing, and therefore $R_{\kappa\cup\lambda}(p)$ has exactly two crossings. Suppose $R_{\kappa\cup\lambda}(p)$ contains the crossing $\kappa$, and has a 4-2 block-size pattern, say $\{x_1,x_2,x_3,x_4\}\sqcup\{\kappa_{i+1},\kappa_{i+3}\}$. To make a crossing, there must exist $x_a,\,x_b$ such that
            $$\kappa_{i+1}<x_a<\kappa_{i+3}<x_b\text{ or }x_a<\kappa_{i+1}<x_b<\kappa_{i+3}.$$
            Now consider another $x_c\in\{x_1,x_2,x_3,x_4\}\backslash\{x_a,x_b\}$. There are three choices to place  $x_c$ in this ordering:
            $$x_c<\kappa_i,\,\quad \kappa_i<x_c<\kappa_{i+1}\text{, and }\kappa_{i+1}<x_c,$$
            each giving a second crossing in $R_{\kappa\cup\lambda}(p)$. Finally placing the fourth $x_d$ in this ordering produces a third crossing in $R_{\kappa\cup\lambda}(p)$, which is not possible. Therefore it cannot be the case that any $\kappa_i\sim_p \lambda_j$, as this would make a block of size four in $R_{\kappa\cup\lambda}(p)$.

\bigskip\noindent
            Therefore $R_{\kappa\cup\lambda}(p)$ has block-size pattern 2-2-2-2:
            $$R_{\kappa\cup\lambda}(p)=\{\kappa_1,\kappa_3\}\sqcup\{\kappa_2,\kappa_4\}\sqcup\{\lambda_1,\lambda_3\}\sqcup\{\lambda_2,\lambda_4\}.$$
           Consider $R_{\kappa}(p)=\Labab$. Enlarge to $R_{\kappa\cup\{\lambda_{j},\lambda_{j+2}\}}(p)$, and attempt to place $\{\lambda_j,\lambda_{j+2}\}=\Laa$ inside $\kappa=\Labab$ without producing a new crossing. It forces $\lambda_{j}$ to be adjacent to $\lambda_{j+2}$ or
           $$R_{\kappa\cup\{\lambda_{j},\lambda_{j+2}\}}(p)=\BigPartition{
\Pblock 0.125 to 0.25:6,10
\Pblock 0.125 to 0.5:4,8
\Pblock 0.125 to 0.75:2,12
\Ptext(2,0){$\lambda_{j}$}
\Ptext(12,0){$\lambda_{j+2}$}
},$$

\smallskip\noindent
           which can be rotated to have $\lambda_{j}$ adjacent to $\lambda_{j+2}$, and therefore:
           $$R_{\kappa\cup\{\lambda_{j},\lambda_{j+2}\}}(p)=\BigPartition{
\Pblock 0 to 0.25:2,6
\Pblock 0 to 0.25:10,12
\Pblock 0 to 0.5:4,8
}.$$
Now place $\{\lambda_{j+1},\lambda_{j+3}\}$ in this.
The two points must be simultaneously adjacent in
$R_{\kappa\cup\{\lambda_{j+1},\lambda_{j+3}\}}(p)$
and form a crossing $\lambda$.
This gives three possibilities
(the first two of which are equal):
$$\LPartition{}{0.25:1,3;0.25:5,7;0.5:2,4;0.5:6,8},\qquad\LPartition{}{0.25:1,3;0.25:6,8;0.5:2,4;0.5:5,7},\qquad\LPartition{}{0.25:2,4;0.25:6,8;0.5:3,5;0.75:1,7},$$
and the third which can be rotated to $\Labab\Labab$. Therefore, for non-crossing partitions $q',r'\in\mathcal{NC}(4)$, of sizes $|q'|$ and $|r'|$,  up to a rotation:
$$R_{\kappa\cup \lambda}(\xi)=\alpha_1 N^{|p|-4}\xi_{\Labab\Labab}+\alpha_2 N^{|q|-2-|q'|}\xi_{\Labab\otimes q'}+\alpha_3N^{|r|-2-|r'|}\xi_{r'\otimes\Labab}.$$
Where $s=\idpart^{\otimes4}\otimes\Uaaaa$, it is the case that:
\begin{align}(T_s\circ R_{\kappa\cup \lambda})(\xi)=\alpha_1 N^{|p|-3}\xi_{\Labab}+\alpha_2 N^{|q|-1-|q'|}\xi_{\Labab}+\alpha_3N^{|r|-1-|r'|}\xi_{r'},\label{case4}\end{align}
yielding $(\alpha_1N^{|p|-3}+\alpha_2N^{|q|-1-|q'|})\xi_{\Labab}\in \langle\mathcal{NC},\xi\rangle$. Similarly, using $s'=\Uaaaa\otimes\idpart^{\otimes 4}$, $(\alpha_1N^{|p|-3}+\alpha_3N^{|r|-1-|r'|})\xi_{\Labab}\in\langle\mathcal{NC},\xi\rangle$. Multiply these coefficients by $N$:
 \begin{align*}
             \alpha_1N^{|p|-2}+\alpha_2 N^{|q|-|q'|} & =0, \\
             \alpha_1N^{|p|-2}+\alpha_3 N^{|r|-|r'|} & =0.
           \end{align*}
If these together with equations (\ref{case0}) are to have non-trivial solutions, then $q'$ and $r'$ must both have two blocks, $|q'|=|r'|=2$. This implies that:

$$q',\,r'\in\{\Labaa,\Labba,\Laaba,\Laabb,\Labbb,\Laaab\}.$$
\begin{enumerate}
  \item[\textbf{Subcase 4a}]: If $q'=\Labaa$, consider $T_{s_1}\in\hom_{\mathcal{NC}}(8,4)$, where
  $$s_1=\BigPartition{
\Psingletons 1 to 0.5:1,2,3,6
\Pblock 1 to 0.75:4,5
\Pblock 1 to 0.75:7,8
},$$
Observe:
\begin{align*}
T_{s_1}(\Labab\otimes\Labab)&={\Labab}
\\ T_{s_1}(\Labab\otimes\Labaa)&={\Labac}
\end{align*}
For $r'\in \{\Labbb,\Labaa,\Laaba,\Laaab\}$, $T_{s_1}(r'\otimes\Labab)$ has a singleton, so is also non-crossing. Furthermore if $r'\in\{\Laabb,\Labba\}$, then $T_{s_1}(r'\otimes\Labab)=r'$, non-crossing. Therefore $(T_{s_1}\circ R_{\kappa\cup \lambda})(\xi)$ has a crossing ``unique to $(T_{s}\circ R_{\kappa\cup\lambda}(p))$'', so by Proposition \ref{lem1}, $\langle\mathcal{NC},\xi\rangle =\mathcal{P}$.

\bigskip\noindent

This assertion about $(T_{s_a}\circ R_{\kappa\cup\lambda})(p)$ will also hold in \textbf{Subcases 4b-d}:
  \item[\textbf{Subcase 4b}]: If $q'\in\{\Labba,\Laaba,\Laabb\}$, consider $T_{s_2}\in\hom_{\mathcal{NC}}(8,4)$, where
  $$s_2=\BigPartition{
\Psingletons 1 to 0.5:1,2,3,7
\Psingletons 1 to 0.75:6,8
\Pblock 1 to 0.75:4,5
},$$
It is the case that:
\begin{align*}
T_{s_2}(\Labab\otimes\Labab)&={\Labab}  \\
  T_{s_2}(\Labab\otimes q') & ={\Labac}
\end{align*}
Similarly to the above, for $r'\in \{\Labbb,\Labaa,\Laaba,\Laaab\}$, $T_{s_2}(r'\otimes\Labab)$ has a singleton, so is also not crossing, and otherwise $T_{s_2}(r'\otimes\Labab)=r'$, non-crossing.

  \item[\textbf{Subcase 4c}]: If $q'=\Labbb$, consider $T_{s_3}\in\hom_{\mathcal{NC}}(8,4)$, where
    $$s_3=\BigPartition{
\Psingletons 1 to 0.5:1,2,3,8
\Pblock 1 to 0.75:4,5
\Pblock 1 to 0.75:6,7
}$$
Again, $T_{s_3}(\Labab\Labab)=\Labab$, while $T_{s_3}(\Labab\Labbb)$ has a singleton. Similarly to before, $T_{s_3}(r'\otimes\Labab)\in\mathcal{NC}$.

\item[\textbf{Subcase 4d}]: If $q'=\Laaab$, consider $T_{s_4}\in\hom_{\mathcal{NC}}(8,4)$, where:
$$s_4=\BigPartition{
\Psingletons 1 to 0.25:1,2,3,8
\Psingletons 1 to 0.75:5,7
\Pblock 1 to 0.75:4,6
}$$
Again, $T_{s_4}(\Labab\Labab)=\Labab$, $T_{s_4}(\Labab\Laaab)=\Labac$, and as before $T_{s_4}(r'\otimes\Labab)\in\mathcal{NC}$.
\end{enumerate}

       \end{enumerate}

\end{proof}

\section{Easiness Level}\label{Easi}
By Theorem \ref{noez}, there is no easy exotic quantum
permutation group
$S_N\lneq\mathbb{G}\lneq S_N^+$.
Therefore, necessarily, the Tannaka--Krein dual
$\mathcal{C}$ of an exotic quantum permutation group
must contain some nonzero linear combination
$\xi=\sum_{i=1}^{m}\alpha_i\xi_{p_i}$ with all
$\alpha_i\neq0$, $p_i\in\mathcal{CR}$ distinct, and $m>1$.
Following Banica \cite{ban}, the easiness level of
$\mathbb{G}$ is the least $m$ such that $\mathcal{C}$ is
generated by its intertwiners expressible as linear
combinations of at most $m$ partition maps, as recalled
in the Introduction; the level is infinite if no such
finite $m$ exists.

In the present setting, maximality at easiness level
$m$ can be expressed by the following generation
criterion: for every $k$ and every
$\xi\in\hom_{\mathcal{P}}(0,k)$ expressible using at most
$m$ partition terms,
$\langle\mathcal{NC},\xi\rangle=\mathcal{P}$ or
$\langle\mathcal{NC},\xi\rangle=\mathcal{NC}$.
Banica showed that the inclusion is maximal at
easiness level two
(\cite{ban}, Prop. 8.3; Theorem \ref{bane2}(2) here).
If $S_N\leq S_N^+$ is maximal at all easiness levels,
then there is no exotic Tannaka--Krein dual, and thus
$S_N\leq S_N^+$ is maximal.

\bigskip\noindent

This is a further advance on this problem:
\begin{theorem}\label{thez3}
The inclusion $S_N\leq S_N^+$ is maximal at easiness level three.
\end{theorem}
\begin{proof}
Let nonzero $\xi=\sum_{i=1}^3\alpha_i\xi_{p_i}$ in $\hom_{\mathcal{P}}(0,k)$. Consider $\langle \mathcal{NC},\xi\rangle$, and, where the ambient set is the set of subsets of $[k]$ of size four, consider the Venn diagram with sets $\operatorname{cr}(p_1),\,\operatorname{cr}(p_2),\,\operatorname{cr}(p_3)$, with regions associated with binary strings of length three; i.e. $\operatorname{cr}(p_1)\cap\overline{\operatorname{cr}(p_2)}\cap\operatorname{cr}(p_3)\sim 101$, etc., and classified by weight, e.g. the weight of 110 is two.

\smallskip\noindent

If these sets are all empty, only the region of weight zero is shaded. Then  by Proposition \ref{proptriv} $\langle \mathcal{NC},\xi\rangle=\mathcal{NC}$.

\smallskip\noindent

If any region of weight one is shaded, then by Proposition \ref{lem1}, $\langle \mathcal{NC},\xi\rangle=\mathcal{P}$.

\smallskip\noindent

Otherwise, the regions shaded are those of weights zero (which can be ignored), two, and three. If only the region of weight three is shaded, then the $\operatorname{cr}(p_i)$ are all equal, and by Theorem \ref{alleq}, $\langle \mathcal{NC},\xi\rangle=\mathcal{P}$.

\smallskip\noindent

If the region of weight three, and at least one of weight two are shaded, then Proposition \ref{propall} gives either some $\alpha_j=0$, and so $\sum_{i\neq j}\alpha_i\xi_{p_i}\in \langle\mathcal{NC},\xi\rangle$, in which case $\langle\mathcal{NC},\xi\rangle=\mathcal{P}$ by Theorem \ref{bane2} (2); or $\alpha_j\neq 0$ and $\langle \mathcal{NC},\xi\rangle=\mathcal{P}$.

\smallskip\noindent

If the region of weight three is unshaded, and all the regions of weight two are shaded, Proposition \ref{propall2} gives $\langle \mathcal{NC},\xi\rangle=\mathcal{P}$.

\bigskip\noindent

The two remaining cases are if one or two regions of weight two are shaded. If it is only one region of weight two, say $\operatorname{cr}(p_1)\cap \operatorname{cr}(p_2)\cap\overline{\operatorname{cr}(p_3)}$, then all of $\operatorname{cr}(p_3)$ is unshaded. In this case $\operatorname{cr}(p_3)=\emptyset$, therefore $\operatorname{red}_{\mathcal{CR}}(\xi)=\alpha_1\xi_{p_1}+\alpha_2\xi_{p_2}$. As $\alpha_1\xi_{p_1}+\alpha_2\xi_{p_2}\in \langle \mathcal{NC},\xi\rangle$,  by Theorem \ref{bane2} (2), $\langle \mathcal{NC},\xi\rangle=\mathcal{P}$. The remaining case therefore is two regions of weight two, in which case, without loss of generality:
$$\operatorname{cr}(p_1)=\operatorname{cr}(p_2)\sqcup\operatorname{cr}(p_3).$$
By Proposition \ref{propdis}, $\langle \mathcal{NC},\xi\rangle=\mathcal{P}$.
  \end{proof}
  Note that the above theorem implies that a linear combination of length three in $\hom_{\mathcal{CR}}(0,k)$ is always outside $\mathcal{NC}$. It shows that
$$\xi=\sum_{i=1}^{3}\alpha_i\xi_{p_i}\implies \langle \mathcal{NC},\xi\rangle=\mathcal{P}\qquad (\alpha_i\in\mathbb{C},\,(\alpha_1,\alpha_2,\alpha_3)\neq 0,\,p_i\in\mathcal{CR}).$$
 Due to Corollary \ref{bad}, the definition of maximality at easiness level $m$ requires the \emph{or $\langle\mathcal{NC},\xi\rangle=\mathcal{NC}$} clause. This leads to a question:
\begin{question}\label{streasi}
What is the maximal $m$ such that for all $\xi=\sum_{i=1}^{m}\alpha_i\xi_{p_i}$ with all $\alpha_i\neq0$, $p_i\in\mathcal{CR}$, $\langle \mathcal{NC},\xi\rangle=\mathcal{P}$?
\end{question}
After Corollary \ref{maxm5}, it will be seen that at parameter $N=4$ there exists a linear combination of length ten, $\xi=\sum_{p\in\mathcal{CR}(5)}\alpha_p\xi_p$, with all  $\alpha_p\neq 0$, such that $\langle\mathcal{NC},\xi\rangle=\mathcal{NC}$. This shows that the uniform answer to Question \ref{streasi} is at most ten. It seems this phenomenon can be ruled out using the following rewriting protocol: given $\xi=\sum_i\alpha_i\xi_{p_i}$ in $\hom_{\mathcal{CR}}(0,k)$, express $\xi$ in the Theorem \ref{inj} basis:
$$\xi=\sum_{\underset{|p|\leq N}{p\in\mathcal{NC}(k)}}\beta_p\xi_p+\sum_{\underset{|p|\leq N}{p\in\mathcal{CR}(k)}}\beta_p\xi_p.$$
Then apply $\operatorname{red}_{\mathcal{CR}}$:
$$\sum_{\underset{|p|\leq N}{p\in\mathcal{CR}(k)}}\beta_p\xi_p\in\langle\mathcal{NC},\xi\rangle.$$
This suggests a decoupling from the uniform-in-$N$ Question \ref{streasi} to an $N$-dependent question: however, while the above vector is not in $\operatorname{span}\left(\xi_p\colon p\in \mathcal{NC}(k),\,|p|\leq N\right)$, it could potentially still be in $\mathcal{NC}$ via non-crossing partitions with greater than $N$ blocks.

\section{Moment Level}\label{Moment}
The previous section concerns generation in the context
of a linear combination of a fixed number $m$ of
partitions of generic $[k]$. An alternative paradigm
is to study generation in the context of an arbitrary
linear combination of partitions of a fixed $[k]$.
The key result, whose proof uses Frobenius reciprocity
and Tannaka--Krein duality (see \cite{mal}), is the
following (and a similar version can be stated for any
strict intermediate compact matrix quantum group
$\mathbb{G}\lneq\mathbb{G}'\lneq\mathbb{G}''$).
\begin{theorem}[Woronowicz]\label{frob}
If there exists an exotic quantum permutation group $S_N\lneq\mathbb{G}\lneq S_N^+$ with Tannaka--Krein dual $\mathcal{C}$, then there exists $k\geq 1$ such that
$$\dim\hom_{\mathcal{NC}}(0,k)\lneq \dim\hom_{\mathcal{C}}(0,k)\lneq \dim\hom_{\mathcal{P}}(0,k).$$
\end{theorem}
Note that the dimension on the left is the Catalan number $C_k$, while the dimension on the right is related to Stirling numbers of the second kind, and bounded above by the Bell number $B_k$ (see Theorem \ref{inj}).
If for all $\xi=\sum_{p\in\mathcal{P}(k)}\alpha_p\xi_p$, it is the case that $\langle \mathcal{NC},\xi \rangle=\mathcal{P}$ or $\langle \mathcal{NC},\xi \rangle=\mathcal{NC}$,  then say that the inclusion $S_N\leq S_N^+$ is \emph{maximal at moment level $k$}.

\bigskip\noindent

  In fact, by (\cite{fsw}, Appendix A) and \cite{mcc}, the inclusion $S_N\leq S_N^+$ is maximal at moment level four. This can be seen directly from the observation that the fourth Bell and Catalan numbers differ by one. Freslon--Speicher used the properties of the Haar state to conclude that the inclusion $S_N\leq S_N^+$ is also maximal at moment level five (see Appendix A). If $S_N\leq S_N^+$ is maximal at all moment levels, then there is no exotic Tannaka--Krein dual, and thus $S_N\leq S_N^+$ is maximal. This is because the dimensions of the $\hom_{\mathcal{C}}(0,k)$ spaces determine the values of the Haar state on monomials of degree $k$ (see Appendix A: Theorem \ref{Haardet}).

\subsection{Moment Level Five}
To show maximality at moment level five using Tannaka--Krein duality, the proof strategy proceeds as follows: assuming $\dim \hom_{\mathcal{C}}(0,5)>\dim \hom_{\mathcal{NC}}(0,5)$, one obtains a nonzero crossing vector $\eta\in \hom_{\mathcal{C}}(0,5)$ with  crossing partition terms. Then, $\xi=\operatorname{red}_{\mathcal{CR}}(\eta)$ is some nonzero $\xi\in\mathcal{C}$ with all terms coming from crossing partitions, which, because there are $B_5-C_5=10$ crossing partitions, is a linear combination of at most ten crossing-partition vectors:
$$\xi=\sum_{i=1}^{10}\alpha_i\xi_{p_i}\qquad (p_i\in\mathcal{CR}(5)).$$
There are nine useful maps in $\hom_{\mathcal{NC}}(5,4)\subseteq \hom_{\mathcal{C}}(5,4)$: these are the five singleton cappings, $S_x:=R_{[5]\backslash\{x\}}$; and the four merge maps  $M_{x,x+1}\in\hom_{\mathcal{C}}(5,4)$ given by the elements of $\mathcal{NC}(5,4)$:
$$M_{1,2}=\mathsf{Y}|||,\quad M_{2,3}=|\mathsf{Y}||,\quad M_{3,4}=||\mathsf{Y}|,\quad M_{4,5}=|||\mathsf{Y}.$$

Thereafter,  the larger $\operatorname{Op}_{\mathcal{NC}}(5)$ structure can be used to provide a tenth map and to conclude, like Freslon--Speicher, maximality at moment level five. Consider the crossing  partition:
  $$p=
\BigPartition{
\Psingletons 1 to 0.75:5
\Pline (5,0.75) (1,0.25)
\Pline (1,0) to (1,1)
\Pline (2,0) to (2,1)
\Pline (3,0) to (3,1)
\Pline (4,0) to (4,1)
}$$
Note $T_p\not\in \hom_{\mathcal{NC}}(5,4)$; however, $T_p\in\operatorname{Op}_{\mathcal{NC}}(5,4)$:
$$T_p=M_{1,2}\circ T_r.$$
Denote this $M_{5,1}$.  Applying  $(\operatorname{red}_{\mathcal{CR}}\circ T)$ to $\xi$ gives:
$$\left(\sum_{i=1}^{10}d_T(i)N^{c_T(i)}\alpha_i\right)\xi_{\Labab}.$$
Noting $T=T_{p_T}$ for some $p_T\in\mathcal{P}(5,4)$, the coefficients $d_T(i)$ are given by:
$$d_T(i)=\begin{cases}
                1, & \mbox{if }  p_T\circ p_i=\Labab,\\
                0, & \mbox{otherwise}.
              \end{cases}$$
The powers of $N$ are determined by the number of closed blocks in $p_T\circ p_i$, with $c_T(i)\in\{0,1\}$ for the singleton cappings $S_x$, and $c_T(i)=0$ for the merging maps. Define a matrix $M\in M_{10}(\mathbb{C})$:
\begin{equation}\label{eq:M5}
M(T,j)=d_T(j)N^{c_T(j)}\qquad (T\in [10],\,j\in [10]).
\end{equation}
With $\alpha\in\mathbb{C}^{10}$ the column vector with entries $\alpha_j$, the condition that $(\operatorname{red}_{\mathcal{CR}}\circ T)(\xi)=0$ for all ten maps $T\in [10]$ is equivalent to solving:
$$M(\alpha)=0.$$
The software system \emph{Maple} was used to build $M$, and $M$ is sufficiently sparse for the determinant to be calculated explicitly:
\begin{equation}\label{eq:det-M5}
|\det M|=(N-4)(N^2-3N+1)^2.
\end{equation}
\begin{corollary}\label{maxm5}
For all $N\geq 5$, $S_N< S_N^+$ is maximal at moment level five.
\end{corollary}
For $S_5<S_5^+$, this is independent of Theorem \ref{bane2} (1). Of course, by Theorem \ref{fourmax}, $S_4<S_4^+$ is maximal at moment level five; here is an independent verification:
\begin{proposition}
  $S_4<S_4^+$ is maximal at moment level five.
\end{proposition}
\begin{proof}
At $N=4$, Maple computes $\operatorname{rank}(M)=9$
for the matrix defined in \eqref{eq:M5}. It gives $\ker M=\mathbb{C}\eta$, with
$$\eta=\sum_{\underset{|p|=3}{p\in\mathcal{CR}(5)}}\xi_p-2\sum_{\underset{|p|=2}{p\in\mathcal{CR}(5)}}\xi_p.$$
Let $\xi=\sum_{p\in\mathcal{P}(5)}\alpha_p\xi_p$. Then $\xi_0=\operatorname{red}_{\mathcal{CR}}(\xi)$ is an element of $\langle \mathcal{NC},\xi\rangle$ given by $\xi_0=\sum_{p\in\mathcal{CR}(5)}\alpha_p\xi_p$. If $\xi_0$ is outside $\ker M$, then some $(\operatorname{red}_{\mathcal{CR}}\circ T)(\xi_0)$ gives the basic crossing and $\langle \mathcal{NC},\xi\rangle=\mathcal{P}$.

\bigskip\noindent

If $\xi_0$ is inside $\ker M$, then $\xi_0=\lambda \eta$ for some $\lambda\in\mathbb{C}$, and so $\langle \mathcal{NC},\xi\rangle=\langle \mathcal{NC},\eta\rangle$. Recall Corollary \ref{bad}: at $N=4$,
\begin{align*}
  \xi_{\idpart^{\otimes 5}} & =\sum_{p\in\mathcal{P}(5)\backslash\{\idpart^{\otimes 5}\}}c_p\xi_p.\end{align*}
  The following can be verified by direct computation in \emph{Maple}; it should also be possible to derive it via M\"{o}bius inversion on the partition lattice (see \cite{nis}, Lecture 10, Ex. 10.33):
  $$c_p=(-1)^{|p|}\prod_{B\in p}(|B|-1)!$$
  Therefore
\begin{align*} \xi_{\idpart^{\otimes 5}} & =-\sum_{\underset{|p|=3}{p\in\mathcal{CR}(5)}}\xi_p+2\sum_{\underset{|p|=2}{p\in\mathcal{CR}(5)}}\xi_p+\sum_{p\in\mathcal{NC}(5)\backslash\{\idpart^{\otimes 5}\}}c_p\xi_p
  \\ \implies \eta&=\sum_{p\in\mathcal{NC}(5)}c_p\xi_p,
\end{align*}
that is $\eta\in\mathcal{NC}$, and so $\langle \mathcal{NC},\xi\rangle=\mathcal{NC}$.
\end{proof}
\subsection{Moment Level Six}\label{ml6}
The 66 operations below leave a five-dimensional kernel.
For $N\ge6$, a strategy proposed by ChatGPT treats this
remaining space by averaging the rotation action,
giving a decomposition into a three-dimensional fixed
space and a two-dimensional kernel of the averaging
projection. A quadratic form produces a crossing from
every nonzero vector in the latter space.
For the fixed space, a rotate-square-rotate construction
produces a vector outside that space, reducing the
argument to the preceding cases.

\bigskip\noindent

There are $B_6-C_6=71$ crossing partitions of $[6]$, and in the first instance, there is a five-dimensional gap. The matrix $M$ defined below has nonzero entries equal to $N^c$ with $c\in\{0,1,2\}$.
\begin{proposition}\label{propmaps}
Let $\xi=\sum_{p\in\mathcal{CR}(6)}\alpha_p\xi_{p}$. There is a set $\mathcal{F}$ of 66 maps in $\operatorname{Op}_{\mathcal{NC}}(6,4)$ such that the linear system
$$\left\{(\operatorname{red}_{\mathcal{CR}}\circ T)(\xi)=0\colon T\in\mathcal{F}\right\}$$
is of full rank for $N\geq 5$.
\end{proposition}

\begin{proof}
Firstly, the 66 maps:
\begin{enumerate}
  \item There are 15 maps $R_{\kappa}\in\hom_{\mathcal{NC}}(6,4)$, parameterised by $\kappa\in\mathcal{P}_4([6])$.
  \item There are 30 maps $T_{\left(\BigPartition{
\Pblock .75 to 0.625:2,3
\Psingletons .75 to 0.625:1.5
\Pline (2.5,0.25) (2.5,0.625)
},x,y\right)}\in \operatorname{Op}_{\mathcal{NC}}(6,4)$, parameterised by the set of $(x,y)\in[6]^2$ such that $x\neq y$. The element $x$ is capped with a singleton, while $y$ and $y+1$ are merged, except when $y+1=x$ in which case $y$ and $y+2$ are merged, for example, $T_{\left(\BigPartition{
\Pblock .75 to 0.625:2,3
\Psingletons .75 to 0.625:1.5
\Pline (2.5,0.25) (2.5,0.625)
},3,2\right)}=T_p$ where:
$$p=
\BigPartition{
\Psingletons 1 to 0.75:1,2,3,4,5,6
\Psingletons 0.25 to 0:2,3,4,5
\Pline (1,.75) (2,0.25)
\Pline (2,0.25) (2,0)
\Pline (2,0.75) (3,0.5)
\Pline (4,0.75) (3,0.5)
\Pline (3,0.5)  (3,0)
\Pline (5,0.75)  (4,0.25)
\Pline (4,0.25) (4,0)
\Pline (6,0.75)  (5,0.25)
\Pline (5,0.25) (5,0)
}$$
Note, again, the larger $\operatorname{Op}_{\mathcal{NC}}$ comes into play: e.g. $T_{\left(\BigPartition{
\Pblock .75 to 0.625:2,3
\Psingletons .75 to 0.625:1.5
\Pline (2.5,0.25) (2.5,0.625)
},2,6\right)}$ is not an element of $\operatorname{hom}_{\mathcal{NC}}(6,4)$.
    \item There are six semi-circle cappings $S_{x,x+1}\in \operatorname{Op}_{\mathcal{NC}}(6,4)$, parameterised by $x\in [6]$.
  \item There are nine maps $T_{(\mathsf{Y}\mathsf{Y},\{t_1,t_2\})}\in\operatorname{Op}_{\mathcal{NC}}(6,4)$, parameterised by pairs of adjacent elements $\{t_1,t_2\}\subset \mathcal{P}_2([6])$ such that $t_1\cap t_2=\emptyset$. Both pairs $t_1$ and $t_2$ are merged; for example, $T_{(\mathsf{Y}\mathsf{Y},\{\{6,1\},\{2,3\}\})}=T_p$ where
      $$p=
\BigPartition{
\Psingletons 1 to 0.75:1,2,3,4,5,6
\Psingletons 0.25 to 0:2.5,3.5,4.5,5.5
\Pline (1,0.75) (5.5,0.5)
\Pline (2,0.75) (2.5,0.5)
\Pline (3,0.75) (2.5,0.5)
\Pline (4,0.75) (3.5,0.25)
\Pline (5,0.75) (4.5,0.25)
\Pline (6,0.75) (5.5,0.5)
\Pline (5.5,0.5) (5.5,0.25)
\Pline (2.5,0.5) (2.5,0)
\Pline (3.5,0.25) (3.5,0)
\Pline (4.5,0.25) (4.5,0)
\Pline (5.5,0.25) (5.5,0)
}$$
  \item There are six maps $T_{\left(\BigPartition{
\Pblock .75 to 0.625:1,2,3
\Pline (2,0.25) (2,0.625)
},x\right)}\in\operatorname{Op}_{\mathcal{NC}}(6,4)$, parameterised by $x\in [6]$. These maps merge the triple $\{x,x+1,x+2\}$.
\end{enumerate}
Noting each $T$ in this collection is of the form $T_{p_T}$ for $p_T\in\mathcal{P}(6,4)$, applying  $(\operatorname{red}_{\mathcal{CR}}\circ T)$ to $\xi$ gives:
$$\left(\sum_{p\in\mathcal{CR}(6)}\alpha_pd_T(p)N^{c_T(p)}\right)\xi_{\Labab}.$$
The coefficients $d_T(p)$ are given by:
$$d_T(p)=\begin{cases}
                1, & \mbox{if }  p_T\circ p=\Labab,\\
                0, & \mbox{otherwise}.
              \end{cases}$$
The powers of $N$ are determined by the number of closed blocks in $p_T\circ p$, with $c_T(p)\in\{0,1,2\}$ for the restriction maps $R_\kappa$, $c_T(p)\in\{0,1\}$ for the $T_{\left(\BigPartition{
\Pblock .75 to 0.625:2,3
\Psingletons .75 to 0.625:1.5
\Pline (2.5,0.25) (2.5,0.625)
},x,y\right)}$ and the semi-circle cappings, and $c_T(p)=0$ for merging maps. Define a matrix $M\in M_{66\times 71}(\mathbb{C})$:
\begin{equation}\label{eq:M6}
M(T,p)=d_T(p)N^{c_T(p)}\qquad (T\in [66],\,p\in \mathcal{CR}(6)).
\end{equation}
With $\alpha\in\mathbb{C}^{71}$ the column vector with entries $\alpha_p$, the condition that $(\operatorname{red}_{\mathcal{CR}}\circ T)(\xi)=0$ for all 66 maps is equivalent to solving:
$$M(\alpha)=0.$$
With the row and column orderings specified in the accompanying
Maple worksheet, let $A(N)$ be the $66\times66$ submatrix consisting
of the first 66 columns of $M$. An exact symbolic computation gives
\begin{equation}\label{eq:minor6}
|\det A(N)|=(N-4)^{12}(N-2)(N^2-3N+1)^{23}.
\end{equation}
Every factor is nonzero for $N\geq5$, so $A(N)$ has rank 66.
Since $M$ contains $A(N)$ and has only 66 rows,
$$
66=\operatorname{rank}A(N)
\leq\operatorname{rank}M
\leq66.
$$
Thus $M$ has full row rank for every $N\geq5$.
In particular, $\dim\ker M=71-66=5$.
\end{proof}

At $N=4$, Maple finds $\operatorname{rank}(M)=54$.  For $N=5$, by Corollary \ref{bad}, there should exist in $\ker M$ a one-dimensional space $\mathbb{C}\eta$ such that $\langle\mathcal{NC},\eta\rangle=\mathcal{NC}$, so a priori the gap is four-dimensional. For $N=4$, by Corollary \ref{int}, there should exist in $\ker M$ a 16-dimensional such space, the gap being one-dimensional.

\bigskip\noindent

For $N\leq 5$, maximality of $S_N\leq S_N^+$ gives maximality at all moment levels.  \textbf{For the remainder of this section, assume $N\geq 6$.}  The following is immediate:
\begin{proposition}\label{propobv}
If $\xi$ is outside $\ker M$, then $\langle \mathcal{NC},\xi\rangle=\mathcal{P}$.
\end{proposition}
Note an invariance of $\ker M$ under the dihedral action is expected from the dihedral action on $\hom_{\mathcal{NC}}(6,4)$. For example, for the 30 maps of the form  $T_{\left(\BigPartition{
\Pblock .75 to 0.625:2,3
\Psingletons .75 to 0.625:1.5
\Pline (2.5,0.25) (2.5,0.625)
},x,y\right)}$, 12 are in the $D_6$ orbit of $T_{\left(\BigPartition{
\Pblock .75 to 0.625:2,3
\Psingletons .75 to 0.625:1.5
\Pline (2.5,0.25) (2.5,0.625)
},1,2\right)}$, 12 are in the $D_6$ orbit of  $T_{\left(\BigPartition{
\Pblock .75 to 0.625:2,3
\Psingletons .75 to 0.625:1.5
\Pline (2.5,0.25) (2.5,0.625)
},1,3\right)}$, and six are in the $D_6$ orbit of  $T_{\left(\BigPartition{
\Pblock .75 to 0.625:2,3
\Psingletons .75 to 0.625:1.5
\Pline (2.5,0.25) (2.5,0.625)
},2,1\right)}$.

\begin{proposition}
The following properties of $\ker M$ hold:
\begin{enumerate}
\item $\ker M$ has a basis $\{v_1,v_2,v_3,v_4,v_5\}$ whose coefficients are real-valued polynomials in $N$, of degree at most two,
 \item $\ker M$ is invariant under the dihedral action $D_6\curvearrowright \hom_{\mathcal{P}}(0,6)$,

  \item $T_{r^3}$ is the identity on $\ker M$.
\end{enumerate}

\end{proposition}
\begin{proof}
These are verified with the use of \emph{Maple}:
\begin{enumerate}
  \item Let $V(N)$ be the explicit $71\times5$ polynomial matrix recorded
in the accompanying Maple worksheet, and denote its columns by
$v_1,\ldots,v_5$. Its entries are real polynomials in $N$ of degree
at most two. Exact symbolic computation verifies
$$
M(N)V(N)=0.
$$
Thus each $v_i$ belongs to $\ker M$. The submatrix formed by the
last five rows of $V(N)$ is
$$
\operatorname{diag}(N-2,1,1,N-2,1),
$$
whose determinant is $(N-2)^2\ne0$ for $N\geq5$.
Therefore the five columns are linearly independent.
By the rank calculation following \eqref{eq:minor6},
$\dim\ker M=5$.
Therefore the five linearly independent columns
of $V(N)$ form a basis of $\ker M$.
  \item For each of the basis vectors:
  $$M(T_{r}(v_i))=0 \text{, \quad and }M(T_s(v_i))=0.$$
  \item For each of the basis vectors, $T_{r^3}(v_i)=v_i$.
\end{enumerate}

\end{proof}
Note  that as $\ker M$ is a well-parameterised five-dimensional subspace, with $\xi\in\ker M$ expressed as
$$\xi=\sum_{i=1}^5 t_iv_i\qquad(t_i\in\mathbb{C}),$$
one can build in $\langle \mathcal{NC},\xi\rangle$ a map $\xi\mapsto Q(t)\xi_{\Labab}$, for a quadratic form $Q:\ker M\to \mathbb{C}$:
$$Q(t)=t^*Bt\qquad (t=(t_1,t_2,t_3,t_4,t_5)).$$
Recalling that the entries of the basis $\{v_i\}_{i=1}^5$ are real, one can build towards such a form using the conjugate-linear reflection:
$$\overline{T_s}(\xi)=\sum_{j=1}^{5}\overline{t_j}T_s(v_j).$$

Next, for $\xi\in\ker M$, consider $\xi\otimes \overline{T_s}(\xi)$. Now apply appropriate $T_p\in \operatorname{Op}_{\mathcal{NC}}(12,4)$,
and with $V=[v_1|v_2|v_3|v_4|v_5]$, define
$$B_{ji}=\sum_{q,r\in\mathcal{CR}(6)}V(q,i)V(r,j)d_{T_p}(q,r)N^{c_{T_p}(q,r)}.$$
Here
$$d_{T_p}(q,r)=\begin{cases}
              1, & \mbox{if } p\circ (q\otimes s(r))=\Labab, \\
              0, & \mbox{otherwise},
            \end{cases}$$
            and $c_{T_p}(q,r)$ counts the number of deleted blocks in $q\otimes s(r)\mapsto p\circ (q\otimes s(r))$. This gives a quadratic form $Q(t)=t^*Bt$ via
$$(\operatorname{red}_{\mathcal{CR}}\circ T_p)(\xi\otimes \overline{T_s}(\xi))=Q(t)\xi_{\Labab}.$$

\bigskip

\textbf{The arguments in the rest of this section are by ChatGPT}.
Maple also shows that with respect to the basis $\{v_1,\dots,v_5\}$, $R:={T_r}_{\left|\ker M\right.}$ is given by the matrix
\begin{equation}\label{eq:rotation-matrix}
R=\begin{pmatrix}
      0 & 0 & 0 & 1 & 0 \\
      0 & 0 & 0 & 0 & 1 \\
      0 & -1 & 1 & 0 & 1 \\
      1 & -1 & 0 & 0 & 1 \\
      1 & 0 & 0 & -1 & 1
    \end{pmatrix}.
\end{equation}
The $j$-th column records the coordinates of $T_r(v_j)$:
\[
T_r(v_j)=R(v_j)=\sum_{i=1}^{5}R_{ij}v_i.
\]
Solving $Rt=t$ for the matrix in
\eqref{eq:rotation-matrix} gives the following
three-dimensional fixed space:
\begin{equation}\label{eq:fixed-space}
\operatorname{Fix}(R)=\left\{bv_1+cv_2+av_3+bv_4+cv_5\colon a,b,c\in\mathbb{C}\right\}.
\end{equation}
Define
\begin{equation}\label{eq:projection}
P=\frac{I+R+R^2}{3}.
\end{equation}
The operation
$$
\frac{I+T_r+T_r^2}{3}
$$
restricts to $P$ on $\ker M$. Note as each of $I$, $T_r$, and $T_{r^2}$ are category operations, it follows from the vector space structure of $\hom_{\mathcal{C}}(0,6)$ that  $P\xi\in\langle\mathcal{NC},\xi\rangle$
whenever $\xi\in\ker M$.

\begin{proposition}[ChatGPT]
For $N\geq 6$,
$$\ker M=\operatorname{Fix}(R)\oplus \ker P.$$
\end{proposition}
\begin{proof}
Decompose
$$\ker M=\operatorname{im}P\oplus \ker P.$$
Let $x\in \operatorname{im}P$ so that $x=Pt$. Via $R^3=I$, $P^2=P$:
$$Px=P^2t\implies Px=Pt\implies x=Px.$$
Also $RP=P$ so that
$$Rx=RPx=Px=x,$$
that is $x\in \operatorname{Fix}(R)$. On the other hand suppose $y\in \operatorname{Fix}(R)$
$$Ry=y\implies R^2y=y\implies Py=\frac{y+Ry+R^2y}{3}=y,$$
that is $y\in \operatorname{im} P$.
\end{proof}

The quadratic forms considered above may not be positive definite on $\ker M$ but they can be on the two dimensional $\ker P$.
\begin{proposition}[ChatGPT]\label{propquad}
Let nonzero $\xi\in \ker M\cap \ker P$. Then $\langle\mathcal{NC},\xi\rangle=\mathcal{P}$.
\end{proposition}
\begin{proof}
Let
$$
\xi=\sum_{i=1}^{5}t_iv_i\in\ker P.
$$
Then $\xi=(I-P)\xi$.
Put $z_1=t_1-t_4$ and $z_2=t_2-t_5$. With the following partition:

$$p=\BigPartition{
\Psingletons 1 to 0:1,2,11,12
\Pblock 1 to 0.2:3,10
\Pblock 1 to 0.4:4,9
\Pblock 1 to 0.6:5,8
\Pblock 1 to 0.8:6,7
}.$$
For this partition, exact symbolic computation in Maple gives, where
$$(\operatorname{red}_{\mathcal{CR}}\circ T_p)(\xi\otimes \overline{T_s}(\xi))=Q(t)\xi_{\Labab},$$
\begin{subequations}\label{eq:positive-form}
\begin{equation}\label{eq:positive-form-expression}
Q(t)=A_N|z_1+z_2|^2+B_N|z_1-z_2|^2,
\end{equation}
where
\begin{equation}\label{eq:positive-form-coefficients}
\begin{aligned}
A_N&=\frac{(N-4)(2N^2-9N+8)}{18}
\\ B_N&=\frac{(N-4)(2N^2-7N+4)}{2}.
\end{aligned}
\end{equation}
\end{subequations}
For $N\geq 6$, $A_N,\,B_N>0$.
If $Q(t)=0$, then $t_1=t_4$ and $t_2=t_5$, and $\xi\in \operatorname{Fix}(R)= \operatorname{im} P$, so $\xi=0$. Therefore $Q(t)\neq 0$, and $\xi_{\Labab}\in\langle \mathcal{NC},\xi\rangle$.
\end{proof}
\bigskip\noindent
Let $E_{ijk,\ell mn}$ denote the map $e_{\ell}\otimes e_m\otimes e_n$ to $e_i\otimes e_j\otimes e_k$ (and every other basis tensor to zero). Using the rotation conventions of \eqref{eq:rotations}, applying the anticlockwise right-side rotation $\operatorname{rot}_R^3\in \operatorname{Op}_{\mathcal{C}}(0,6;3,3)$ to $\xi\in \hom_{\mathcal{C}}(0,6)$ gives $A_\xi\in \hom_{\mathcal{C}}(3,3)$ given by:
\begin{equation}\label{eq:rotation-A}
\sum_{i,j,k,\ell,m,n=1}^N(A_\xi)_{ijk,\ell mn}E_{ijk,\ell mn}=\sum_{i,j,k,\ell,m,n=1}^N\xi_{ijknm\ell}E_{ijk,\ell mn}.
\end{equation}
Composed with itself, $A_\xi^2\in \langle\mathcal{NC},\xi \rangle$:
\begin{equation}\label{eq:square}
A_\xi^2=\sum_{i,j,k,\ell,m,n=1}^N\left(\sum_{x,y,z=1}^N\xi_{ijkzyx}\xi_{xyznm\ell}\right)E_{ijk,\ell mn}.
\end{equation}
Now apply the clockwise right-side rotation $\operatorname{rot}_R^{-3}\in\operatorname{Op}_{\mathcal{C}}(3,3;0,6)$ to $A^2_\xi$ to find:
\begin{subequations}\label{eq:theta}
\begin{equation}\label{eq:theta0}
\theta_0:=\sum_{i,j,k,\ell,m,n=1}^N\left(\sum_{x,y,z=1}^N\xi_{ijkzyx}\xi_{xyznm\ell}\right)e_i\otimes e_j\otimes e_k\otimes e_n\otimes e_m\otimes e_\ell.
\end{equation}
In the region $N\geq 6$, $\operatorname{red}_{\text{CR}}$ is the well-defined $\pi_{\text{CR}}$
\begin{equation}\label{eq:theta-crossing}
\theta:=\pi_{\text{CR}}(\theta_0).
\end{equation}
\end{subequations}
For an illustration, consider the partitions
$$
p=\LPartition{}{0.8:1,5;1.0:2,6;0.4:3,4},
\qquad
q=\LPartition{}{1.0:1,6;0.7:2,5;0.4:3,4}.
$$
By definition
$$
\xi_p=\sum_{x,y,z=1}^{N}
e_x\otimes e_y\otimes e_z\otimes e_z\otimes e_x\otimes e_y.
$$
The anticlockwise right-side rotation gives
$$
A_{\xi_p}(e_x\otimes e_y\otimes e_z)
=e_y\otimes e_x\otimes e_z.
$$
Thus $A_{\xi_p}$ exchanges the first two tensor factors,
so $A_{\xi_p}^{\,2}=I$. Rotating back gives
$$
\operatorname{rot}_{R}^{-3}(A_{\xi_p}^{\,2})
=\sum_{x,y,z=1}^{N}
e_x\otimes e_y\otimes e_z\otimes e_z\otimes e_y\otimes e_x
=\xi_q.
$$
Here $q$ is non-crossing, so the subsequent crossing
projection gives zero.
\begin{proposition}[ChatGPT]\label{propoutside}
Let nonzero $\xi\in\ker M\cap\operatorname{Fix}(R)$.
Then there exists nonzero
$$
\theta\in
\langle\mathcal{NC},\xi\rangle
$$
such that $\theta\in\operatorname{span}(\xi_p:p\in\mathcal{CR}(6))$ and $\theta\notin\operatorname{Fix}(R)$.
\end{proposition}
\begin{proof}
By \eqref{eq:fixed-space}, write
$\xi=bv_1+cv_2+av_3+bv_4+cv_5$.
Applying the construction in \eqref{eq:theta}
gives $\theta\in\langle\mathcal{NC},\xi\rangle$
with coefficients homogeneous quadratic
polynomials in $a,b,c$:
$$\theta=\sum_{p\in\mathcal{CR}(6)}\beta_p(a,b,c)\xi_p.$$
Suppose, for a contradiction, that
$\theta\in\operatorname{Fix}(R)$. Since $\operatorname{Fix}(R)\subseteq\ker M$, all 66
crossing coefficients obtained from the operations in
Proposition \ref{propmaps} vanish.
Apply $\operatorname{red}_{\mathcal{CR}}\circ S_{3,4}$
and $\operatorname{red}_{\mathcal{CR}}\circ
T_{(\mathsf{Y}\mathsf{Y},\{\{6,1\},\{3,4\}\})}$
to $\theta$.
Exact symbolic computation in Maple gives\footnote{Of the 66 maps constructed in
Proposition \ref{propmaps}, only three have potentially
nonzero crossing components on this family: the two
displayed maps and $S_{6,1}$, which gives the same
relation as $S_{3,4}$.}
\begin{equation}\label{eq:two-outputs}
\begin{aligned}
\operatorname{red}_{\mathcal{CR}}(S_{3,4}\theta)
   &=\left(Ng_N(a,b,c)-(N-2)h_N(a,b,c)\right)\xi_{\Labab},\\
 \operatorname{red}_{\mathcal{CR}}(T_{(\mathsf{Y}\mathsf{Y},\{\{6,1\},\{3,4\}\})}\theta)
   &=\left((N^2-4N+2)g_N(a,b,c)+(3-N)h_N(a,b,c)\right)\xi_{\Labab},
\end{aligned}
\end{equation}
where
\begin{equation}\label{eq:gh}
\begin{aligned}
g_N(a,b,c)={}&a^2-2(N-2)^2ab+4ac+(N-2)^2(N+2)b^2\\
&{}-4(N-2)(N-3)bc+(N^2-8N+17)c^2,\\
h_N(a,b,c)={}&(2N-5)a^2-2(N-2)(N^2-4N+6)ab+2(N+1)ac\\
&{}+(N-2)^2(5N-4)b^2-4(N-2)(2N-5)bc\\
&{}+(N^2-8N+19)c^2.
\end{aligned}
\end{equation}
Since $\theta\in\ker M$, both left-hand sides
in \eqref{eq:two-outputs} vanish. Therefore
$$
Ng_N-(N-2)h_N=0,\qquad
(N^2-4N+2)g_N+(3-N)h_N=0.
$$
The determinant of this system is
$$
N^3-7N^2+13N-4=(N-4)(N^2-3N+1),
$$
which is nonzero for $N\ge6$.
Hence $g_N=h_N=0$. Write $\theta=\sum_{i=1}^5t_iv_i$.

Since $\theta\in\operatorname{Fix}(R)$,
$t_1=t_4$ and $t_2=t_5$.
For the two crossing partitions
$$
p=\{\{1,5\},\{2\},\{3,6\},\{4\}\},\qquad
q=\{\{1\},\{2,5\},\{3\},\{4,6\}\},
$$
the normalization of the basis gives
$$
\beta_p=(N-2)t_1,\qquad \beta_q=(N-2)t_4.
$$
Thus, for $\theta\in \operatorname{Fix}(R)$, $\beta_p-\beta_q=0$. Maple calculates this difference as
\begin{equation}\label{eq:kN}
\begin{aligned}
k_N(a,b,c):={}&\beta_p(a,b,c)-\beta_q(a,b,c)\\
={}&2a^2-2(N-2)(N-4)ab-2ac\\
&{}+(N-2)^2(N^2-8N+10)b^2+4(N-2)^2bc\\
&{}-(N-4)(2N-5)c^2.
\end{aligned}
\end{equation}
Therefore  three equations must be satisfied:
\begin{equation}\label{eq:three-equations}
g_N(a,b,c)=h_N(a,b,c)=k_N(a,b,c)=0.
\end{equation}
It will be shown that the system
\eqref{eq:three-equations} has no nonzero solution
when $N\geq6$. Suppressing the arguments and subscript $N$ on $g,h,k$,  direct expansion
gives the factorization
\begin{equation}\label{eq:factorisation}
h+k-(N-1)g
=(N-2)\bigl(a-2(N-2)b+(N-4)c\bigr)
       \bigl(a+2(N-2)b-(N-2)c\bigr).
\end{equation}
Since $g=h=k=0$ and $N-2\ne0$, at least one of the two linear factors
must vanish.

\smallskip
\noindent\emph{Case 1: $a=2(N-2)b-(N-4)c$.}
Substituting this value of $a$ into
\eqref{eq:gh} and \eqref{eq:kN} gives:
\begin{equation}\label{eq:case1}
\begin{aligned}
g&=(14-3N)(N-2)^2b^2
   +2(N^2-10N+26)(N-2)bc
   +(2N^2-20N+49)c^2,\\
g-k&=(N-5)\bigl((2N-9)c^2-(N-4)(N-2)^2b^2\bigr).
\end{aligned}
\end{equation}
If $c=0$, the first equation gives $b=0$, since
$(14-3N)(N-2)^2\ne0$.
Then $a=b=c=0$. Therefore assume $c\ne0$. Write
$$
D=N^2-10N+26,\qquad B=(N-4)(N-5)(N-8)-1.
$$
The second identity in \eqref{eq:case1},
together with $g=k=0$ and $N-5\ne0$, gives:
\begin{equation}\label{eq:case1-b2}
(N-4)(N-2)^2b^2=(2N-9)c^2.
\end{equation}
Multiply $g$ by $N-4$ so that the coefficient of $b^2$ is:
$$-(N - 4)(3N - 14)(N - 2)^2.$$
Using \eqref{eq:case1-b2}, replace
$(N-4)(N-2)^2b^2$ in $(N-4)g=0$
with $(2N-9)c^2$.
The result is:
$$
2c\left((N-4)(N-2)Db+Bc\right)=0,
$$
so $(N-4)(N-2)Db=-Bc$. Squaring this equality and using
\eqref{eq:case1-b2}, then dividing both sides
by $c^2$, gives:
\begin{equation}\label{eq:case1-integer}
B^2=(N-4)(2N-9)D^2.
\end{equation}
In \eqref{eq:case1-integer}, the right-hand side is divisible by
$N-4$, whereas
$$
B\equiv-1\mod{(N-4)}
\quad\Longrightarrow\quad
B^2\equiv1\mod{(N-4)}.
$$
It would follow that $N-4$ divides $1$, which is impossible for $N\geq6$.

\smallskip
\noindent\emph{Case 2: $a=-2(N-2)b+(N-2)c$.}
Put $v=(N-2)b-c$ and $E=N^2-4N+2$. Substituting this value of $a$ into
\eqref{eq:gh} and \eqref{eq:kN} gives:
\begin{equation}\label{eq:case2}
\begin{aligned}
g&=(5N-2)v^2-2N(N-5)vc-3(N-5)c^2,\\
k&=Ev^2-(N-5)(N-2)c^2.
\end{aligned}
\end{equation}
If $c=0$, the first equation gives $v=0$, and hence $b=0$ and $a=0$.
Therefore assume $c\ne0$. Since $g=k=0$, forming $Eg-(5N-2)k$
using \eqref{eq:case2} eliminates $v^2$
and gives:
\begin{equation}\label{eq:case2-linear}
\begin{aligned}
Eg-(5N-2)k&
 =2(N-5)c\left((N^2-1)c-NEv\right)
 \\ \implies NEv&=(N^2-1)c.
\end{aligned}
\end{equation}
Squaring the last equality in
\eqref{eq:case2-linear} and using,
from $k=0$ in \eqref{eq:case2},
$Ev^2=(N-5)(N-2)c^2$ gives,
after dividing both sides by $c^2$,
\begin{equation}\label{eq:case2-integer}
N^2E(N-5)(N-2)=(N^2-1)^2.
\end{equation}
In \eqref{eq:case2-integer}, the left-hand side
is divisible by $N$, while the right-hand side
is congruent to $1$ modulo $N$.
This is impossible for $N\geq6$.

\bigskip

Both cases therefore force $a=b=c=0$, contradicting $\xi\ne0$. Therefore $\theta\notin\operatorname{Fix}(R)$.

\end{proof}
ChatGPT's handling of the five-dimensional kernel allows the conclusion:
\begin{theorem}\label{mom6}
  $S_N\leq S_N^+$ is maximal at moment order six.
\end{theorem}
\begin{proof}
Let nonzero $\xi= \sum_{p\in\mathcal{CR}(6)}\alpha_p\xi_p$. If $\xi\not\in\ker M$, by Proposition \ref{propobv}, $\langle\mathcal{NC},\xi\rangle=\mathcal{P}$.

\bigskip\noindent

Suppose now $\xi\in\ker M$, $\xi=\sum_{i=1}^{5}t_iv_i$. Suppose $\xi\not\in \operatorname{Fix}(R)$, then nonzero $(I-P)\xi\in \ker P$, and therefore by Proposition \ref{propquad}, $\langle \mathcal{NC},\xi\rangle=\mathcal{P}$. Suppose $\xi \in\operatorname{Fix}(R)$: then by Proposition \ref{propoutside}, there exists a $\theta\in\langle\mathcal{NC},\xi\rangle$ outside $\operatorname{Fix}(R)$. If $\theta\notin\ker M$, Proposition \ref{propobv} gives
$\langle\mathcal{NC},\theta\rangle=\mathcal P$.
Otherwise, nonzero
$$
(I-P)\theta\in\ker M\cap\ker P,$$
and Proposition \ref{propquad} gives the same conclusion.
Since $\theta\in\langle\mathcal{NC},\xi\rangle$,
$$
\mathcal P
=\langle\mathcal{NC},\theta\rangle
\subseteq\langle\mathcal{NC},\xi\rangle
\subseteq\mathcal P.
$$
Thus $\langle\mathcal{NC},\xi\rangle=\mathcal P$.
\end{proof}

Whether the methods of Theorem \ref{thez3} and
Theorem \ref{mom6} extend to easiness level four and
moment level seven, respectively, remains open.
It is natural to consider the two restrictions together.
Say that $S_N\leq S_N^+$ is maximal at
\emph{easiness-moment level $(m,k)$} if, for every
$\xi\in\hom_{\mathcal P}(0,k)$ expressible as a linear
combination of at most $m$ partition vectors,
\[
\langle\mathcal{NC},\xi\rangle=\mathcal{NC}
\quad\text{or}\quad
\langle\mathcal{NC},\xi\rangle=\mathcal P.
\]
A natural next case is the following.

\begin{question}
Is $S_N\leq S_N^+$ maximal at easiness--moment level $(4,7)$?
\end{question}

The next unresolved value of the parameter $N$
also gives an immediate question.

\begin{question}
Is $S_6<S_6^+$ maximal?
\end{question}

A proof at the first non-classical value of $N$ using the
methods developed here would also be of interest.

\begin{question}
Is there a proof of the maximality of $S_4<S_4^+$
using Tannaka--Krein duality?
\end{question}

\begingroup
\renewcommand{\addcontentsline}[3]{}
\section*{Acknowledgements}
The author is grateful to Amaury Freslon and Roland Speicher for permission to reproduce their argument on maximality at moment level five as Appendix \ref{app:freslon}. The author also thanks Daniel Gromada for making the \texttt{partmac} package available; this package was used to typeset the partition diagrams.
\par
\endgroup
\section*{Declaration of generative AI and AI-assisted technologies in the manuscript preparation process}
Claude (Anthropic) assisted with editing an earlier version of this paper and with the development of the author's original Maple code. ChatGPT (OpenAI) proposed the strategy for treating the remaining vectors $\xi\in\ker M$ in the proof of maximality at moment level six. This included the rotation-average decomposition, the rotate--square--rotate construction, and the development and simplification of the resulting polynomial argument. The idea of constructing a quadratic form using the conjugate-linear reflection originated with the author. The AI-assisted mathematical arguments were subsequently examined, edited and revised by the author for Section~\ref{ml6}.

ChatGPT also assisted in consolidating, extending and debugging the Maple calculations, producing the two self-contained worksheets and corresponding source files, and preparing pseudocode, computational checks and documentation for the accompanying GitHub repository (\url{https://github.com/JPMcCarthy101/maximality-levels-maple}). It suggested symbolic computations and arguments valid uniformly in the stated ranges of $N$, rather than only generically. ChatGPT drafted Appendix~\ref{app:maple}, which describes the computational implementation and its reproducibility, and assisted with editing the manuscript, including this declaration. The author reviewed and revised the AI-assisted material, checked the mathematical arguments, and executed both final worksheets in Maple, retaining the outputs in the repository. The author accepts full responsibility for the accuracy, validity and presentation of the work.

\appendix
\section{Freslon--Speicher's Maximality at Moment Level Five}\label{app:freslon}
\emph{The following argument of Freslon--Speicher appears in an unpublished note \cite{fre}. Amaury Freslon and Roland Speicher kindly permit the argument to be reproduced here.}

\bigskip

Suppose that $\mathbb{G}$ is exotic. Let $\pi:\mathcal{O}(S_N^+)\to \mathcal{O}(\mathbb{G})$, and $\varphi:=h_{\mathcal{O}(\mathbb{G})}\circ \pi$ the associated Haar idempotent. For any homogeneous orthogonal quantum group, by invariance under permutation of the indices (Section \ref{symm}, (\ref{inv})), the value of
$$h(u_{i_1j_1}\cdots u_{i_kj_k})=h\left(u^{\otimes k}_{\mathbf{i}\mathbf{j}}\right),$$
is determined by the kernels of the indices $\mathbf{i}=(i_1,\dots,i_k)$ and $\mathbf{j}=(j_1,\dots,j_k)$. These are the level sets $\ker \mathbf{i},\,\ker \mathbf{j}\in\mathcal{P}(k)$ of $\mathbf{i},\,\mathbf{j}$ viewed as functions $[k]\to [N]$:
$$\ker \mathbf{i}=\bigsqcup_{a\in \operatorname{Im}(\mathbf{i})}\mathbf{i}^{-1}(a).$$
Let $\mathcal{P}_{N}(k)$ denote the set of partitions of $[k]$ with at most $N$ blocks.
Therefore pairs $(\mathbf{i},\mathbf{j}),\,(\mathbf{i}',\mathbf{j}')\in [N]^k\times[N]^k$ with equal kernels, $\ker \mathbf{i}=\ker \mathbf{i}'$, $\ker\mathbf{j}=\ker \mathbf{j}'$, give equal moments:
$$h\left(u^{\otimes k}_{\mathbf{i}'\mathbf{j}'}\right)=h\left(u^{\otimes k}_{\mathbf{i}\mathbf{j}}\right).$$

Therefore the Haar state $h$ (and similar for $\varphi$) is, via abuse of notation, determined by:
$$h:\bigsqcup_{k\geq 1}(\mathcal{P}_N(k)\times \mathcal{P}_{N}(k))\to \mathbb{C},\qquad (p,q)\mapsto h(p,q),$$
where
$$h(p,q)=h(u^{\otimes k}_{\mathbf{i}\mathbf{j}}),\text{ such that }(\ker \mathbf{i},\ker \mathbf{j})=(p,q).$$
The $\hom(0,k)$ spaces determine the values of the Haar state on monomials of degree $k$:

\begin{theorem}\label{Haardet}
Let $\mathcal{C}$ be the Tannaka--Krein dual of an exotic $\mathbb{G}\lneq S_N^+$. Then
$$\left[\text{for all }p,\,q\in \mathcal{P}_N(k),\quad \varphi(p,q)=h(p,q)\right]\iff\dim\hom_{\mathcal{C}}(0,k)=\dim\hom_{\mathcal{NC}}(0,k).$$
\end{theorem}
\begin{proof}
This follows from the fact that the $k$-th moment of the character of the fundamental representation is equal to $\dim\hom(0,k)$, and the Weingarten formula (\cite{ba1}, Th. 3.20).
\end{proof}

Suppose $p,\,q\in\mathcal{P}(k)$. If $p$ (or $q$) has a pair of related neighbours, then $h(p,q)$ is equal to the Haar state at a monomial of degree strictly less than $k$ (and, in particular, can be zero due to orthogonality along rows or columns). Therefore, if $S_N\leq S_N^+$ is known to be maximal for $k<k_0$, then if $p$ (or $q$) in $\mathcal{P}(k_0)$ has a related neighbour:
$$\varphi(p,q)=h(p,q).$$
Freslon--Speicher used this and the below to argue that $S_N\leq S_N^+$ is maximal at moment level five.
\begin{lemma}[Freslon--Speicher's Singleton Reduction Lemma]\label{srl}
If $p$ or $q\in\mathcal{P}_N(k)$ contains a singleton, then a moment $\varphi(p,q)$ is a linear combination of moments of order less than $k$, and moments of order at most $k$ with no singletons in the partitions.
\end{lemma}
\begin{proof}
Choose $\mathbf{i},\mathbf{j}$ such that:
$$\varphi(p,q)=\varphi(u_{i_1j_1}u_{i_2j_2}\cdots u_{i_kj_k}).$$
Let $\#_1(p)$ be the number of singletons in $p$. Without loss of generality, assume that $\#_1(p)\geq 1$, so that, after rotation, $p=\{1\}\sqcup p'$ for some $p'\in\mathcal{P}(k-1)$. Let $q'=\ker (j_2,j_3,\dots,j_k)$ and consider:
$$\varphi(p',q')=\varphi(u_{i_2j_2}u_{i_3j_3}\cdots u_{i_kj_k}).$$
Using $\sum_{\ell}u_{\ell j_1}=1$,
$$\varphi(p',q')=\sum_{\ell=1}^{N}\varphi(u_{\ell j_1}u_{i_2j_2}u_{i_3j_3}\cdots u_{i_kj_k})=\sum_{\ell=1}^{N}\varphi(p_\ell,q),$$
where $p_{\ell}=\ker(\ell,i_2,i_3,\dots,i_k)$. Note
$$
\varphi(p_\ell,q)=
\begin{cases}
  \varphi(p,q), & \text{if }\{1\}\in p_\ell,\\
  \varphi(p_\ell,q), & \text{otherwise},
\end{cases}
$$
where in the second case
$\#_1(p_\ell)<\#_1(p)$.
That is, where $\alpha$ is the number of values
$\ell\in[N]$ for which $\{1\}$ is a singleton block
of $p_\ell$:
$$
\varphi(p,q)=\frac{1}{\alpha}
\left(
  \varphi(p',q')
  -\sum_{\ell:\{1\}\notin p_\ell}
       \varphi(p_\ell,q)
\right).
$$
Here $\alpha=N-|p|+1\geq1$, and
$p',q'\in\mathcal{P}_N(k-1)$.
Iterate this procedure on the remaining moments
of order $k$, applying it to either partition
whenever it contains a singleton. Each such step
strictly decreases the total number of singletons
in the two partitions. Thus the procedure expresses
$\varphi(p,q)$ as a linear combination of moments
of order less than $k$ and moments of order $k$
whose two partitions have no singletons.
\end{proof}
\begin{lemma}
Every $p\in\mathcal{P}(5)$ contains a singleton, or, up to rotation, related neighbours.
\end{lemma}
As a corollary, using the explicit formulae in \cite{mcc}, any fifth moment can be calculated easily, for example:
$$h(u_{11}u_{22}u_{11}u_{23}u_{32})=\dfrac{N-3}{N(N-1)(N-2)(N^2-3N+1)}.$$

\bigskip\noindent

Combining Theorem \ref{Haardet} together with Freslon--Speicher's Singleton Reduction Lemma gives maximality at moment level five:
\begin{theorem}[Freslon--Speicher]\label{max5}
The inclusion  $S_N\leq S_N^+$ is maximal at moment level five.
\end{theorem}
\begin{proof}
The previous two lemmas imply that the calculation of any moment of order five reduces to a linear combination:
\begin{align*}
\varphi(p,q)
&= \sum_{\substack{ p_a,q_b \in \mathcal{P}(\ell)\\\ell \le 4}} \alpha_{a,b} \varphi(p_a,q_b)
 + \sum_{\substack{p_c,q_d \in \mathcal{P}(5) \\ \#_1(p_c)=\#_1(p_d)=0}} \beta_{c,d} \varphi(p_c,q_d)& \\
&= \sum_{\substack{ p_a,q_b \in \mathcal{P}(\ell)\\ \ell \le 4}} \alpha_{a,b} h(p_a,q_b)
 + \sum_{\substack{p_c,q_d \in \mathcal{P}(5) \\ \#_1(p_c)=\#_1(p_d)=0}} \beta_{c,d} h(p_c,q_d)&=h(p,q)
\end{align*}

by, respectively, maximality of $S_N\leq S_N^+$ at moment level four, and the previous lemma giving that $p_c$ or $q_d$ have related neighbours.
\end{proof}

\section{Maple computations and reproducibility}\label{app:maple}

The exact symbolic computations used in Section~\ref{Moment} are available in
the accompanying GitHub repository:
\begin{center}
\url{https://github.com/JPMcCarthy101/maximality-levels-maple}.
\end{center}
The repository contains two self-contained Maple worksheets, together with
equivalent plain-text source files, pseudocode, a correspondence between the
equations in this paper and the relevant computations, and saved executed
worksheets. No earlier exploratory notebooks are needed to reproduce the
calculations described here.

\smallskip\noindent
\textbf{Partition encoding.}
A set partition of $[k]$ is represented by a restricted-growth word
$(w_1,\ldots,w_k)$: the first label is zero, and a new block receives the
next unused label. For instance, $(0,1,0,1)$ represents the basic crossing
partition of $[4]$. The code enumerates partitions in a specified order and
filters for crossing partitions, giving $10$ at moment five and $71$ at
moment six. Each operation is encoded by its ordered output groups and
the groups capped off. A gluing routine determines the resulting partition
and the number $c$ of closed components. Its contribution to the corresponding
matrix entry is $N^c$ if the output is the basic crossing, and zero
otherwise. Thus the matrices are constructed from the partition operations,
rather than entered as lists of precomputed coefficients.

\smallskip\noindent
\textbf{Moment five.}
The worksheet \texttt{notebooks/MomentLevel5.mw}, or equivalently
\texttt{src/MomentLevel5.mpl}, constructs the ten operations in
Section~\ref{Moment} and the matrix~\eqref{eq:M5}, then computes the
determinant~\eqref{eq:det-M5}. At $N=4$ it checks the rank-nine
specialisation, the stated generator of the kernel and the identity
expressing the discrete-partition vector in terms of the other partition
vectors. The last check is performed on all $4^5$ index tuples.

\smallskip\noindent
\textbf{Moment six.}
The worksheet \texttt{notebooks/MomentLevel6.mw}, with matching source
\texttt{src/MomentLevel6.mpl}, enumerates the $66$ operations of
Proposition~\ref{propmaps} in five families of sizes
$15+30+6+9+6$, constructs the $66\times71$ matrix~\eqref{eq:M6},
and computes the determinant of its first $66$ columns as in
\eqref{eq:minor6}. It verifies the polynomial kernel basis $V(N)$ by
checking $M(N)V(N)=0$ and its normalising five-row submatrix. The code
also checks the rotation and reflection actions, the matrix
\eqref{eq:rotation-matrix} and the projection~\eqref{eq:projection}.
The gluing computation for the conjugate-linear reflection gives the
quadratic form~\eqref{eq:positive-form}; the
rotate--square--rotate computation~\eqref{eq:theta} produces the
crossing coefficients from which $g_N,h_N$ and $k_N$ are obtained.
Finally, symbolic expansion checks the factorisation
\eqref{eq:factorisation} and the algebraic identities used in the
two cases that follow it.

\smallskip\noindent
\textbf{Execution and scope.}
Either worksheet can be opened in Maple and executed from the beginning;
both begin with \texttt{restart}. For a complete moment-six run, the
option \texttt{RUN\_UNIFORM\_RANK\_CHECK} must remain set to
\texttt{true}, so that the full symbolic determinant is evaluated rather
than omitted. The author executed both worksheets successfully in
Maple 2026.1 on 29 October 2026. The files in
\texttt{results/executed/} preserve the corresponding outputs, including
the determinant check; the repository also contains a readable execution
record. Optional Python checks at selected values of $N$ are included
as supplementary comparisons, not substitutes for the uniform symbolic
calculations. The software establishes the stated computational identities;
the categorical deductions and the arguments excluding nonzero solutions
for integer $N\geq6$ remain those given in Section~\ref{ml6}.


\begin{thebibliography}{99}

\bibitem{ba2}
T.~Banica,
\emph{Quantum permutation groups},
arXiv:2012.10975 [math.QA] (2020).
\url{https://doi.org/10.48550/arXiv.2012.10975}.

\bibitem{ban}
T.~Banica,
Homogeneous quantum groups and their easiness level,
\emph{Kyoto J. Math.} \textbf{61} (2021), 171--205.
\url{https://doi.org/10.1215/21562261-2019-0077}.

\bibitem{ba1}
T.~Banica,
\emph{Introduction to Quantum Groups},
Springer, Cham, 2023.
\url{https://doi.org/10.1007/978-3-031-23817-8}.

\bibitem{bb3}
T.~Banica and J.~Bichon,
Quantum groups acting on 4 points,
\emph{J. Reine Angew. Math.} \textbf{626} (2009), 75--114.
\url{https://doi.org/10.1515/CRELLE.2009.003}.

\bibitem{bbc}
T.~Banica, J.~Bichon, and B.~Collins,
Quantum permutation groups: a survey,
in \emph{Noncommutative Harmonic Analysis with
Applications to Probability},
Banach Center Publ., vol.~78,
Polish Acad. Sci. Inst. Math., Warsaw, 2007, pp.~13--34.
\url{https://doi.org/10.4064/bc78-0-1}.

\bibitem{bbcc}
T.~Banica, J.~Bichon, B.~Collins, and S.~Curran,
A maximality result for orthogonal quantum groups,
\emph{Comm. Algebra} \textbf{41} (2013), 656--665.
\url{https://doi.org/10.1080/00927872.2011.633138}.

\bibitem{bc2}
T.~Banica and B.~Collins,
Integration over compact quantum groups,
\emph{Publ. Res. Inst. Math. Sci.} \textbf{43} (2007),
277--302.
\url{https://doi.org/10.2977/PRIMS/1201011782}.

\bibitem{bac}
T.~Banica and B.~Collins,
Integration over quantum permutation groups,
\emph{J. Funct. Anal.} \textbf{242} (2007), 641--657.
\url{https://doi.org/10.1016/j.jfa.2006.09.005}.

\bibitem{bcs}
T.~Banica, S.~Curran, and R.~Speicher,
Classification results for easy quantum groups,
\emph{Pacific J. Math.} \textbf{247} (2010), 1--26.
\url{https://doi.org/10.2140/pjm.2010.247.1}.

\bibitem{bas}
T.~Banica and R.~Speicher,
Liberation of orthogonal Lie groups,
\emph{Adv. Math.} \textbf{222} (2009), 1461--1501.
\url{https://doi.org/10.1016/j.aim.2009.06.009}.

\bibitem{bra}
R.~Brauer,
On algebras which are connected with the semisimple
continuous groups,
\emph{Ann. of Math.} (2) \textbf{38} (1937), 857--872.
\url{https://doi.org/10.2307/1968843}.

\bibitem{bcf}
M.~Brannan, A.~Chirvasitu, and A.~Freslon,
Topological generation and matrix models for quantum
reflection groups,
\emph{Adv. Math.} \textbf{363} (2020), Article~106982.
\url{https://doi.org/10.1016/j.aim.2020.106982}.

\bibitem{fr1}
A.~Freslon,
\emph{Compact Matrix Quantum Groups and Their Combinatorics},
London Mathematical Society Student Texts, vol.~106,
Cambridge University Press, Cambridge, 2023.
\url{https://doi.org/10.1017/9781009345705}.

\bibitem{fr2}
A.~Freslon,
On the classification of partition quantum groups,
\emph{Expo. Math.} \textbf{39} (2021), 238--270.
\url{https://doi.org/10.1016/j.exmath.2021.02.001}.

\bibitem{fre}
A.~Freslon,
\emph{The intermediate quantum permutation problem},
unpublished notes from a talk given on the occasion of
Roland Speicher's sixty-fifth birthday, 2025.
\url{https://www.imo.universite-paris-saclay.fr/~amaury.freslon/Documents/Talks/Speicher_65.pdf}.

\bibitem{fsw}
A.~Freslon, A.~Skalski, and S.~Wang,
Tracial central states on compact quantum groups,
\emph{J. Funct. Anal.} \textbf{289} (2025), Article~110988.
\url{https://doi.org/10.1016/j.jfa.2025.110988}.

\bibitem{grw}
D.~Gromada and M.~Weber,
Intertwiner spaces of quantum group subrepresentations,
\emph{Comm. Math. Phys.} \textbf{376} (2020), 81--115.
\url{https://doi.org/10.1007/s00220-019-03463-y}.

\bibitem{juw}
S.~Jung and M.~Weber,
Models of quantum permutations,
\emph{J. Funct. Anal.} \textbf{279} (2020), Article~108516.
\url{https://doi.org/10.1016/j.jfa.2020.108516}.

\bibitem{maa}
L.~Maa{\ss}en,
The intertwiner spaces of non-easy group-theoretical
quantum groups,
\emph{J. Noncommut. Geom.} \textbf{14} (2020), 987--1017.
\url{https://doi.org/10.4171/JNCG/384}.

\bibitem{mal}
S.~Malacarne,
Woronowicz Tannaka--Krein duality and free orthogonal
quantum groups,
\emph{Math. Scand.} \textbf{122} (2018), 151--160.
\url{https://doi.org/10.7146/math.scand.a-97320}.

\bibitem{mc1}
J.~P.~McCarthy,
A state-space approach to quantum permutations,
\emph{Expo. Math.} \textbf{40} (2022), 628--664.
\url{https://doi.org/10.1016/j.exmath.2021.12.003}.

\bibitem{mc2}
J.~P.~McCarthy,
Tracing the orbitals of the quantum permutation group,
\emph{Arch. Math.} \textbf{121} (2023), 211--224.
\url{https://doi.org/10.1007/s00013-023-01883-w}.

\bibitem{mcc}
J.~P.~McCarthy,
Correction to: Tracing the orbitals of the quantum
permutation group,
\emph{Arch. Math.} \textbf{123} (2024), 681--682.
\url{https://doi.org/10.1007/s00013-024-02056-z}.

\bibitem{mc3}
J.~P.~McCarthy,
Analysis for idempotent states on quantum permutation groups,
\emph{Math. Phys. Anal. Geom.} \textbf{28} (2025),
Article~14.
\url{https://doi.org/10.1007/s11040-025-09511-5}.

\bibitem{nis}
A.~Nica and R.~Speicher,
\emph{Lectures on the Combinatorics of Free Probability},
London Mathematical Society Lecture Note Series, vol.~335,
Cambridge University Press, Cambridge, 2006.
\url{https://doi.org/10.1017/CBO9780511735127}.

\bibitem{scw}
L.~Schmitz and M.~Wack,
Finite Gr{\"o}bner bases for quantum symmetric groups,
\emph{J. Symbolic Comput.} \textbf{136} (2026),
Article~102565.
\url{https://doi.org/10.1016/j.jsc.2026.102565}.

\bibitem{wa1}
S.~Wang,
Free products of compact quantum groups,
\emph{Comm. Math. Phys.} \textbf{167} (1995), 671--692.
\url{https://doi.org/10.1007/BF02101540}.

\bibitem{wa2}
S.~Wang,
Quantum symmetry groups of finite spaces,
\emph{Comm. Math. Phys.} \textbf{195} (1998), 195--211.
\url{https://doi.org/10.1007/s002200050385}.

\bibitem{web}
M.~Weber,
Quantum permutation matrices,
\emph{Complex Anal. Oper. Theory} \textbf{17} (2023),
Article~37.
\url{https://doi.org/10.1007/s11785-023-01335-x}.

\bibitem{wo1}
S.~L.~Woronowicz,
Compact matrix pseudogroups,
\emph{Comm. Math. Phys.} \textbf{111} (1987), 613--665.
\url{https://doi.org/10.1007/BF01219077}.

\end{thebibliography}
\end{document}